\documentclass[12pt]{amsart}
\usepackage{amssymb}
\usepackage{hyperref}
\usepackage{pstricks}
\usepackage{pst-node}
\usepackage{tikz}
\usetikzlibrary{decorations} \usetikzlibrary{snakes}
\usepackage[margin=1.1in]{geometry}

\title{Polypositroids}

\author{Thomas Lam \and Alexander Postnikov}

\address{Department of Mathematics, University of Michigan, Ann Arbor, MI 48109}
\address{Department of Mathematics, M.I.T., Cambridge, MA 02139}

\date{October 9, 2020}

\theoremstyle{plain}
\newtheorem{theorem}{Theorem}[section]
\newtheorem{proposition}[theorem]{Proposition}
\newtheorem{corollary}[theorem]{Corollary}
\newtheorem{conjecture}[theorem]{Conjecture}
\newtheorem{lemma}[theorem]{Lemma}

\newtheorem{question}[theorem]{Question}

\theoremstyle{definition}
\newtheorem{definition}[theorem]{Definition}
\newtheorem{example}[theorem]{Example}

\theoremstyle{remark}
\newtheorem{remark}[theorem]{Remark}

\def\R{\mathbb{R}}
\def\N{\mathbb{N}}
\def\Z{\mathbb{Z}}

\def\v{\mathbf{v}}

\def\u{\mathbf{u}}
\def\T{{\mathcal{T}}}
\def\E{{\mathcal{E}}}

\def\Vert{\mathrm{Vert}}
\def\emb{\mathit{f}}
\def\Strand{\mathrm{Strand}}
\def\Area{\mathrm{Area}} 
\def\Perim{\mathrm{Perim}} 
\def\val{\mathrm{val}}

\def\cloud{\mathrm{cloud}}
\def\Octa{\mathrm{Octa}}
\def\Memb{\mathrm{Memb}}
\def\lift{\mathrm{lift}}
\def\proj{\mathrm{proj}}
\def\d{\mathbf{d}}

\def\up{\mathrm{up}}
\def\des{\mathrm{des}}
\def\conv{\mathrm{conv}}
\def\env{\mathrm{env}}
\def\id{\mathrm{id}}
\def\sort{\mathrm{sort}}

\def\Gr{\mathrm{Gr}}
\def\GL{\mathrm{GL}}
\def\Csub{{\mathcal{C}}_{\rm sub}}
\def\Calc{{\mathcal{C}}_{\rm alc}}
\def\Cpoly{{\mathcal{C}}_{\rm pol}}
\def\CsubW{{\mathcal{C}}_{{\rm sub}}^W}
\def\CpolyW{{\mathcal{C}}_{{\rm pre}}^{W,c}}
\def\cyc{\mathcal{I}}

\def\B{{\mathcal B}}

\def\F{{\mathcal F}}
\def\Int{{\rm Int}}
\def\sp{{\rm sp}}

\def\v{{\mathbf{v}}}
\def\C{\mathbb{C}}
\def\<{\left<}
\def\>{\right>}

\def\sp{{\rm span}}
\def\id{{\rm id}}
\def\tR{{\tilde R}}
\def\A{{\mathcal{A}}}
\def\tbeta{{\tilde \beta}}
\def\tgamma{{\tilde \gamma}}
\def\tdelta{{\tilde \delta}}
\def\M{{\mathcal{M}}}
\def\Ncal{{\mathcal{N}}}
\def\sbinom{{\mbox{$\binom{[n]}{k}$}}}
\def\I{{\mathcal{I}}}
\def\SS{{\mathcal{S}}}
\def\col{{\rm col}}

\begin{document}

\begin{abstract}
We initiate the study of a class of polytopes, which we coin {\it polypositroids}, defined to be those polytopes that are simultaneously generalized permutohedra (or polymatroids) and alcoved polytopes.  Whereas positroids are the matroids arising from the totally nonnegative Grassmannian, polypositroids are ``positive" polymatroids.  We parametrize polypositroids using Coxeter necklaces and 
balanced graphs, and describe the cone of polypositroids by extremal rays and facet inequalities.  We introduce a notion of {\it $(W,c)$-polypositroid} for a finite Weyl group $W$ and a choice of Coxeter element $c$.  We connect the theory of $(W,c)$-polypositroids to cluster algebras of finite type and to generalized associahedra.
We discuss {\it membranes,} 
which are certain triangulated 2-dimensional surfaces 
inside polypositroids.  Membranes 
extend the notion of plabic graphs from positroids to polypositroids.
\end{abstract}

\maketitle
\setcounter{tocdepth}{1}
\tableofcontents

\section{Introduction}

\subsection{Polypositroids}
The aim of this work is to study a new class of polytopes, which we
call {\it polypositroids}.  These are polytopes which are
simultaneously generalized permutohedra (or polymatroids) 
and alcoved polytopes.

A permutohedron is the convex hull of the orbit of a point in $\R^n$
under the action of the symmetric group $S_n$.  A
{\it generalized permutohedron} is obtained from a permutohedron by
parallel translation of some of the facets.  The class of
generalized permutohedra include many classical polytopes: the usual
permutohedron, the associahedron, hypersimplices, matroid polytopes,
and many others.  Generalized permutohedra are polytopal analogues of matroids, and are essentially
equivalent to the {\it polymatroids} of Edmonds and the closely related {\it submodular functions} \cite{Edm,P,CL}.

An {\it alcoved polytope} is a polytope whose facets are normal to
roots in the type $A_{n-1}$ root system.  In \cite{LP}, we studied
the class of integer alcoved polytopes, which are (convex) unions of alcoves in the
affine Coxeter arrangement of type $A$.  In the present paper, we work with alcoved polytopes that
may not be integral.

Our initial motivation for studying polypositroids as a subclass of polymatroids comes from the theory of
total positivity.  Lusztig \cite{Lus} and Postnikov \cite{TP} have defined the totally nonnegative Grassmannian
$\Gr(k,n)_{\geq 0}$, a subspace of the real Grassmannian $\Gr(k,n)$ of $k$-planes in $\R^n$.
Any point $X \in \Gr(k,n)$ gives rise to a (realizable) matroid $M_X$.  When $X
\in \Gr(k,n)_{\geq 0}$ is totally nonnegative, the matroid $M_X$ is called a
{\it positroid}, short for ``positive matroid".  Positroids were classified in
\cite{TP} and \cite{Oh}, and the geometry of {\it
positroid varieties\/} in the Grassmannian was studied by Knutson-Lam-Speyer
\cite{KLS}.

To a matroid $M$ on the set $\{1,2,\ldots,n\}$, one has an associated matroid polytope $P_M \subset \R^n$.
Our investigations began with the observation that a matroid $M$ is a
positroid if and only if $P_M$ is an alcoved polytope (Theorem \ref{thm:positroid}).  Polypositroids are thus ``positive polymatroids".  In the present work, we study polypositroids from a discrete geometer's perspective, leaving aside potential connections to Grassmannians, and so on.
  
\medskip
Whereas matroids are notoriously difficult to parametrize, the subclass of positroids were parametrized in \cite{TP,Oh}.  In Theorem~\ref{thm:main}, we give a parametrization of polypositroids, showing that they are in bijection with {\it Coxeter necklaces} and with {\it balanced graphs}.  Coxeter necklaces and balanced graphs are generalizations of the {\it Grassmann necklaces} and {\it decorated permutations} of \cite{TP}.

The set of generalized permutohedra attains a cone structure under Minkowski sum, and the corresponding cone is the {\it cone of submodular functions} $\Csub$ \cite{Edm}.  Alcoved polytopes and polypositroids also form cones $\Calc$ and $\Cpoly$.  There is a (projection) map of cones $\Csub \to \Calc$ and we show in Theorem~\ref{thm:subpoly} that the image of this map is $\Cpoly$.  In Corollary~\ref{cor:rays} and Theorem~\ref{thm:conenoncross}, we describe the extremal rays and give defining inequalities for the cone $\Cpoly$.  For example, the extremal rays of $\Cpoly$ are indexed by directed cycles on $\{1,2,\ldots,n\}$. 


\medskip
The normal fan of the permutohedron is the {\it braid fan} associated to the arrangement of hyperplanes $\{x_i-x_j =0\}$.  Generalized permutohedra are exactly the polytopes with normal fan a coarsening of the braid fan.  In Section~\ref{sec:normal}, we study the possible normal fans $\Ncal(P)$ for a generic simple polypositroid $P$.  In contrast to generalized permutohedra, there is more than one such normal fan, and each such $\Ncal(P)$ is a coarsening of the braid fan.

We show in Lemma~\ref{lem:Talternating} that each maximal cone $C_T$ of $\Ncal(P)$ is labeled by an {\it alternating noncrossing} tree.  The maximal cones $C_w$, $w \in S_n$ satisfying $C_w \subset C_T$ are described in Proposition~\ref{prop:braidcoarse} in terms of a dual {\it noncrossing circular-alternating tree}.  In Theorem~\ref{thm:ensemble}, we show that the normal fan $\Ncal(P)$ is characterized by a {\it matching ensemble}, a collection of perfect matchings, one for each bipartite subgraph of the complete graph, satisfying certain axioms.  Our matching ensembles are a variant of the matching ensembles of Oh and Yoo \cite{OY}, and the matching fields of Sturmfels and Zelevinsky \cite{SZ}.  In Theorem~\ref{thm:fvector} we show that all generic simple polypositroids have the same $f$-vector, identical to that of the {\it cyclohedron}.

\subsection{Coxeter polypositroids}
In the second part of this work, we generalize the theory of polypositroids to the root-system theoretic setting.  Let $V$ be a real vector space, and $R \subset V$ be a crystallographic root system with Weyl group $W$.  A {\it generalized $W$-permutohedron} is a polytope $P \subset V$ whose edges are in the directions of $R$. 

Now fix the choice of a Coxeter element $c$.  Define the {\it twisted root system} $\tR:= (I-c)^{-1} R$.  
We define a \emph{$(W,c)$-twisted alcoved polytope} to be a polytope $P \subset V$ whose facet normals belong to the twisted roots $\tR$.  These polytopes are $c$-twisted variants of the alcoved polytopes we studied in \cite{LP2}.  A \emph{$(W,c)$-polypositroid} is a polytope that is simultaneously a generalized $W$-permutohedron and a $(W,c)$-twisted alcoved polytope.  When $R$ is of type 
$A_{n-1}$ and $c$ is the long cycle $(12 \cdots n)$, this definition reduces to the earlier one.  

In general, the class of $(W,c)$-polypositroids cannot be parametrized in the same manner that we did for polypositroids.  We introduce a larger class of {\it $(W,c)$-prepolypositroids} as a compromise.  The set of $(W,c)$-prepolypositroids is defined as a cone, by giving the facet inequalities \eqref{eq:conefacets} satisfied by the support function.
We define $(W,R^+,c)$-Coxeter necklaces and $(W,c)$-balanced arrays and show in Theorem~\ref{thm:threecones} that these objects are in bijection with $(W,c)$-prepolypositroids.  We also show that there is a projection map from the cone $\CsubW$ of $W$-submodular functions (see \cite{ACEP}) to the cone $\CpolyW$ of $(W,c)$-prepolypositroids and we conjecture (Conjecture~\ref{conj:genpermalc}) that this map is surjective.

In Section~\ref{sec:cluster}, we connect the theory of $(W,c)$-prepolypositroids to cluster algebras.  Let $\A(W,R^+,c)$ be the cluster algebra of finite type associated to the Coxeter element $c$, as studied by Yang and Zelevinsky \cite{YZ}.  We show using the work of Padrol, Palu, Pilaud, and Plamondon \cite{PPPP}, in Theorem~\ref{thm:moreinequalities} that each exchange relation of $\A(W,R^+,c)$ gives rise to an inequality satisfied by the cone $\CpolyW$ of $(W,c)$-prepolypositroids.  We show that $(W,c)$-prepolypositroids are closely related to generalized associahedra, one of our main examples of $(W,c)$-polypositroids.

In Section~\ref{sec:Wnormal}, with the aim of studying the normal fans of $(W,c)$-prepolypositroids, we develop notions of \emph{alternating} and \emph{$c$-noncrossing} for root systems, and a notion of \emph{$c$-noncrossing tree}.  These notions are related to the theory of finite type cluster algebras, and to the theory of reflection factorizations of Coxeter elements.

\subsection{Membranes}
The third part of the paper is devoted to {\it membranes.} 
In Section~\ref{sec:root_membranes},
we define $R$-membranes, for a root system $R\subset V$, 
which are essentially triangulated 2-dimensional surfaces in $V$
homeomorphic to wedges of disks such that every edge of every
triangle in them is a parallel translation of a root from $R$.
An $R$-membrane is {\it minimal\/} if it has minimal possible surface area
among all membranes with the same {\it boundary loop.}
We view the problem of describing minimal membranes as a discrete
version of {\it Plateau's problem\/} 
from geometric measure theory
concerning the existence of a minimal surface with a given boundary.

In the rest of the paper we discuss membranes of type $A$.
In Section~\ref{sec:Mem_A} we show that membranes of type $A$
are closely related to Postnikov's {\it plabic graphs}, introduced in the study 
of the totally nonnegative Grassmannian $\Gr(k,n)_{\geq 0}$ \cite{TP}.  
For each positroid, there is a class of {\it reduced\/} 
plabic graphs connected with each other by 
local moves. Each reduced plabic graph gives a parametrization
of the associated positroid cell in $\Gr(k,n)_{\geq 0}$.

Membranes (of type $A$) are in bijection with plabic graphs 
with faces labelled by integer vectors.
In Section~\ref{sec:moves_membranes} we show that 
local moves of plabic graphs correspond to octahedron
and tetrahedron moves of membranes.
In Section~\ref{sec:minimal_membranes_reduced} we show
that minimal membranes correspond to reduced plabic graphs.

In Sections~\ref{sec:positr_membr}, \ref{sec:polypositroid_membr},
and~\ref{sec:postr_lifts} we discuss special classes of membranes
associated with positroids and polypositroids.
In Section~\ref{sec:semisimple_membranes} we define {\it semisimple\/}
membranes as membranes that project bijectively onto the Coxeter plane.
We show that notions of semisimple membranes and
minimal membranes are equivalent to each other 
(for a particular class of boundary loops).

The structures that we study in this paper are related to the theory of 
{\it cluster algebras\/} in several different ways.  
While in Section~\ref{sec:cluster} we connect $(W,c)$-prepolypositroids 
to cluster algebras of finite type, in
Section~\ref{sec:higher_octahedron} we connect membranes (of type $A$)
to a certain class of cluster algebras, which in general are not 
of finite type.
We discuss the {\it higher octahedron
recurrence\/} as a certain rational recurrence relation on variables
$x_\lambda$ labelled by integer vectors $\lambda$.  This relation naturally
extends the octahedron recurrence on $\Z^3$ \cite{Spe} to a higher dimensional integer
lattice.
With each polypositroid $P$, we associate a cluster algebra $\A_P$ generated 
by some finite subset of variables $x_\lambda$.
Minimal membranes for the polypositroid $P$ correspond to a class of 
clusters of this cluster algebra, and local moves of membranes
correspond to cluster {\it mutations}.  Remarkably, the class of cluster algebras $\A_P$ is the same as the subclass associated to positroids, which Galashin and Lam have shown \cite{GL} to be isomorphic to the coordinate rings
of open positroid varieties \cite{KLS}. 

Finally, in Section~\ref{sec:asympt_membr} we 
propose an area of study, 
which we dub the ``Asymptotic Cluster Algebra''.
We pose a problem related to asymptotics
of membranes under dilations.  The setup of membranes extend several 
models from statistical physics.

\medskip

\subsection*{Acknowledgements}
T.L.\ was supported by NSF DMS-0600677, DMS-0652641, DMS-0901111, DMS-0968696, DMS-1160726, DMS-1464693, and DMS-1953852.
A.P.\ was supported by NSF DMS-0546209, DMS-1100147, DMS-1362336,
DMS-1500219, and DMS-1764370.

\part{Polypositroids}\label{part:1}

Throughout the paper, we use the following notation.
Let $[n]:=\{1,2,\ldots,n\}$. 
Let $\binom{[n]}{k}$ denote the set of $k$-element subsets of $[n]$.
Also let $e_1,\dots,e_n$ denote the standard basis of~$\R^n$.

\section{Positroid polytopes}
\label{sec:positroids}
The material of this section serves as motivation, and is largely independent of the rest of the work.

Let $\Gr(k,n)$ denote the Grassmannian of $k$-planes in $\R^n$.  We may represent a point $X \in \Gr(k,n)$ as a $k \times n$ matrix.  For a $k$-element subset $I\in \binom{[n]}{k}$, the {\it Pl\"ucker coordinate} $\Delta_I(X)$ is defined to be the $k \times k$ minor indexed by the columns $I$.  The {\it matroid} 
$\M_X$ is given by $\M_X:= \{I \in \text{$\binom{[n]}{k}$} \mid \Delta_I(X) \neq 0\}$.

The {\it totally nonnegative Grassmannian} $\Gr(k,n)_{\geq 0}$ is the subspace of $\Gr(k,n)$ represented by matrices all of whose Pl\"ucker coordinates are nonnegative.  The matroid of a totally nonnegative point $X \in \Gr(k,n)_{\geq 0}$ is called a {\it positroid}.

The matroid polytope $P_\M$ of a matroid $\M$ is the convex hull of the vectors $e_I$, $I \in \M$, where $e_I:=e_{i_1} + \cdots + e_{i_k} \in \R^n$ if $I = \{i_1,\ldots,i_k\}$.  By \cite{GGMS}, matroid polytopes are exactly those polytopes whose vertices belong to $\{e_I \mid I \in \binom{[n]}{k}\}$ and whose edges are parallel to vectors of the form $e_i - e_j$.  We have the following characterization of the matroid polytopes of positroids.

\begin{theorem}\label{thm:positroid}
A matroid $\M$ is a positroid if and only if the matroid polytope $P_\M$ is an alcoved polytope, that is, it is given by inequalities of the form $c_{ij} \leq x_i + x_{i+1} + \cdots + x_j \leq b_{ij}$, for $1 \leq i < j \leq n$.
\end{theorem}
We discovered Theorem~\ref{thm:positroid} over a decade ago, and it has since found a number of applications, for example to {\it positively oriented matroids} \cite{ARW17}.  
See \cite{ALS, Ear, LPW} for other appearances of positroid polytopes.

To prove Theorem~\ref{thm:positroid}, we use a classification of positroids due to Oh \cite{Oh}; see also \cite[Section 8.2]{Lam}.
 The {\it Bruhat partial order} on $\binom{[n]}{k}$ is defined as follows.  For two subsets $I, J \in \binom{[n]}{k}$, we write $I \leq J$ if $I = \{i_1 < i_2 <\cdots < i_k\}$, $J = \{j_1 < j_2 < \cdots < j_k\}$ and we have $i_r \leq j_r$ for $r = 1,2,\ldots,k$.  
For $I \in \binom{[n]}{k}$, the {\it Schubert matroid} $\SS_I$ is defined as 
$$
\SS_I:= \{J \in \sbinom \mid I \leq J\}
$$
and has minimal element $I$ in the Bruhat order. 
For $a \in [n]$, let $<_a$ denote the 
cyclically rotated order on $[n]$ with minimum $a$,
i.e., $a<_a (a+1) <_a\cdots <_a n <_a 1 <_a\cdots <_a (a-1)$,
 which induces a partial order $\leq_a$ on $\binom{[n]}{k}$.  Let $\SS_{I,a} := \{J \in \binom{[n]}{k} \mid I \leq_a J\}$ denote the cyclically rotated Schubert matroid.
Equivalently, $\SS_{I,a} := c^{a-1}(\SS_{c^{-a+1}(I)})$,
where $c$ is the long cycle $(1,2,\dots,n)$ in the symmetric group $S_n$ 
naturally acting on $[n]$ and ${[n]\choose k}$.

A {\it $(k,n)$-Grassmann necklace} $\I = (I_1,I_2,\ldots,I_n)$ is a $n$-tuple of $k$-element subsets of $[n]$ satisfying the following condition: for each $a \in [n]$, we have
\begin{enumerate}
\item $I_{a+1} = I_a$ if $a \notin I_a$, 
\item $I_{a+1} = (I_a \setminus \{a\}) \cup \{a'\}$ if $a \in I_a$,
\end{enumerate}
with indices taken modulo $n$.  

\begin{theorem}[\cite{Oh, TP}] \label{thm:Oh}
Let $\I= (I_1,I_2,\ldots,I_n)$ be a $(k,n)$-Grassmann necklace.  Then the intersection of cyclically rotated Schubert matroids
\begin{equation}\label{eq:MI}
\M_\I = \SS_{I_1,1} \cap \SS_{I_2,2} \cap \cdots \cap \SS_{I_n,n}
\end{equation}
is a positroid, and the map $\I \mapsto \M_\I$ gives a bijection between $(k,n)$-Grassmann necklaces and  positroids of rank $k$ on $[n]$.
\end{theorem}

\begin{proof}[Proof of Theorem~\ref{thm:positroid}]
Let $\Delta(k,n)$ denote the hypersimplex, the convex hull of all points $e_I$, for $I \in \binom{[n]}{k}$.
The matroid polytope $P_{\SS_{I,a}}$ is the intersection of the hypersimplex $\Delta(k,n)$ with the inequalities
$$
x_a + x_{a+1} + \cdots + x_b \geq \#(I \cap [a,b])
$$
for $i = 1,2,\ldots,n$.  In particular, $P_{\SS_{I,a}}$ is an alcoved polytope.

Let $\M$ be an arbitrary matroid.  
Recall, that any matroid has a unique minimal base 
in the Bruhat partial order $\leq$ on ${[n]\choose k}$, 
and thus in any cyclically rotated partial order $\leq_a$.
Denote by $I_a(\M)$ the minimal base of $\M$ with respect to $\leq_a$.
Let $Q = \env(P_\M)$ be the alcoved envelope of $P_\M$, i.e., the smallest alcoved polytope that contains $P_\M$.  Then $Q$ is given by the intersection of the rotated Schubert matroid polytopes $P_{\SS_{I_a(\M),a}}$ for $a =1,2,\ldots,n$ (see Lemma~\ref{lem:envelope}).  It is known \cite[Lemma 16.3]{TP} that for any matroid $\M$, the $n$-tuple $\I(\M)=(I_1(\M),I_2(\M),\ldots,I_n(\M))$ is a $(k,n)$-Grassmann necklace, and $\M_\I$ is called the {\it positroid envelope} of $\M$ \cite{KLS}.  Thus, $Q$ is the matroid polytope of the positroid envelope of $\M$.  In particular, $P_\M$ is alcoved if and only if $Q = P_\M$ if and only if $\M$ is a positroid.
 \end{proof}

A \emph{decorated permutation} on $[n]$ is a pair $\pi^: = (\pi,\col)$ where $\pi$ is a permutation of $[n]$ and $\col$ is an assignment of one of two colors ``black" and ``white" to each of the fixed points $\{i \in [n] \mid \pi(i) = i\}$.  We say that $i \in [n]$ is an {\it anti-exceedance} of $\pi^:$ if $\pi^{-1}(i) > i$ or $\pi(i) = i$ and $i$ is colored white.  Given a $(k,n)$-Grassmann necklace $\I = (I_1,I_2,\ldots,I_n)$, we define a decorated permutation $\pi^:(\I)$ by 
\begin{enumerate}
\item if $I_{a+1} = I_a - \{a\} \cup \{a'\}, a' \neq a$, then $\pi(a) = a'$;
\item $I_{a+1} = I_a$ and $a \notin I_a$, then $\pi(i) =i$ and $i$ is colored black;
\item $I_{a+1} = I_a$ and $a \in I_a$, then $\pi(i) = i$ and $i$ is colored white.
\end{enumerate}
\cite[Lemma 16.2]{TP} states that the map $\I \to \pi^:(\I)$ is a bijection between $(k,n)$-Grassmann necklaces and decorated permutations on $[n]$ with $k$ anti-exceedances.

\section{Polypositroids}
In this paper we consider several classes of convex polytopes in $\R^n$.  All
polytopes lie in an affine hyperplane $H=H_k:=\{x \in \R^n \mid x_1 + \cdots + x_n = k\}$,
for some constant $k$.  For the majority of this work, the reader may assume that the hyperplane $H$ has been fixed.

\subsection{Generalized permutohedra}


\begin{definition}[\cite{P}]
A polytope $P\subset \R^n$ is called a {\it generalized permutohedron\/} if all edges of $P$
are directed as multiples of the vectors $e_i-e_j$.
\end{definition}

The class of generalized permutohedra include many classical polytopes: the usual permutohedron, the associahedron, hypersimplices,
any many others; see~\cite{P}.  
Let us give several alternative ways to describe the class of generalized permutohedra.
For a permutation $w\in S_n$, let $v_w = -(w^{-1}(1),\dots,w^{-1}(n))\in\R^n$, and let
$P_n := \conv(v_w\mid w\in S_n)$ be the standard permutohedron in $\R^n$.  For the following result see \cite[Theorem 15.3]{PRW} and \cite{CL}.

\begin{theorem}\label{thm:PRW}
The following are equivalent for a polytope $P$ in $\R^n$:
\begin{enumerate}
\item
The polytope $P$ is a generalized permutohedron.
\item
The normal fan of $P$ is a coarsening of the normal fan of the standard permutohedron $P_n$.
\item
The vertices of $P$ can be (possibly redundantly) labelled $v_w'$, $w\in S_n$, such that for any edge $(v_u,v_w)$ in $P_n$, there is a nonnegative real $t$ such that
$(v_u' - v_w') = t(v_u-v_w)$.
\end{enumerate}
\end{theorem}

For a polytope $P$, we define the {\it support function} $f_P: (\R^n)^* \to \R$ given by
$$
f_P(h) = \max_{v \in P} h(v).
$$
The function $f_P$ is a piecewise linear function on $(\R^n)^*$ whose maximal domains of linearity are exactly the top-dimensional cones of the normal fan of $P$.  

The normal fan of the standard permutohedron $P_n$ is the braid fan (see Section~\ref{ssec:braid}), and by Theorem~\ref{thm:PRW}(2), a generalized permutohedron $P$ is uniquely determined by the values $f_P(S):= f_P(h_S)$ where $h_S(x_1,\ldots,x_n)= \sum_{i \in S} x_i$, and $S$ varies over proper non-empty subsets of $[n]$.  The polytope $P$ is then given by 
\begin{equation}\label{eq:permfacet}
x_1+\cdots + x_n = k = f_P([n]) \quad\text{and}\quad
\sum_{i\in S} x_i \leq f_P(S).
\end{equation}
We write $f_P|_{2^{[n]}}$ for the function sending a subset $S \subset [n]$ to $f_P(S)$, and more generally we will use notation such as $f_P|_\SS$ for a collection $\SS \subseteq 2^{[n]}$ of subsets.

\subsection{Alcoved polytopes}

Let $h_1,\dots,h_n$ be the basis of $(\R^n)^*$ such that $h_i(x) = x_1 + x_2 + \cdots + x_i$.  Thus $h_i:= h_{[1,i]}$.

\begin{definition} \cite{LP} \label{def:alcoved} \  A polytope $P\subset \R^n$ is called an {\it alcoved polytope\/} if it is given by inequalities of
the form 
\begin{equation}\label{eq:alcoveineq}
(h_i - h_j)(x) \leq a_{ij}
\end{equation} and the equation $x_1 + \cdots + x_n = k$, for some real\footnote{In \cite{LP}
   we required that $a_{ij}$ and $k$ were integer numbers while in this paper we allow any real coefficients.  We will call
   the polytopes from \cite{LP} {\it integer\/} alcoved polytopes.  Their vertices are always integer lattice points.}
numbers $a_{ij}$ and $k$.
\end{definition}
Alcoved polytopes are also called {\it polytropes} in \cite{JK}.   We will always assume in \eqref{eq:alcoveineq} that the $a_{ij}$ have been chosen to be minimal, that is, they are values of the support function.

\begin{lemma}\label{lem:add}
The set of alcoved polytopes in 
the affine hyperplane $H_k$ and the set of alcoved polytopes in $H_\ell$ are in natural bijection via the isomorphism $+(\ell{-}k) e_1: H_k \to H_\ell$ adding $(\ell-k)$ to the first coordinate.  This bijection preserves the values of $a_{ij}$.
\end{lemma}

An alcoved polytope $P$ is called {\it generic} if it is full-dimensional in $H$, and each equality $(h_i - h_j)(x) = a_{ij}$ defines a facet of $P$, for all 
$i,j\in [n]$, $i \neq j$.

\begin{definition}\label{def:env}
Let $Q \subset H$ be a bounded subset.  The {\it alcoved envelope} 
$\env(Q)\subset H$ is defined to be the smallest alcoved polytope containing $Q$.
Equivalently, $\env(Q)$ is given by inequalities~(\ref{eq:alcoveineq}),
where $a_{ij}=\mathrm{supremum}_{x\in Q} (h_i-h_j)(x)$.

\end{definition}

\begin{example}
For $n=3$, the class of generalized permutohedra coincides with the class of alcoved polytopes.
However, for $n\geq 4$, neither of these two
classes of polytopes contains the other.  For example, the standard permutohedron $P_4 = \conv((w(1),w(2),w(3),w(4))\mid w\in S_4)\subset \R^4$
is a generalized permutohedron but it is not an alcoved polytope, because one of its facets is given by $x_1 + x_3 \leq 7$.
On the other hand, the simplex $\conv(e_1 + e_3,\, e_2 + e_3,\, e_2+e_4,\, e_3+e_4)\subset \R^4$ is an alcoved polytope but it is not a generalized permutohedron,
because it has the edge $(e_1+e_3,\,e_2 + e_4)$.   Here, ``conv'' means the convex hull of points.
\end{example}

Define the {\it cyclic interval\/} $[r,s]$ in $[n]$ as
$$
[r,s]:=\left\{\begin{array}{ll}
\{r,r+1,\dots,s\}&\text{if } 1\leq r\leq s \leq n,\\
\{r,r+1,\dots,n,1,2,\dots,s\}&\text{if } 1\leq s < r \leq n.
\end{array}
\right\}
$$
and set $h_{[r,s]}:= \sum_{i \in [r,s]} x_i \in (\R^n)^*$.  Alcoved polytopes $P$ are exactly all polytopes of the form
\begin{equation}\label{eq:frs}
x_1+\cdots + x_n = k \quad\text{and}\quad
\sum_{i\in [r,s]} x_i \leq f_{[r,s]}
\end{equation}
for {\it cyclic intervals} $[r,s]\subset [n]$, where $f_{[r,s]}= f_P([r,s]):= f_P(h_{[r,s]})$.  Note that we have 
\begin{equation}\label{eq:af}
a_{ij}  = \begin{cases} f_{[j+1,i]} &\mbox{if $i > j$} \\
 f_{[j+1,i]} -f_{[n]} =   f_{[j+1,i]} -k& \mbox{if $i < j$} 
\end{cases}
\end{equation}
and it will be convenient to use both sets of parameters $f_{[r,s]}$ and $a_{ij}$.

If $x =(x_1,x_2,\ldots,x_n) \in \R^n$, we shall use the shorthands
$x_{[r,s]} := h_{[r,s]}(x)= \sum_{i \in [r,s]} x_i$ and $x_S:= h_S(x) = \sum_{i \in S} x_i$.

\begin{remark} 
The class of generalized permutohedra in $\R^n$ is closed under 
the operation of taking the Minkowski sum polytopes, 
but is not closed under the operation of taking the intersection of 
polytopes.
On the other hand, the class of alcoved polytopes in $\R^n$ is closed under
the operation of taking the intersection of polytopes (if it is nonempty),
but is not closed under the operation of taking the Minkowski sum of polytopes.
\end{remark}

\subsection{Polypositroids}

\begin{definition}\label{def:poly}
A {\it polypositroid\/} is a polytope $P\subset H_k$ which is both a generalized permutohedron and an alcoved polytope.
\end{definition}
As for alcoved polytopes, polypositroids in $H_k$ and in $H_\ell$ are naturally in bijection.  One of the main results of this paper is an explicit parametrization of all polypositroids.

We give some examples of polypositroids.   
For $a \in \R$ and $P$ a polytope, the notation $aP$ denotes the polytope $\{ax \mid x
\in P\}$.  For a cyclic interval $[r,s]$, let
$$
\Delta_{[r,s]} := \conv(e_r, e_{r+1},\ldots,e_s)
$$
denote the corresponding coordinate simplex.

The {\it cyclohedron} is the Minkowski sum of simplices
\begin{equation}\label{E:cyclo}
P = \sum_{[r,s] \neq [n]} c_{r,s}\Delta_{[r,s]} + b \Delta{[n]}
\end{equation}
where each $c_{r,s} > 0$ and $b > 0$.   By \cite[Thereom 7.4]{P}, the cyclohedron $P$ is both an alcoved polytope and a generalized permutohedron, and in addition $P$ is simple.  The $f_{[r,s]}$ from \eqref{eq:frs} are given by
\begin{equation}\label{eq:cyclfrs}
f_{[r,s]}=  b + \sum_{[t,u] \cap [r,s] \neq \emptyset} c_{t,u}
\end{equation}
and we also set $k:= b+ \sum c_{r,s}$ so that $P \subset H_k$.  

If we set $c_{r,s} = 0$ whenever $r>s$
(i.e., whenever $[r,s]$ is not an honest interval), 
then we obtain the {\it associahedron}, which is also a polypositroid.  Another possibility is to consider the polytope $P(b,c_{r,s})$ defined by \eqref{E:cyclo}, where we allow $b \in \R$ and $c_{r,s} \geq 0$.  When $b < 0$, we use the {\it Minkowski difference}:
$$
P - Q := \{ x \in \R^n \mid x + Q \subset P\}.
$$
It is not hard to see that the deformed cyclohedron $P(b,c_{r,s})$ is either empty, or a polypositroid.  When $b \geq 0$, the polytope $P(b,c_{r,s})$ is a deformation of a cyclohedron: the normal fan of $P(b,c_{r,s})$ is a coarsening of that of the cyclohedron.  However, when $b < 0$, this may no longer be the case.

\section{The cone of polypositroids}
The numbers $f_P(S)$ in \eqref{eq:permfacet} are not arbitrary.  
It is well known that support functions of generalized permutohedra 
are exactly the {\it submodular functions},
e.g., see \cite[Proposition~12]{MPSSW}, \cite[Section~12]{AA}, \cite[Theorem 3.11]{CL}. 

\begin{theorem}
The polytope $P$ given by \eqref{eq:permfacet} is a generalized permutohedron if and only if $f_P$ is submodular, that is
\begin{equation}\label{eq:submod}
f_P(S) + f_P(T) \geq f_P(S \cap T) + f_P(S \cup T)
\end{equation}
for any subsets $S,T \subset [n]$, where we assume that $f_P(\emptyset) = 0$.
\end{theorem}

\begin{definition}
The {\it submodular cone}, or the cone of generalized permutohedra, is the cone $\Csub$ of all functions $f :2^{[n]} \to \R$ satisfying \eqref{eq:submod}.
\end{definition}

The cone structure of $\Csub$ corresponds to taking Minkowski sums of generalized permutohedra.  Let us now turn to alcoved polytopes.

\begin{theorem} \label{thm:triangle}
Suppose $n \geq 3$.  Let $P$ be an alcoved polytope given by \eqref{eq:alcoveineq} (with $a_{ij}$ minimal).  Then the $a_{ij}$ satisfy the triangle inequality
\begin{equation}\label{eq:triangle}
a_{ij} + a_{jk} \geq a_{ik}
\end{equation}
for distinct $i,j,k \in [n]$.  Conversely, any $a_{ij}$ satisfying \eqref{eq:triangle} define a (nonempty) alcoved polytope in $H_k$.
\end{theorem}

\begin{proposition}\label{prop:triangle}
Let $n \geq 3$.  Let $P$ be an alcoved polytope given by the
inequalities $(h_i - h_j)(x) \leq a_{ij}$.  Then $P$ is {\it
generic} if and only if for any three pairwise distinct indices $i,
j, k \in [n]$ we have the strict triangle inequality $a_{ij} + a_{jk} >
a_{ik}$.
\end{proposition}
\begin{proof}
The ``only if'' direction is trivial.  We prove the ``if''
direction.  Let $P$ denote the ``alcoved polyhedron''
inside $\R^n$ given by the inequalities $(h_i - h_j)(x) \leq a_{ij}$.  The intersections $P \cap H_k$ for varying $k$ are linearly isomorphic via Lemma~\ref{lem:add}.  Thus it suffices to show that $P$ itself is generic.

Since $n \geq 3$, summing $a_{cd}+a_{de} > a_{ce}$ and
$a_{dc}+a_{ce}> a_{de}$ for distinct $c,d,e$ we obtain $a_{cd} + a_{dc}
> 0$.  It follows that $P$ is not strictly contained in any hyperplane $(h_c-h_d)(x) = a_{cd}$, and therefore $P$ is full-dimensional.  

Fix $ c\neq d$ and let $A = \{x \in
\R^n \mid (h_c-h_d)(x) = a_{cd}\}$.  To show that $P \cap A$ is a facet of $P$ it suffices to show that $P \cap A \neq \emptyset$ because by the same reasoning as above $P \cap A$ would be full-dimensional in $A$.  

For $x \in \R^n$,
define
$$
d_P(x) = \max_{i,j}(\max(0,(h_i-h_j)(x)-a_{ij})).
$$
Thus $x \in P$ if and only if $d_P(x) = 0$.  Since the function
$d_P$ is continuous, it is straightforward to see that it achieves a
minimum on $A$.  Let $v \in A$ be such a minimum, and assume $d_P(v)
> 0$. We also assume that $v$ is chosen so that
$$
R(v) = \#\{(i,j) \mid (h_i-h_j)(v) = a_{ij} + d_P(v)\}
$$
is minimal.  Note that since $v \in A$, we have $(c,d) \notin R(v)$.
 It follows from $a_{cd} + a_{dc}
> 0$ that $(d,c) \notin R(v)$.

Suppose $(e,f) \in R(v)$.  Assume that $e \notin \{c,d\}$; the case
$f \notin \{c,d\}$ is similar.  If there does not exist $(g,e) \in
R(v)$, then we can modify $v$ slightly so that $h_e(v)$ decreases
but the other $h_{e'}(v)$ remain unchanged, and reducing the size of
$R(v)$. This would contradict the construction of $v$.  But if
$(g,e) \in R(v)$ then using the condition of the Proposition, we
have $(h_g-h_f)(v) = (h_g-h_e)(v) + (h_e-h_f)(v) = a_{ge}+a_{ef} + 2d_P(v) > a_{gf}+d_P(v)$,
contradicting the definition of $d_P(v)$.  Thus we conclude that $d_P(v)
= 0$, as desired.
\end{proof}

\begin{proof}[Proof of Theorem~\ref{thm:triangle}]
By our assumption that the $a_{ij}$-s are taken minimal, it is clear that \eqref{eq:triangle} holds for any alcoved polytope. 
By Proposition~\ref{prop:triangle}, the inequalities $a_{ij} + a_{jk} >
a_{ik}$ define an open cone $C \subset \R^{n(n-1)}$, each point of which represents a generic alcoved polytope.  It follows  from \eqref{eq:cyclfrs} that $C$ is nonempty: the cyclohedron is a generic alcoved polytope.  The closure of $C$ is thus the closed cone cut out by \eqref{eq:triangle}.  The corresponding limits of generic alcoved polytopes are nonempty alcoved polytopes, finishing the proof of the theorem.
\end{proof}

\begin{definition}
The {\it triangle inequality cone}, or the cone of alcoved polytopes, is the cone $\Calc \subset \R^{n(n-1)}$ of all $a_{ij}$ satisfying \eqref{eq:triangle}.
\end{definition}

\begin{remark}
We caution the reader that the cone $\Csub$ contains the information $f_P([n]) = k$ and thus parametrizes generalized permutohedra in various affine hyperplanes $H_k$.  In contrast, an alcoved polytope is determined by a point in $\Calc$ together with the value of $k$.  Equivalently, $\Calc$ is the cone of alcoved polytopes inside $H_0$.
\end{remark}

The cone structure of $\Calc$ corresponds to the composition of the following two operations: first take the Minkowski sum $P_1 + P_2$ of two alcoved polytopes, and then take the alcoved envelope $\env(P_1+P_2)$ (see Definition~\ref{def:env}).

\begin{definition}
The {\it polypositroid cone} $\Cpoly \subset \Calc$ is the subset of $\Calc$ representing alcoved polytopes that are polypositroids.
\end{definition}

The fact that $\Cpoly$ is closed under the addition and positive scalar multiplication is a consequence of Theorem~\ref{thm:subpoly} or Theorem~\ref{thm:conenoncross} below.  Let $\cyc \subset 2^{[n]}$ denote the collection of all nonempty cyclic intervals, including $[n]$ itself.
We define a map $\pi_{\cyc}: \R^{2^{[n]}} \to \R^{\cyc} \simeq \R^{n(n-1)}$ by first projecting to cyclic intervals, and then applying the transformation \eqref{eq:af}.  

\begin{theorem}\label{thm:subpoly}
We have $\pi_{\cyc}(\Csub) = \Cpoly$.
\end{theorem}

Theorem~\ref{thm:subpoly} will be proved in Section~\ref{ssec:subpolyproof}.
Note that $\pi_{\cyc}: \Csub \to \Cpoly$ is a homomorphism of cones: it commutes with addition and with scalar multiplication, and sends $0$ to $0$.

\begin{theorem}\label{thm:conenoncross}
The cone $\Cpoly$ is the subcone of $\Calc$ satisfying
\begin{equation}\label{eq:anoncross}
a_{ik} + a_{jl} \geq a_{il} + a_{jk}
\end{equation}
for any four indices $i,j,k,l$ in cyclic order.
%
\end{theorem}

Theorem~\ref{thm:conenoncross} will be proved in Section~\ref{ssec:conenoncrossproof}.

\begin{example}
Let $P$ be the cyclohedron of \eqref{E:cyclo}.  Then \eqref{eq:anoncross} is immediate from \eqref{eq:cyclfrs}.
\end{example}

\begin{remark}
We have described the three cones $\Csub, \Calc, \Cpoly$ in terms of defining (possibly redundant) inequalities.  
The extremal rays of $\Csub$ (modulo translation) 
are hard to describe explicitly; 
see, e.g., \cite{Ng, MPSSW, ACEP} for discussions of the rays.
Among the rays of $\Csub$ are all connected matroids \cite{Ng}. 

However, we will give in Corollary~\ref{cor:rays} an explicit description of the rays of $\Cpoly$ (modulo translation).
\end{remark}
%
%

\section{Alcoved envelopes }\label{sec:alcenv}

Let $C \subset H_0 \subset \R^n$ denote the following polyhedral $(n-1)$-dimensional pointed cone
\begin{equation}\label{eq:typeAcone}
C:=\{x\in \R^n\mid x_1 \leq 0,\, x_1+ x_2\leq 0,\dots ,x_1 +\cdots+ x_{n-1}\leq 0,\, x_1+\cdots + x_n =0\}.
\end{equation}
Define the {\it dominance order\/} on the hyperplane $H=H_k$ as the
partial order $x \preceq  y$ if and only if $x_1 \leq y_1$, $x_1 +
x_2 \leq y_1 + y_2$, $\ldots$, or equivalently, $x-y \in C$. For any
bounded subset $Q\subset H$, there is a unique maximal in the
dominance order point $v=v(Q)\in H$ such that $Q\subset v + C$.
Informally, $v+C$ is the cone containing $Q$ such that every facet
of $v+C$ touches $Q$.  Note that the cone $C$ is an alcoved polyhedron: its facets have normals given by $h_i-h_n$.

Let $c = (12\cdots n) \in S_n$ be the long cycle, with $c = 23 \cdots n1$ in
one-line notation.  The permutations $w \in S_n$ act on $\R^n$ by
the formula
$$
w \cdot (x_1, x_2,\ldots, x_n) =
(x_{w^{-1}(1)},x_{w^{-1}(2)},\ldots,x_{w^{-1}(n)}).
$$
Then $c$ acts on $\R^n$ as the {\it cyclic shift\/} linear operator
given by $c(e_i) = e_{i+1}$, for $i=1,\dots,n$. (Here and below we
assume that indices $i$ are taken modulo $n$.)  For $i \in \Z/n\Z$, define the cyclically shifted cone $C_i := c^{i-1}(C)$.  Thus, $C_1 = C$.  

If $\v = (v^{(1)},\ldots,v^{(n)})$ is a sequence of points in $H$ we
denote by $Q(\v)$ the intersection 
\begin{equation}\label{eq:Qdef}
Q(\v):=\bigcap_{i \in \Z/n\Z}
(v^{(i)}+C_{i}).
\end{equation}
It is clear that $Q(\v)$ is an alcoved polytope whenever it is non-empty.

For a bounded
subset $Q\subset H$, define $v^{(i)} = v^{(i)}(Q) := c^{i-1}(v(c^{1-i}(Q)))$.  Again, the
cone $v^{(i)} + C_i$ contains $Q$, and every facet of $v^{(i)} + C_i$ touches $Q$.

\begin{lemma}\label{lem:envelope}
For the cones $C_i$ and points $v^{(i)}$, $i \in \Z/n\Z$, as above, the
alcoved envelope of $Q$ is the intersection of the following cones
$$
\env(Q) = \bigcap_{i \in \Z/n\Z} (v^{(i)}+C_i).
$$
\end{lemma}

\begin{proof}  Let $f_{[r,s]}$ be the minimal number such that $Q$ belongs to the halfspace
$S_{[r,s]}=\{x\in H\mid \sum_{i\in[r,s]} x_i \leq f_{[r,s]}\}$, that is the (unique) facet of $S_{[r,s]}$ touches $Q$.
Then the alcoved envelope of $Q$ is the alcoved polytope $\env(Q) = \bigcap_{r,s} S_{[r,s]}$.  For fixed $r$, the
intersection $\bigcap_{s} S_{[r,s]}$ is the affine translation $v^{(r)} + C_r$ of the cone $C_r$ that satisfies the conditions
of the lemma.   So $\env(Q) = \bigcap_{r \in \Z/n\Z} (v^{(r)}+C_r)$, as needed.
\end{proof}


\begin{proposition}\label{prop:genpermbalanced} Suppose $P$ is a
generalized permutohedron in $H$ with vertices $\{v_w \mid w \in S_n\}$. Then
\begin{enumerate}
\item
the sequence $(v_{\id}, v_{c},
v_{c^{2}},\ldots,v_{c^{n-1}})$ is balanced,
\item
one has $v_{c^{i-1}} = v_{i}(P)$ for $i = 1, \ldots, n$.
\end{enumerate}
\end{proposition}
\begin{proof}
For (1), we show that all but the first coordinate of $v_{c} -
v_{\id}$ is nonnegative, and a similar argument shows that the
entire sequence is balanced.  We have $c = 23 \cdots 1$ in one-line notation.  By Theorem \ref{thm:PRW}(3),
\begin{align*}
v_{213 \cdots n} - v_{123\cdots n} &\in &&\R_{\geq 0}((-2,-1,-3,\ldots) - (-1,-2,-3,\ldots)) = \R_{\geq0}(e_2-e_1) \\
v_{2314 \cdots n} - v_{213\cdots n} &\in&& \R_{\geq 0}((-3,-1,-2,\ldots) - (-2,-1,-3,\ldots)) = \R_{\geq 0}(e_3-e_1) \\
\cdots  \\
v_{23 \cdots (n-1) n 1}-v_{23 \cdots(n-1) 1 n}  &\in &&\R_{\geq0}((-n,-1,-2,\ldots,1-n) - (1-n,-1,-2,\ldots,-n))  \\
&&& =\R_{\geq 0}(e_n-e_1).
\end{align*}
It follows that all but the first coordinate of $v_{c} - v_{\id}$
is nonnegative.

For (2), we note that all edges of $P$ incident to $v_{c^{i-1}}$ are
in the same direction as an edge of the cone $C_{i}$.  It follows that
$P \subset v_{c^{i-1}} + C_i$, so that $v_{c^{i-1}} = v_i(P)$.
\end{proof}

\section{Parametrization of polypositroids}
\label{sec:para_poly}

\subsection{Coxeter necklaces}
\begin{definition}\label{def:Anecklace}
Let $\v = (v^{(1)},v^{(2)},\ldots,v^{(n)})$ be a sequence of points
in $H$. We say that $\v$ is a {\it Coxeter necklace} if, for each $i$, we have
\begin{equation}
\label{eq:neckdef}
v^{(i+1)} - v^{(i)} \mbox{ is nonnegative in all coordinates except the
$i$-th coordinate.}
\end{equation}
Here, the superscript $i$ is taken modulo $n$.
\end{definition}

\begin{remark}\label{rem:Grassmann}
The set of $(k,n)$-Grassmann necklaces (see Section~\ref{sec:positroids}) is exactly the set of Coxeter necklaces $\v= (v^{(1)},v^{(2)},\ldots,v^{(n)})$ in $H_k$ such that 
each $v^{(i)}$ is a $01$-vector.
\end{remark}

The sum of two Coxeter necklaces $\v \in H_k$ and $\v'\in H_{k'}$ is a Coxeter necklace $\v''\in H_{k+k'}$.  Recall that the polytope $Q(\v)$ was defined in \eqref{eq:Qdef}.

\begin{lemma}
The space of Coxeter necklaces for varying $H=H_k$ form a cone.  The map $\v \to Q(\v)$ induces a homomorphism of cones from Coxeter necklaces to $\Calc$.
\end{lemma}

It will follow from Theorem~\ref{thm:main} below that the image of this homomorphism of cones is $\Cpoly$.

\begin{lemma}\label{lem:balancedvertex}
Suppose that $\v=(v^{(1)},v^{(2)},\ldots,v^{(n)})$ is a Coxeter necklace. Then
each $v^{(i)}$ is a vertex of $Q(\v)$.
\end{lemma}
\begin{proof}
It suffices to show that $v^{(i)} \in Q(\v)$.  We will show that
$v^{(i)}_{[1,j]} \leq v^{(1)}_{[1,j]}$; the full set of inequalities
follows by cyclic symmetry.  Suppose $i \leq j+1$.  Then
$$
(v^{(i)}-v^{(1)})_{[1,j]} = \sum_{k \in [2,i]}
(v^{(k)}-v^{(k-1)})_{[1,j]} \leq 0
$$
using the definition \eqref{eq:neckdef} and the fact $k-1 \in [1,j]$.
Suppose $i > j +1$.  Then
$$
(v^{(1)}-v^{(i)})_{[1,j]} = \sum_{k \in [i+1,1]}
(v^{(k)}-v^{(k-1)})_{[1,j]} \geq 0
$$
using the definition \eqref{eq:neckdef} and the fact that $k-1 \notin
[1,j]$.
\end{proof}

It follows from Lemma~\ref{lem:balancedvertex} that the map $\v \to Q(\v)$ is injective.  


\subsection{Balanced digraphs}
\begin{definition}
Let $G$ be a real-weighted directed graph on the vertex set $[n]$.  
Let $m_{ij} \in
\R$ denote the weight of the edge $i \to j$ of $G$.  
(We assume that $m_{ij}=0$ if $G$ does not contain the edge $i\to j$.)
We say that $G$ is {\it balanced} if
\begin{enumerate}
\item
the weight of non-loop edges are nonnegative, that is $m_{ij} \in
\R_{\geq 0}$ when $i \neq j$;
\item the total outdegree is
equal to the the total indegree of each vertex; that is, for each
$i$ one has
\begin{equation}\label{eq:bal}
\sum_{j=1}^n m_{ij} = \sum_{j =1}^n m_{ji}.
\end{equation}
\end{enumerate}
\end{definition}
\begin{remark}
The set of decorated permutations (see Section \ref{sec:positroids}) is the same as the set of balanced graphs satisfying the following two conditions: (a) every vertex has one non-zero incoming edge and one
non-zero outgoing edge, and (b) all edges have weight $\pm 1$ (only loop edges can have weight $-1$).
\end{remark}

\begin{lemma}\label{lem:sum}
Suppose $G$ is a balanced digraph.  Then the sum
$$
S_i(G) := \sum_{j \in [n]} (m_{j,j} + m_{j+1,j} + \cdots + m_{i,j})
$$
does not depend on $i$.
\end{lemma}
\begin{proof}
We have $S_i - S_{i+1} = \sum_{j \in [n]} m_{j,i+1} - \sum_{j \in
[n]} m_{i+1,j} = 0. $\end{proof} For a balanced digraph $G$, we
define $S(G) = S_i(G)$ to be the sum of Lemma \ref{lem:sum}.

Note that the space of balanced graphs forms a cone in a natural way: $G= \alpha G' + \beta G''$ has weights given by $m_{ij}= \alpha m'_{ij} + \beta m''_{ij}$.
\begin{proposition}\label{prop:rays}
Every balanced graph $G$ is a linear combination of balanced directed cycles (including loops and cycles of length $2$) such that the coefficient of those cycles that are not loops is nonnegative.
\end{proposition}
\begin{proof}
Let $G$ be a balanced graph.  We may assume that $m_{ii}=0$.  Let $m_{ij} >0$ be minimal amongst the positive weights.  Then the unweighted graph underlying $G$ must contain a directed cycle $C$ containing the edge $i \to j$.  All the weights $m_e$ for $e$ an edge of $C$ satisfy $m_e \geq m_{ij}$.  Thus, we may write $G = G' + C(m_{ij})$ where $C(m_{ij})$ is the balanced directed cycle where all edges have weight $m_{ij}$, and $G'$ is still a balanced graph.  But $G'$ has fewer edges (with nonzero weight) than $G$, so repeating this reduction we deduce the proposition.
\end{proof}

Let $\v = (v^{(1)},v^{(2)},\ldots,v^{(n)})$ be a balanced sequence of points in $H$.
We define a weighted directed graph $G(\v)$ by the formula
\begin{equation}
\label{eq:Gv}
m_{ij} = \begin{cases} (v^{(i+1)}- v^{(i)})_j &\mbox{ if $j \neq
i$,} \\
v^{(i+1)}_i & \mbox{if $j = i$.} \end{cases} 
\end{equation}



\begin{lemma}\label{lem:balancedseqgraph}
The map $\v \mapsto G(\v)$ is a bijection between Coxeter necklaces in $H$, and balanced graphs $G$ satisfying $S(G) = k$.  Furthermore, allowing $k$ to vary, we obtain an isomorphism of cones.
\end{lemma}
\begin{proof} We first check that $G(\v)$ is a balanced graph.
By definition of a balanced sequence, $m_{ij} \in \R_{\geq 0}$.  We
have for each $i \in [n]$,
\begin{align*}
\sum_{j \in [n]} m_{ij} &=\sum_{j \in [n]} (v^{(i+1)} - v^{(i)})_j +
v^{(i)}_i 
= v^{(i)}_i  
 = \sum_{j \in [n]} (v^{(j+1)} - v^{(j)})_i + v^{(i)}_i  
= \sum_{j \in [n]} m_{ji},
\end{align*}
so $G(\v)$ is balanced.  Finally, using $v^{(i+1)}_j = m_{j,j}+
\cdots + m_{i,j}$, one obtains $S(G) = \sum_j v^{(i+1)}_j = k$.

Conversely, suppose we are given a balanced graph satisfying $S(G) =
k$.  Then we define a sequence $\v$ by $v^{(i+1)}_j = m_{j,j}+
\cdots + m_{i,j}$.  It is easy to verify that $\v$ is a balanced
sequence in $H$, and that this is inverse to $\v \mapsto G(\v)$.

The last statement follows immediately from the linearity of \eqref{eq:Gv}.
\end{proof}
We write $\v(G)$ for the Coxeter necklace associated to a balanced digraph $G$.

\subsection{Parametrization}
\begin{theorem}\label{thm:main}
There are natural bijections between the following sets:
\begin{enumerate}
\item The set of all polypositroids $P \subset H_k$.
\item The set of all Coxeter necklaces in $H_k$.
\item The set of all balanced digraphs $G$ with $S(G) = k$.
\end{enumerate}
Furthermore, these bijections are compatible with the respective cone structures on the three sets.
\end{theorem}

After Lemma~\ref{lem:balancedseqgraph}, it suffices to show that the map $\v \to Q(\v)$ is a bijection between Coxeter necklaces and polypositroids.  We delay the proof of Theorem~\ref{thm:main} to Section~\ref{ssec:proofmain}.

The following result follows easily from the definition of the cone of balanced digraphs.
\begin{corollary}\label{cor:conefacets}
The cone of balanced digraphs $G$ satisfying $S(G) = 0$ has $n(n-1)$ facets, given by the inequalities $m_{ij} \geq 0$ for $i \neq j$.  Thus, the cone $\Cpoly$ also has $n(n-1)$ facets.  
\end{corollary}

The group $H_0=\R^{n-1}$ of translations preserves the affine hyperplane $H_k$ and acts on the cone $\Cpoly$ of polypositroids via the formula $z: (a_{ij}) \mapsto (a_{ij} + (h_i-h_j)(z))$ where $z = (z_1,z_2,\ldots,z_n) \in H_0$.  The lineality space of the cone $\Cpoly$ can be identified with $\R^{n-1}$.  Let $\Cpoly' := \Cpoly/\R^{n-1}$ denotes the quotient cone, which may be identified with the set of polypositroids inside $H_0$, modulo translations.

\begin{corollary}\label{cor:rays}
  The cone $\Cpoly'$ is a pointed cone, with extremal rays corresponding to the polypositroids $Q(\v(G))$, where $G$ is a balanced directed cycle (including cycles of length two).
\end{corollary}
\begin{proof}
By Theorem~\ref{thm:main}, we are equivalently considering the cone of balanced digraphs with $S(G) = 0$.  The translation action of $\R^{n-1}$ on $H_0$ corresponds to changing the weight of the $n$ loop edges of $G$.  Ignoring the weights of the loops, the statement then follows from Proposition~\ref{prop:rays}.
\end{proof}

\subsection{Face graphs}\label{ssec:facegraphs}
Let $P$ be an alcoved polytope and $F$ a face of
$P$. 
We define the graph $T_F$ of $P$ at $F$ to be the directed graph on
$[n]$ with a directed edge $(i \to j)$ whenever $F$ lies on the hyperplane $(h_j- h_i)(x) \leq a_{ji}$, or equivalently, on the hyperplane $x_{[i+1,j]} = f_{[i+1,j]}$.  We say that a digraph $T$ on $[n]$ is {\it noncrossing} if, when
the graph is drawn inside a circle with the vertices $[n]$ arranged
in clockwise order, there are no intersections in the interior of
the circle.  We say that a digraph $T$ on $[n]$ is {\it alternating}
if no vertex has both incoming and outgoing edges.

\begin{lemma}\label{lem:possibleedge}
Suppose $E$ is an edge of an alcoved polytope $P$,
with a noncrossing graph $T_E$. Then $E$ is parallel to $e_i - e_j$
for some $i$ and $j$.
\end{lemma}
\begin{proof}
Let us take a minimal subgraph $T \subset T_E$ such that the
corresponding $(n-2)$ facets still define (the affine span of) $E$.
We claim that $T$ is a forest with two components.  It is enough to
show that the underlying undirected graph of $T$ has no cycles.  A cycle in $T$ would correspond to a
linear dependence in the equations defining $E$, contradicting the
minimality of $T$.  For the remainder of the proof, it is enough to
think of $T$ as an undirected forest.

Let the two components of $T$ be $T_1$ and $T_2$.  By the
noncrossing assumption, it is clear that $T_1$ and $T_2$ are
induced subgraphs of $T$ on cyclic intervals $[i,j-1]$ and
$[j,i-1]$.

Let $x = (x_1,\ldots,x_n)$ lie on $E$.  We claim that if $k \notin
\{i, j\}$ then $x_k$ is fixed (that is, constant on $E$). For
simplicity, we shall assume that $x_k$ lies in a component $T_1$ of
$T$ which is an induced subgraph on a usual interval (no
wraparound). First note that an edge of the form $(r, r+1) \in T_1$
completely determines $x_{r+1}$. Also if $(r, s) \in T$ is an edge,
and all but one of the coordinates $\{x_{r+1}, x_{r+2}, \ldots,
x_{s}\}$ is determined, then the last coordinate is also determined.
By induction on the length of edges, we thus see that for each edge
$(r,s) \in T_1$ all of $\{x_{r+1}, x_{r+2}, \ldots, x_{s}\}$ are
determined, proving our claim that $x_k$ is fixed for $k \notin \{i,j\}$.

Thus only the coordinates $x_i$ and $x_j$ vary on $E$.  Since $E \in
H$, we deduce that $E$ is parallel to $e_i - e_j$.
\end{proof}

\begin{remark}
The converse of Lemma \ref{lem:possibleedge} is false.  For example, consider an alcoved polytope $P$
in $H = \{x \in \R^5 \mid x_1 + x_2 + \cdots + x_5 = 0\}$ where $x_2 \leq 0$, $x_2 + x_3 \leq 0$, and $x_3+x_4 \geq 0$
are all facets, and so that $x_2 = x_2+x_3 = x_3+x_4 = 0$ defines an edge $E$ of $P$.  Then $T_E$ contains the directed
edges $(1 \to 2), (1 \to 3), (2 \to 4)$ and is alternating but not noncrossing.  However, the edge $E$ is clearly in the direction $e_1 - e_5$.
\end{remark}

Let us say that a Coxeter necklace $(v^{(1)},\ldots,v^{(n)})$ is
{\it generic} if every coordinate of $v^{(i+1)} - v^{(i)}$ is
non-zero, for every $i$.  This is equivalent to saying that all
 the non-loop edges of the graph $G(\v)$ are non-zero.


\begin{lemma}\label{lem:noncrossing}
Let $(v^{(1)},\ldots,v^{(n)})$ be a generic Coxeter necklace in
$H$. Then any face $F$ of the alcoved polytope $Q = \bigcap_{i\in \Z/n\Z}
(v^{(i)} + C_{i})$ has a noncrossing and alternating graph $T_F$.
\end{lemma}
\begin{proof}
Let $F$ be a face such that $T_F$ is either not noncrossing or not alternating.

Suppose $[r,s]$ and $[r',s']$ are cyclic intervals so that the
corresponding directed edges $(r-1) \to s$ and $(r'-1) \to s'$ in $T_F$ are
either crossing or form a directed path $(r-1) \to s = (r'-1) \to
s'$.  In the crossing case, we may assume that $r < r' \leq s < s'<
r$ (interpreted in a cyclic manner). We have the following equations
and inequalities
\begin{align}
\label{E:f1} x_{[r,s]} &= v^{(r)}_{[r,s]}  \\
\label{E:f2}x_{[r',s']}&= v^{(r')}_{[r',s']} \\
\label{E:f3}x_{[r,s']} &\leq v^{(r)}_{[r,s']}  \\
\label{E:f4}x_{[r',s]} &\leq v^{(r')}_{[r',s]}
\end{align}
for points $x = (x_1,x_2,\ldots,x_n)$ on the face $F$.  By \eqref{E:f1} and
\eqref{E:f4}, we have
$$
x_{[r,r'-1]} = x_{[r,s]}- x_{[r',s]} \geq v^{(r)}_{[r,s]} -
v^{(r')}_{[r',s]}
$$
and by \eqref{E:f2} and \eqref{E:f3}, we have
$$
x_{[r,r'-1]} = x_{[r,s']} - x_{[r',s']} \leq v^{(r)}_{[r,s']} -
v^{(r')}_{[r',s']}.
$$
Equating the two expressions for $x_{[r,r'-1]}$, we obtain
$$
v^{(r)}_{[s+1,s']} \geq v^{(r')}_{[s+1,s']}.
$$
This is impossible because $\v$ is generic balanced, implying that
the coordinates in the positions $[s+1,s']$ of $v^{(r')}- v^{(r)}$
are all positive.
\end{proof}

\subsection{Proof of Theorem~\ref{thm:main}}\label{ssec:proofmain}
Let $\v = (v^{(1)},\ldots,v^{(n)})$ be a Coxeter necklace in $H$.
We show that $Q(\v)$ is a polypositroid.  First, suppose that $\v$ is generic. Then by Lemma \ref{lem:noncrossing} and
Lemma \ref{lem:possibleedge}, $Q(\v)$ is a generalized
permutohedron, and thus a polypositroid.

Now suppose that $\v$ is not generic, and let $G(\v)$ be the balanced digraph under the bijection of Lemma~\ref{lem:balancedseqgraph}.  Let $G_\varepsilon$ be obtained from $G(\v)$ by adding $\varepsilon >0$ to every non-loop edge.  It immediate that $G(\varepsilon)$ is again a balanced digraph, and we define the Coxeter necklace $\v_\varepsilon$ by $G(\v_\varepsilon) = G_\varepsilon$.  Then 
$\v = \lim_{\varepsilon \to 0} \v_\varepsilon$ is a limit of the
generic balanced sequences $\v_\varepsilon$.  For sufficiently
small, but non-zero $\varepsilon$ the combinatorial type of the
polytope $Q(\v_\varepsilon)$ corresponding to $\v_\varepsilon$ does
not change.  The alcoved polytope $Q$ is thus a deformation of such
a $Q(\v_\varepsilon)$, in the sense of moving facets.  Since
$Q(\v_\varepsilon)$ is a generalized permutohedron, so is $Q(\v)$; see for example \cite{CL}.

Now suppose $P$ is a polypositroid with vertices $v_w$.  Since $P$
is a generalized permutohedron, by Proposition
\ref{prop:genpermbalanced} and Lemma \ref{lem:envelope}, $\v =
(v_{\id}, v_{c}, v_{c^{2}},\ldots,v_{c^{n-1}})$ is
balanced, and we have $\env(P) = Q(\v)$.  But $P$ is also alcoved,
so we have $P = \env(P) = Q(\v)$.  Thus the map $\v \mapsto Q(\v)$
is surjective.  Finally, it follows from Lemma \ref{lem:balancedvertex} that $\v
\mapsto Q(\v)$ is injective.  This proves the equivalence of (1) and (2) in Theorem~\ref{thm:main}.

\subsection{Proof of Theorem~\ref{thm:subpoly}}\label{ssec:subpolyproof}
The following result follows immediately from Proposition \ref{prop:genpermbalanced}(1) and Theorem
\ref{thm:main}.
\begin{proposition}\label{prop:alcgen}
The alcoved envelope of a generalized permutohedron is a generalized permutohedron, and thus a polypositroid.
\end{proposition}
Recall that $\pi_\I$ denotes the composition of the restriction map $f_P \mapsto f_P|_{\cyc}$ with the transformation \eqref{eq:af} from $f_{[r,s]}$-coordinates to $a_{ij}$ coordinates.  Let $P$ be a generalized permutohedron and $f_P|_{2^{[n]}} \in \Csub$ be its support function.  Then $\pi_\I(f_P)$ represents the alcoved polytope $\env(P)$.  By Proposition~\ref{prop:alcgen}, we thus have $\pi_I(\Csub) \subseteq \Cpoly$.  But if $P$ is a polypositroid then $\env(P) = P$.  It follows that $\pi_\I(\Csub) = \Cpoly$. 

\section{Components of a polypositroid}
A {\it noncrossing partition} $\tau = (\tau_1|\tau_2|\cdots|\tau_\ell)$ of $[n]$ is a partition of $[n]$ such that there do not exist $1 \leq i < j < k <l \leq n$ such that $i,k \in \tau_a$ and $j,l \in \tau_b$ for $a \neq b$.
Let $G$ be a graph on $[n]$.  Then there exists a finest noncrossing partition $\tau(G)= (\tau_1|\tau_2|\cdots|\tau_\ell)$ of $[n]$ such that $G$ is the disjoint union of the induced subgraphs $G|_{\tau_a}$, $a = 1,2,\ldots,\ell$.  Note that the graphs $G|_{\tau_a}$ need not be connected.

If
$x = (x_1,x_2,\ldots,x_n) \in \R^n$, then $x|_{\tau_a}  \in \R^{\tau_a}$ denotes the
vector obtained by projecting $x$ to the components indexed by $\tau_a$.  Given $\v =
(v^{(1)},\ldots,v^{(n)})$ and $\tau_a = \{i_1,i_2,\ldots,i_t\}$, we define
$$
\v|_{\tau_a} := (v^{(i_1)}|_{\tau_a}, v^{(i_2)}|_{\tau_a}, \ldots,
v^{(i_t)}|_{\tau_a}).
$$
The following result should be compared to \cite[Theorem 7.6]{ARW} in the positroid case.

\begin{lemma} \label{lem:product}
Let $\v$ be a Coxeter necklace and let $\tau(G(\v))= (\tau_1|\tau_2|\cdots|\tau_\ell)$.
Then $\v|_{\tau_a}$, $a =1,2,\ldots, k$ are Coxeter necklaces satisfying $G(\v|_{\tau_a}) = G|_{\tau_a}$, and we have
$$
Q(\v)= Q(\v|_{\tau_1}) \times Q(\v|_{\tau_2})
\times\cdots \times Q(\v|_{\tau_\ell})
$$
where $Q(\v|_{\tau_a})$ lies inside $\R^{\tau_a}$.
\end{lemma}
\begin{proof}
The lemma holds more generally for $\tau$ any noncrossing partition such that $G$ is the disjoint union of the induced subgraphs $G|_{\tau_a}$ i.e., $\tau$ need not be chosen finest.  The sums $k_a:= v^{(j)}_{\tau_a} = \sum_{i \in \tau_a} v^{(j)}_i$ (not to be confused with the projection $v^{(j)}|_{\tau_a}$) do not depend on $j \in [n]$.  It follows that $Q(\v)$ lies in the
hyperplane $\sum_{i \in \tau_a} x_i= k_a$, and $Q(\v_{\tau_a})$ lies in the same hyperplane intersected with $\R^{\tau_a}$.

The first statement of the lemma is
straightforward.  For the second statement, we may assume by induction that $\tau = (\tau_1|\tau_2)$ where $\tau_a$ are cyclic intervals.  The polytope $Q(\v)$ is cut out by the inequalities $x_{[r,s]} \leq v^{(r)}_{[r,s]}$.  Suppose that $[r,s] \cap \tau_1$ and $[r,s] \cap \tau_2$ are both non-empty.  For simplicity, we suppose that $[r,s] \cap \tau_1 = [r,t]$ and $[r,s] \cap \tau_2 = [t+1,s]$.  The assumption that there are no edges in $G(\v)$ between $\tau_1$ and $\tau_2$ implies that $v^{(r)}|_{[t+1,s]}= v^{(t+1)}|_{[t+1,s]}$.  Thus $v^{(r)}_{[r,s]} = v^{(r)}_{[r,t]} + v^{(t+1)}|_{[t+1,s]}$.  It follows that the inequality $x_{[r,s]} \leq v^{(r)}_{[r,s]}$ is implied by the inequalities $x_{[r,t]} \leq v^{(r)}_{[r,t]} $ and $x_{[t+1,s]} \leq v^{(t+1)}|_{[t+1,s]}$.  The latter inequalities are among those cutting out $Q(\v_{\tau_1})$ and $Q(\v_{\tau_2})$ respectively.  It follows that $Q(\v) = Q(\v_{\tau_1}) \times Q(\v_{\tau_2})$.
\end{proof}

\begin{proposition}
Let $\v$ be a Coxeter necklace.  The dimension of the polypositroid
$Q(\v)$ is equal to $n - \#\{\text{parts in $\tau(G)$}\}$.
\end{proposition}
\begin{proof}
By induction and Lemma \ref{lem:product}, it suffices to show $\dim(Q(\v)) = n-1$ whenever $\tau(G)=([n])$ has a single part.

Assume that $\tau(G) = ([n])$ and that $Q(\v)$ has dimension less than $n-1$.  Then since $Q(\v)$
is alcoved, it must lie in some hyperplane $(h_i - h_j)(x) = a_{ij}$.  This implies that
$v^{(j+1)}_{[j+1,i]} = a_{ij} = v^{(i+1)}_{[j+1,i]}$.  But
$$(v^{(i+1)}-v^{(j+1)})_{[j+1,i]} = \sum_{k \in [j+2,i+1]}
(v^{(k)}-v^{(k-1)})_{[j+1,i]} \leq 0
$$
with equality if and only if there are no edges in $G(\v)$ from the
vertices $[j+1,i]$ to $[i+1,j]$.  The same argument shows that there
are no edges from $[i+1,j]$ to $[j+1,i]$, so $\tau(G)$ must be a refinement of the partition $\tau = ([i+1,j]|[j+1,i])$, contradicting our assumption.
\end{proof}

\section{Normal fans of polypositroids}\label{sec:normal}
\subsection{Normal fans to generic simple alcoved polytopes}
Recall that we say that an alcoved polytope is {\it generic} if every inequality
$(h_i-h_j)(x) \leq a_{ij}$ determines a facet.  The $f$-vector $(f_0,f_1,\ldots,f_{d})$ of a $d$-dimensional polytope $P$ is given by
$$f_i := \#\{\text{$i$-dimensional faces of $P$}\}.$$

\begin{theorem}\label{thm:fvector}
The $f$-vectors of any two generic simple alcoved polytopes $P\subset \R^n$  
are the same.  The face numbers are given by 
\begin{equation}\label{eq:cyclof}
f_{i}= \binom{n-1}{i} \binom{2n-i-2}{n-1} \qquad \mbox{ for $i=0,1,\dots,n-1$.}
\end{equation}
\end{theorem}

  The {\it root polytope} $R$ is the convex hull of
the vectors $\{ h_i - h_j \mid i, j \in \{1,2,\ldots,n\}\}$, which
we think of as lying inside $(\R^n)^*/h_n$.   If $T$ is a directed tree on $[n]$ (or more generally, a directed
graph on $[n]$), we let $\Delta_T \subset (\R^n)^*$ denote the convex hull of the
points $\{h_i - h_j \mid j \to i \; \text{is an edge of $T$}\} \cup
\{0\}$.  
A {\it local triangulation} of $R$ is a triangulation such that every simplex is one of the $\Delta_T$. 

Let $P$ be a generic
simple alcoved polytope. The normal fan $\F_P$ of $P$ lies in
$(\R^n)^*/h_n$.  The condition that $P$ is generic implies that
every root $h_i - h_j \neq 0$ is an edge of $\F_P$.  The condition
that $P$ is simple implies that each maximal cone $C_v$ of $P$ is spanned
by $(n-1)$ roots.  Thus the collection of maximal cones of $\F_P$
induces a local triangulation of $P$.  Let $C_T := \sp_{\geq 0} \Delta_T$ be the cone spanned by $\Delta_T$.

\begin{proof}[Proof of Theorem~\ref{thm:fvector}]
The $f$-vector of $P$ is given by counting the number of cones of
each dimension of $\F_P$, which is the same as counting the number of
simplices (with the origin as a vertex) of each dimension in the corresponding local triangulation
of the root polytope $R$.

We claim that every local triangulation of $R$ has the same number
of simplices of each dimension.  To see this we note that since the type $A$ root system
is unimodular, every simplex $\Delta_T$ has the same volume (in fact, normalized volume 1).
The number $\#\{\Int(m\Delta) \cap \Z^r\}$ of integer points lying in the interior of an integer scalar multiple of a normalized volume 1 simplex $\Delta$ with integer coordinates depends only on the scalar multiple $m$ and the dimension $\dim(\Delta)$ of the simplex.

It follows easily from this that the
Erhart polynomial of $R$ can be written in terms of, and in fact
determines, the number of simplices (with the origin as a vertex) in each dimension of a local triangulation.  But clearly the Erhart
polynomial of $R$ does not depend on the triangulation of $R$.

The cyclohedron (defined in \eqref{E:cyclo}) is a generic and simple polypositroid.  Thus every generic simple polypositroid has the same $f$-vector as the cyclohedron. 
According to \cite{Sim},  the 
$f$-vector of the cyclohedron is given by (\ref{eq:cyclof}).
For example, the two-dimensional cyclohedron is a hexagon with face numbers $(f_0,f_1,f_2) = (6,6,1)$.
\end{proof}



\subsection{Matching ensembles}
Let $P$ be a generic simple alcoved polytope.   Let $\T(P)$ denote the set of vertex trees of $P$.  The data of $\T(P)$ is equivalent to the knowledge of the normal fan $\F_P$.   Our convention is that $\F_P$ denotes the outer normal fan.  Thus the fan $\F_P$ is complete, and the maximal cones $C_v$ of $\F_P$ are indexed by vertices $v$, such that $C_v$ is the positive span of the vectors $h_i - h_j$ for $j \to i$ an edge of $T_v$.

The first part of the following result is similar to \cite[Lemma 13.2]{P}.

\begin{lemma}\label{lem:Talternating} 
Let $P$ be a generic simple alcoved polytope, and $v$ a vertex
of $P$.  Then the tree $T_v$ is alternating.  
Furthermore, if $P$ is a polypositroid, then $T_v$ is in addition noncrossing. 
\end{lemma}
\begin{proof}
If $j \to i$ and $k \to j$ both belong to $T_v$, then $h_i - h_k \in C_v$ and is not one of the edges of $C_v$, contradicting the assumption that all roots are edges of $\F_P$.  We conclude that $T_v$ is alternating.

Now suppose that $P$ is a polypositroid.  Since $P$ is generic and simple, the polytope $P'$ obtained from a small pertubation of the facets of $P$ will have the same combinatorial
type as $P$.  We can pick such a $P'$ to be a polypositroid $Q(\v)$ for a generic Coxeter necklace $\v$.  It follows 
from Lemma \ref{lem:noncrossing} that the trees $T_v$ are noncrossing.
\end{proof}

Suppose $T$ is an alternating tree $[n]$.  A matching of $(I,J)$ in
$T$ is a collection of edges of $T$ which form a matching of $I$
with $J$, such that vertices in $I$ are sources, and the vertices in
$J$ are sinks.  Say that two directed alternating trees $T,T'$ on
$[n]$ are {\it compatible} if there do not exist disjoint subsets
$I, J \subset [n]$ of the same cardinality, such that both $T$ and
$T'$ contain matchings of $(I,J)$, and these matchings disagree.

\begin{lemma}[cf. {\cite[Lemma 12.6]{P}}] \label{lem:compatible}
Let $T, T'$ be distinct directed alternating trees on $[n]$.  The
intersection $C_T \cap C_{T'}$ is a common face of both
$C_T$ and $C_{T'}$ if and only if $T$ and $T'$ are
compatible.
\end{lemma}

\begin{proof} 
Suppose $T$ and $T'$ are not compatible.  Let $I, J \subset [n]$ be
disjoint such that $T$ (resp. $T'$) contains a matching $M$ (resp.
$M'$) from $I$ to $J$, such that $M \neq M'$.  We assume that $I$
and $J$ are chosen to be minimal, so that $M \cap M' = \emptyset$.
Let 
$$
x=\sum_{j \in J} h_j - \sum_{i \in I} h_i =
\sum_{(i \to j) \in M} h_j - h_i = \sum_{(i
\to j) \in M'} h_j - h_i.
$$
Clearly, $x \in C_T \cap C_{T'}$.  The minimal face of
$C_T$ containing $x$ is $C_M$.  The minimal face of
$C_{T'}$ containing $x$ is $C_{M'}$.  Since $M \neq M'$,
we conclude that $C_T \cap C_{T'}$ is not a common face.

Conversely, suppose that $T$ and $T'$ are compatible.  Let $F = T
\cap T'$ be the intersection, a directed forest on $[n]$.  Define a
partial order $\prec$ on the connected components (denoted $A$) of $F$, by letting
$A \prec A'$ if there is a (necessarily unique) sequence $A=A_0,
A_1, \ldots,A_\ell = A'$ of distinct components of $F$ such that $T$
has a (unique) directed edge $f_i$ joining $A_i$ to $A_{i+1}$ for $i
\in [0,\ell-1]$. Similarly, define $\prec'$ using $T'$.  We claim
that $A \prec A'$ if and only if $A' \prec' A$.  Assuming otherwise,
the sequence of components from $A$ to $A'$ for $T$ and from $A'$ to
$A$ for $T'$ can be assumed to be distinct except for $A$ and $A'$.
 Using the directed edges $f_i \in T$ from $A$ to $A'$ and $g_i \in
T'$ from $A'$ to $A$, together with some of the edges in $F$, one
obtains an alternating cycle of even length, such that (picking an
orientation) the clockwise edges belong to $T$, and the
counterclockwise edges belong to $T'$.  This immediately contradicts
the compatibility of $T$ and $T'$.

Now, let $f:[n] \to \R$ be a function with the following properties: it is constant with value $f(A)$ on the
components $A$ of $F$, and such that $f(A) < f(A')$ if and only if $A
\prec A'$ if and only if $A' \prec' A$.  Assume that $f(n) = 0$.
 Then $f$ extends to a linear function $\phi_f: (\R^n)^*/h_n \mapsto \R$,
by setting $h_i \mapsto f(i)$.  It follows by construction that
$\phi_f(C_T) \geq 0$, $\phi_f(C_{T'}) \leq 0$, and
$\phi_f(C_F) = 0$.  It follows that $C_T \cap C_{T'}
= C_F$ is a common face of both cones.
\end{proof}

A {\it matching field} on $[n]$ is a collection $\E = \{M_{I,J}\}$ of
matchings, one for each pair $(I,J)$ of disjoint subsets of $[n]$ of
equal size, such that for each $I' \subset I$ and $J' \subset J$
where $I'$ is matched to $J'$ in $M_{I,J}$, we have $M_{I',J'}$ is
the restriction of $M_{I,J}$ to $(I',J')$.  If $M_{I',J'}$ is a
restriction of $M_{I,J}$, we shall say that $M_{I,J}$ contains
$M_{I',J'}$.

We shall call a matching field {\it noncrossing}, if every
matching $M_{I,J}$ is noncrossing when drawn on the circle.

\begin{theorem}\label{thm:ensemble}
For each generic simple alcoved polytope $P$, there is a unique
matching field $\E(P)$ such that $\E(P)$ and the set of vertex trees
$\T(P)$ are related by the condition: $T \in \T(P)$ if and only if
all matchings in $T$ belong to the $\E(P)$.

Furthermore, if $P$ is a polypositroid then $\E(P)$ consists of noncrossing trees.
\end{theorem}
\begin{proof}
Let $P$ be a generic simple alcoved polytope, and $\T(P)$ be its set
of vertex trees.  We claim that for each pair $(I,J)$ of disjoint
subsets of $[n]$ of equal cardinality, some tree $T \in \T(P)$
contains a matching of $(I,J)$.  To see this, consider the point
$x_{I,J}=\sum_{j \in J} h_j - \sum_{i \in I} h_i$.  Since $\{C_T \mid T \in \T(P)\}$ are the
maximal cones of a complete fan, it belongs to $C_T$
for some $T \in \T(P)$.  But $T$ is alternating by Lemma
\ref{lem:Talternating}, and it follows that $T$ contains a unique matching of
$(I,J)$.  It follows from Lemma \ref{lem:compatible} that $\T(P)$
determines a unique matching ensemble $\E(P)$.  Furthermore, $\T(P)$
is exactly the set of trees $T$ such that all matchings in $T$
belong to $\E(P)$: if $T'$ is another tree satisfying this condition
then by Lemma \ref{lem:compatible} $C_{T'}$ can be added to the
complete fan $\{C_T \mid T \in \T(P)\}$,
which is a contradiction.

The last sentence follows from Lemma~\ref{lem:Talternating}.
%
\end{proof}

\begin{definition}[cf. \cite{OY,SZ}]
Let $\E$ be a matching field.  Then we say that $\E$ satisfies the \emph{linkage axiom} if
\begin{enumerate}
\item
for any disjoint $(I,J)$ of equal size, and $j' \in [n] \setminus (I \cup J)$, there is an edge $(i,j) \in M_{I,J}$ such that the matching $M'_{I,J'} := M_{I,J} \setminus \{(i,j)\} \cup \{(i,j')\}$ belongs to $\E$, where $J' = J \setminus \{j\} \cup \{j'\}$;
\item
for any disjoint $(I,J)$ of equal size, and $i'\, \in [n] \setminus (I \cup J)$, there is an edge $(i,j) \in M_{I,J}$ such that the matching $M'_{I',J} := M_{I,J} \setminus \{(i,j)\} \cup \{(i',j)\}$ belongs to $\E$, where $I' = I \setminus \{i\} \cup \{i'\}$;
\end{enumerate}
If $\E$ satisfies the linkage axiom, we say that $\E$ is a {\it matching ensemble}.
\end{definition}

%
%

\begin{proposition}
Let $P$ be a generic simple alcoved polytope and let $\E(P)$ be the appearing in Theorem~\ref{thm:ensemble}.  Then $\E(P)$ satisfies the linkage axiom, and is a matching ensemble.
\end{proposition}
\begin{proof}
Let $k = |I| = |J|$.  Consider the vector 
$$
x = \left(1+ \frac{1}{k}\right)\sum_{i \in I} h_i + \sum_{j \in J \cup\{j'\}}  h_j.
$$
Then $x$ lies in the cone $C_T$ for some $T \in \T(P)$.  Let $\tilde T = T|_{I \cup J \cup \{j'\}}$ denote the induced subgraph on $I \cup J \cup \{j'\}$.  It is not difficult to see that $\tilde T$ must be connected, and thus itself a tree.  Let $A_1,A_2,\ldots,A_r \subset I \cup J$ be the (vertex sets of the) connected components of the forest $\tilde T \setminus \{j'\}$ obtained by removing $j'$.  Looking at the coefficients of $h_t$, $t \in A_s$ in $x$, we deduce that $|A_s \cap I| = |A_s \cap J|$ for each $s = 1,2,\ldots,r$.  It follows that the matching $M_{I,J} \in \E(P)$ restricts to a matching on $(A_s \cap I, A_s \cap J)$ for each $s = 1,2,\ldots,r$.  

Now let $i \in (A_1 \cap I)$ be the vertex in $A_1$ connected to $j'$ and let $(i,j) \in M_{I,J}$ be the edge of $M_{I,J}$ incident to $i$.  Then $T$ also contains the matching $M_{I,J} \setminus \{(i,j)\} \cup \{(i,j')\}$, and by Theorem~\ref{thm:ensemble} so does $\E(P)$.  This establishes condition (1) of the linkage axiom for $\E(P)$.  Condition (2) is similar.
%
%
%
\end{proof}

\begin{conjecture}\label{conj:matching}
Every noncrossing matching ensemble appears as $\E(P)$ for some generic simple polypositroid $P$.
\end{conjecture}

\begin{remark}  
In some way, matching ensembles are analogous to matroids,
and matching ensembles of the form $\E(P)$ are analogous to 
{\it realizable} matroids.
So Conjecture~\ref{conj:matching} is similar in spirit to the result
of \cite{ARW17} that positive oriented matroids are 
positroids, i.e., they are realizable.
\end{remark}

The following examples support Conjecture~\ref{conj:matching}.

\begin{example}
Let $n =3$.  In this case, every alcoved polytope is automatically a polypositroid.  Indeed, there is a single matching field on $\{1,2,3\}$, it satisfies the linkage axiom, and it is noncrossing.  Thus there is only one possible normal fan for a generic simple polypositroid.
\end{example}

\begin{example}
Let $n = 4$.  Let $P$ be a generic simple polypositroid.  We use Theorem~\ref{thm:ensemble} to understand the possible choices for $\T(P)$.  By the noncrossing condition, a matching field $\E(P)$ is uniquely determined except for matchings on 
$(I,J) = (\{1,3\},\{2,4\})$ and $(I,J) = (\{2,4\},\{1,3\})$, each of which there are two choices of matchings, giving four possibilities for $\E(P)$, all of which satisfy the linkage axiom.  By an explicit calculation, for example by computing whether putative vertices $v_T$ lie inside $P$,  we find that the matching $M_{\{1,3\},\{2,4\}}$ (resp. $M_{\{2,4\},\{1,3\}}$) depends on the sign of $a_{41}+a_{23}- a_{43}-a_{21}$ (resp. $a_{12}+a_{34} - a_{14}-a_{32}$).  (If $a_{41}+a_{23}- a_{43}-a_{21} = 0$, then $P$ is not simple.)  Suppose that $P$ arises from the balanced graph $G$ via Theorem~\ref{thm:main}.  Then we have  
\begin{align}
\begin{split}\label{eq:2sign}
a_{41}+a_{23}- a_{43}-a_{21} &= m_{43}+m_{13} - m_{24}-m_{34} \\
a_{12}+a_{34} - a_{14}-a_{32}&= m_{14}+m_{24} - m_{31}-m_{41}
\end{split} .
\end{align}
(Using \eqref{eq:bal}, the RHS can be written in a number of equivalent ways.)  It is easy to construct generic balanced $G$ such that the RHS has any of the four possible ordered pairs of signs.  For example, a balanced directed cycle $(1 \to 4 \to 3 \to 1)$ (resp. $(4 \to 3 \to 2 \to 4)$) makes the first (resp. second) quantity in \eqref{eq:2sign} positive and the second (resp. first) 0.  It follows that there are exactly four normal fans of generic simple polypositroids for $n = 4$.
\end{example}
%
%

\begin{example}\label{ex:cyclo}
We consider the cyclohedron defined in \eqref{E:cyclo}.

\begin{proposition}\label{prop:cyclo}
Let $P$ be a cyclohedron.  Then the set $\T(P)$ of vertex trees of $P$ is the set of noncrossing, alternating trees
on $[n]$ with the following additional property: there is a cyclic rotation $i \prec i+1 \prec \cdots \prec n \prec 1 \prec \cdots \prec i-1$ 
of the usual order on $[n]$ such that every edge $(i,j)$ of $T_v$ satisfies $i \prec j$.
\end{proposition}
\begin{proof}
Let $r,s,r',s'$ be four indices in cyclic order.  Then it follows from \eqref{eq:cyclfrs} that $f_{[r,s]}+f_{[r',s']} > f_{[r,s']}+f_{[r',s]} - k $, so that
in any vertex tree $T_v$ the directed edges $(r-1) \to s$ and $(r'-1) \to s'$ cannot both be present.  The set of trees satisfying this condition is exactly the set of trees stated in the proposition.
To complete the proof it suffices to note from \eqref{eq:cyclof} that the cyclohedron has $\binom{2n-2}{n-1}$ vertices, and that the number of noncrossing, decreasing (every edge $(i \to j)$ satisfies $j < i$), alternating trees on $[n]$
is equal to the Catalan number $\frac{1}{n}\binom{2n-2}{n-1}$ (see \cite{P} or Remark~\ref{rem:bijectiontree} below).
\end{proof}

For example, let $n = 4$.  Then the cyclohedron has $20$ vertices, and each vertex tree is a rotation of one of the trees in Figure \ref{fig:trees}.

\begin{figure}
\begin{tikzpicture}[scale=0.8]
\tikzset{>=latex}
\draw[->,thick] (3,0) to [out=135,in=45] (0,0);
\draw[->,thick] (3,0) to [out=135,in=45] (1,0);
\draw[->,thick] (2,0) to [out=135,in=45] (1,0);
\begin{scope}[shift={(4,0)}]
\draw[->,thick] (3,0) to [out=135,in=45] (0,0);
\draw[->,thick] (2,0) to [out=135,in=45] (0,0);
\draw[->,thick] (2,0) to [out=135,in=45] (1,0);
\end{scope}
\begin{scope}[shift={(8,0)}]
\draw[->,thick] (3,0) to [out=135,in=45] (0,0);
\draw[->,thick] (2,0) to [out=135,in=45] (0,0);
\draw[->,thick] (1,0) to [out=135,in=45] (0,0);
\end{scope}
\begin{scope}[shift={(12,0)}]
\draw[->,thick] (3,0) to [out=135,in=45] (2,0);
\draw[->,thick] (3,0) to [out=135,in=45] (1,0);
\draw[->,thick] (3,0) to [out=135,in=45] (0,0);
\end{scope}
\begin{scope}[shift={(16,0)}]
\draw[->,thick] (3,0) to [out=135,in=45] (0,0);
\draw[->,thick] (3,0) to [out=135,in=45] (2,0);
\draw[->,thick] (1,0) to [out=135,in=45] (0,0);
\end{scope}
\end{tikzpicture}
\caption{Noncrossing, decreasing, alternating trees on $\{1,2,3,4\}$.}
\label{fig:trees}
\end{figure}
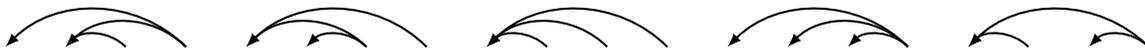

\end{example}

\begin{remark} The decomposition in Proposition~\ref{prop:cyclo} of $\T(P)$ into the $n$ cyclic rotations of the set of noncrossing, alternating, decreasing trees has the following geometric interpretation: the cones $C, c(C), c^2(C),\ldots,c^{n-1}(C)$ cover all of $H_0$ and intersect in lower-dimensional faces.
\end{remark}

\subsection{Proof of Theorem~\ref{thm:conenoncross}}\label{ssec:conenoncrossproof}

\begin{proposition}\label{P:four}
Let $P$ be a generic simple alcoved polytope given by the inequalities \eqref{eq:alcoveineq}.  Then $P$ is a polypositroid if and
only if for any four indices $i,j,k,l$ in cyclic order, we have
$a_{ik} + a_{jl} > a_{il} + a_{jk}$.
\end{proposition}
\begin{proof}
Suppose $P$ is a generic simple polypositroid.  By the uniqueness part of Theorem \ref{thm:ensemble} (or directly from the proof),
there is a vertex tree $T_v$ of $P$ which contains a noncrossing
matching of $(\{k,l\},\{i,j\})$.  Since $i,j,k,l$ are in cyclic
order, the matching must match $i$ with $l$ and match $j$ with $k$.  We
thus have
\begin{align*}
a_{il} + a_{jk} &= (h_{i}-h_{l})(v) +(h_j - h_k)(v) \\
&= (h_i - h_k)(v) + (h_j - h_l)(v) \\
&\leq a_{ik} + a_{jl}.
\end{align*}
Equality cannot occur, for otherwise $T_v$ will contain a cycle.   Conversely, if $P$ is a generic simple alcoved polytope and 
$a_{ik} + a_{jl} > a_{il} + a_{jk}$ holds for any four indices $i,j,k,l$ in cyclic order, then we deduce that $\E(P)$ is noncrossing, so $\T(P)$ consists of noncrossing trees and by Lemma~\ref{lem:possibleedge} we conclude that $P$ is a polypositroid.  
\end{proof}

Now, the inequalities $a_{ik} + a_{jl} > a_{il} + a_{jk}$ define an open subcone $C$ of $\Calc$ each point of which represents a generic alcoved polytope.  By Example~\ref{ex:cyclo}, the cyclohedron is a generic simple polypositroid and thus $C$ is nonempty.  An open dense subset of $C' \subseteq C$ corresponds to generic alcoved polytopes that are simple.  Applying Proposition~\ref{P:four}, we see that these polytopes are generic simple polypositroids.  The closure of $C'$ is the closed cone cut out by \eqref{eq:anoncross}.  The corresponding limits of generic simple polypositroids are nonempty (possibly not generic, possibly not simple) polypositroids, finishing the proof of the Theorem~\ref{thm:conenoncross}.

\subsection{Duality for alternating trees}
\def\Talt{\T_{{\rm alt}}}
\def\Tcir{\T_{{\rm cir}}}

A noncrossing tree $T$ on $[n]$ is called {\it circular-alternating} if for each vertex $v$, the edges incident to $v$ alternate between incoming and outgoing as they are read in order when $T_v$ is drawn on a circle.  Thus for example, if $v$ is incident to $u_1,u_2,u_3,u_4$ with $u_1<v<u_2<u_3<u_4$, then the edges $(v,u_1)$, $(v,u_4)$, $(v,u_3)$, $(v,u_2)$ alternate in direction.

Let $\Talt$ denote the set of alternating noncrossing trees and let $\Tcir$ denote the set of noncrossing, circular-alternating trees.  For $T \in \Talt$, let $T'$ be obtained from $T$ as follows: place the
numbers $1',1,2',2,\ldots,n',n$ in clockwise order around a circle,
and draw $T$ using the numbers $1,2,\ldots,n$.  Then $T'$ is the
unique tree on the numbers $1',2',\ldots,n'$ such that each directed
edge $c' \to d'$ of $T'$ intersects a unique directed edge $a \to
b$ of $T$, and furthermore, $a, d', b, c'$ are in clockwise order.

\begin{proposition}\label{prop:Tbij}
The map $T \mapsto T'$ gives a bijection $\varphi: \Talt \to \Tcir$.
\end{proposition}
\begin{proof}
The tree $T$ cuts the disk up into $n$ pieces, each containing exactly one of $1',2',\ldots,n'$.  Let $D$ be one of these components.  Then $D$ is bounded by an arc of the circle and a number of edges of $T$.  When read in order around the boundary of $D$, these edges of $T$ alternate in direction.  It follows that $T' \in \Tcir$.

In the other direction, let $D'$ be one of the components that $T' \in \Tcir$ divides the disk into.  Then the edges of $T'$ along the boundary of $D'$ are all in the {\it same} direction.  It follows that there is unique $T \in \Talt$ such that $\varphi(T) = T'$.
\end{proof}

For a tree $T$, let $C'_T$ denote the cone in $\R^n$ spanned by $e_i - e_j$ for each directed edge $j \to i$ in $T$.

\begin{proposition}\label{prop:dualcones}
For $T \in \Talt$, the cones $C_T$ and $C'_{\varphi(T)}$ are dual cones.
\end{proposition}
\begin{proof}
We have 
$$
(h_b - h_a)(e_{d}-e_{c}) = \begin{cases} 0 & \mbox{if $\overline{ab}$ and $\overline{d'a'}$ are noncrossing,} \\
1 & \mbox{if $a,d',b,c'$ are in clockwise order,} \\
-1 & \mbox{if $a,d',b,c'$ are in anti-clockwise order.}
\end{cases}
$$
The result then follows from the definition of $\varphi$.
\end{proof}

\begin{example}
A noncrossing alternating tree $T \in \Talt$ and the corresponding noncrossing, circular-alternating tree $\T'=\varphi(T)$ is given in Figure~\ref{fig:dualtrees}.  One can check that the cones $C_T = {\rm span}_{\geq 0}(h_3-h_4, h_2-h_4,h_2-h_6,h_2-h_1,h_5-h_6)$ and $C'_{T'}={\rm span}_{\geq 0}(e_3-e_4, e_5-e_3,e_1-e_5,e_2-e_1,e_5-e_6)$ are dual, in agreement with Proposition~\ref{prop:dualcones}.
\end{example}

Let $T$ be a directed tree on $[n]$.  An edge is {\it directed away from $n$} if it is directed away from $n$ as part of some path connected to $n$.  Define two statistics on $T$ by
\begin{align*}
\up(T) &= \#\{ \text{edges directed away from $n$}\} \\
\des(T) &= \#\{ \text{edges $i \to j$ with $i > j$}\}.
\end{align*}
The edges counted by $\des(T)$ are called descent edges.

\begin{proposition}\label{prop:updes}
Let $T \in \Talt$ and $T' = \varphi(T) \in \Tcir$.
Suppose that $e \in T$ is the unique edge intersecting $e' \in
T' $.  Then $e$ is directed away from $n$ in $T$ if and only if
$e'$ is a descent edge in $T'$.  In particular, we have $\up(T) = \des(T')$.
\end{proposition}

Let $P$ be a generic simple polypositroid.  Let $\T(P) = \{T_v\}$ be the collection of its vertex trees: by Lemma~\ref{lem:Talternating}, $\T(P) \subset \Talt$. 

We now describe the
1-skeleton of $P$ in terms of the trees $T_v$.  Suppose $E = (v,v')$ is an edge of $P$.  The forest $T_E$ has two
components and  one has $T_v = T_E \cup \{e\}$ and $T_{v'} = T_E
\cup \{e'\}$ for distinct directed edges $e,e'$.  The graph $T_E
\cup \{e,e'\}$ has a unique (non-directed) cycle containing both $e$
and $e'$.

\begin{lemma}\label{lem:ee'}
The edges $e$ and $e'$ have the same direction along this cycle.
\end{lemma}
\begin{proof}
The edges $e$ and $e'$ correspond to facets intersecting the edge
$E$, and the direction corresponds to a choice of one infinite
direction along the affine span of $E$, which we assume to be
parallel to $e_i - e_j$.  The direction is determined by which of
the two components $T_1$ and $T_2$ of $T_E$ the source (and hence
sink) of $e$ (resp. $e'$) lies in.  Indeed, if $e$ goes from $T_1$
to $T_2$, and $i \in T_1$ while $j \in T_2$ then the facet
corresponding to $e$ bounds the coordinate $x_j$ above.  It follows
that the ray in the direction of $v'$ emitting from $v$ goes in the
direction $\R_{\geq 0}(e_i - e_j)$.

But the two infinite directions corresponding to $e$ and $e'$ are
opposite, so $e$ must go from $T_1$ to $T_2$ (without loss of
generality), and $e'$ must go from $T_2$ to $T_1$.  But this implies
that $e$ and $e'$ have the same direction along the cycle containing
them both.
\end{proof}

For a vertex $v$ of a generalized permutohedron $P$, define the tree $T'_v$ as follows (cf. \cite{PRW}): $T'_v$ has
a directed edge $j \to i$ if there is an edge incident with $v$
which goes in the direction $\R_{\geq 0}(e_i - e_j)$.  When $P$ is a
polypositroid, we thus have two directed trees $T_v$ and $T'_v$ on
$[n]$ for each vertex $v \in P$.

\begin{theorem}\label{thm:dualtrees}
Let $P$ be a generic simple polypositroid.  Then for each vertex $v$ of $P$, we have $T'_v = \varphi(T_v)$.
\end{theorem}
\begin{proof}
Suppose $e = (a \to b)$ is a directed edge of $T_v$, such that $T_v
\backslash\{e\}$ has two components $T_1 \ni a$ and $T_2 \ni b$. Let
$c$ be the cyclic minimum of $T_1$ and $d$ the cyclic minimum of
$T_2$.  It follows from the discussion in the proof of Lemma
\ref{lem:ee'} that $T'_v$ has a directed edge from $c'$ to $d'$. One
can check that $c' \to d'$ intersects only $a \to b$, and that the
four vertices are in the stated order.
\end{proof}

\psset{unit=0.5mm}

\begin{figure}
\begin{tikzpicture}[scale=0.5]
\tikzset{>=latex}
\draw (0,0) circle (5);

\node at (0:5.4) {$1$};
\node at (-60:5.4){$2$};
\node at (-120:5.4){$3$};
\node at (-180:5.4){$4$};
\node at (-240:5.4){$5$};
\node at (-300:5.4){$6$};

\filldraw[black] (0:5) circle (3pt);
\filldraw[black] (-60:5) circle (3pt);
\filldraw[black] (-120:5) circle (3pt);
\filldraw[black] (-180:5)circle (3pt);
\filldraw[black] (-240:5)circle (3pt);
\filldraw[black] (-300:5)circle (3pt);

\draw[->,thick](0:5)--(-60:5);
\draw[->,thick](-300:5)--(-60:5);
\draw[->,thick](-300:5)--(-240:5);
\draw[->,thick](-180:5)--(-60:5);
\draw[->,thick](-180:5)--(-120:5);

\node at (-330:5.4){$1'$};
\node at (-30:5.4){$2'$};
\node at (-90:5.4){$3'$};
\node at (-150:5.4){$4'$};
\node at (-210:5.4){$5'$};
\node at (-270:5.4){$6'$};

\filldraw[fill=white] (-30:5) circle (3pt);
\filldraw[fill=white] (-90:5) circle (3pt);
\filldraw[fill=white] (-150:5) circle (3pt);
\filldraw[fill=white] (-210:5)circle (3pt);
\filldraw[fill=white] (-270:5)circle (3pt);
\filldraw[fill=white] (-330:5)circle (3pt);

\draw[->,dashed,thick](-330:5)--(-30:5);
\draw[->,dashed,thick](-270:5)--(-210:5);
\draw[->,dashed,thick](-210:5)--(-330:5);
\draw[->,dashed,thick](-90:5)--(-210:5);
\draw[->,dashed,thick](-140:5)--(-90:5);

\end{tikzpicture}
\caption{The bijection of Proposition~\ref{prop:Tbij}.}
\label{fig:dualtrees}
\end{figure}
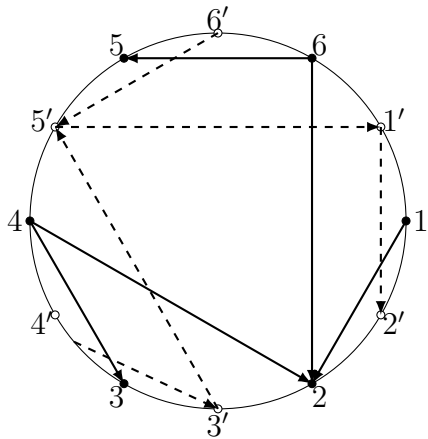

Let $h_P(t)$ denote the $h$-polynomial of a simple $d$-dimensional polytope $P$.  It is given by the equality $\sum_{i=0}^d f_i(P) t^i = h_P(t+1)$.
\begin{corollary}
Suppose $P$ is a generic simple polypositroid.  The $h$-polynomial
of $P$ is
$$
h_P(t) = \sum_{T'_v \in \T(P)} t^{\des(T_v)} = \sum_{T_v \in \T(P)} t^{\up(T_v)}.
$$
\end{corollary}
\begin{proof}
The first equality is shown in \cite{PRW}, and the second follows from Proposition~\ref{prop:updes}, or it can be proved in the same way as the first.
\end{proof}

\begin{remark} \label{rem:bijectiontree}
The bijection of Theorem \ref{thm:dualtrees} gives a bijection between noncrossing, decreasing, alternating trees on $[n]$, and rooted plane binary 
trees on $[n]$ equipped with the depth-first search labeling.
\end{remark}

\subsection{Coarsenings of braid arrangements} \label{ssec:braid}
Recall that the {\it braid arrangement} $\B_n \subset (\R^n)^*/h_n$ is the central arrangement which is the union of all hyperplanes of the form $y_i - y_j = 0$, $i \neq j \in [n]$, where $y = (y_1,y_2,\ldots,y_n)$ is identified with the linear function $y_1x_1 + y_2x_2 + \cdots + y_n x_n \in (\R^n)^*$.
We shall consider the hyperplane arrangement $\B_n$ as a complete fan, the {\it braid fan}.  The maximal cones of $\B_n$ are indexed by $w \in S_n$: 
$$
C_w = \{y = (y_1,y_2,\ldots,y_n) \in (\R^n)^* \mid y_{w(1)} \leq y_{w(2)} \leq \cdots \leq y_{w(n)}\}.
$$
and the rays of $\B_n$ are the $h_S= \sum_{i \in S} x_i \in (\R^n)^*/h_n$ for $S \in 2^{[n]}-\{\emptyset,[n]\}$.

The normal fan to the permutohedron $P_n \subset H$ is the braid fan $\B_n$.  More generally, any generalized permutohedron $P$ that is sufficiently generic has $\B_n$ as its normal fan.

Now, let $P$ be a generic simple polypositroid and $\F = \F(P) \subset  (\R^n)^*/h_n$ be its normal fan.  The rays of $\F$ are the $h_{[r,s]}$, $[r,s] \in \cyc$.  While there are many possibilities for $\F$, as we saw in Theorem~\ref{thm:fvector}, all such $\F$ have the same $f$-vector.  By definition, $P$ is a generalized permutohedron, so $\F$ is a coarsening of the fan $\B_n$.  In particular, each maximal cone $C_T$, $T \in \T(P)$ in $\F$ is a union of a number of the cones $C_w$.  

\begin{proposition}\label{prop:braidcoarse}
Let $T \in \Talt$ and $w \in S_n$.  We have $C_w \subset C_T$ if and only if for each edge $j \to i$ of $T' = \varphi(T)$, we have $w^{-1}(i) > w^{-1}(j)$.
\end{proposition}
\begin{proof}
By Proposition~\ref{prop:dualcones}, the inclusion $C_w \subset C_T$ is equivalent to the condition that for all $y \in C_w$ and $x \in C'_{T'}$ we have $y(x) \geq 0$.  This is equivalent to the condition that $y(e_i - e_j) = y_i - y_j \geq 0$ for edges $j\to i$ of $T'$.
\end{proof}

\begin{corollary}\label{cor:Tw}
Let $T \in \Talt$.  We have $C_T = C_w$ for some $w \in S_n$ if and only if $\varphi(T) \in \Tcir$ is a path.  Furthermore, for such $T$, we have $T \in \T(P)$ for any generic simple polypositroid.  
\end{corollary}
\begin{proof}
Let $T' \in\Tcir$.  If the underlying graph of $T'$ is a path, then $T'$ itself is a directed path.  In such a case, the condition of Proposition~\ref{prop:braidcoarse} uniquely determines $w \in S_n$, and conversely $w$ being uniquely determined implies that $T'$ is a directed path.  The last sentence follows from Theorem~\ref{thm:ensemble}: the sources (resp. the sinks) of $T$ form cyclic intervals when $T'$ is a path.
\end{proof}

\begin{example}
Let $n =3$.  In this case, the normal fan of a generic simple polypositroid is the braid arrangement.  There are six trees in $\Talt$:
\begin{center}
\tikzset{>=latex}
\begin{tikzpicture}[scale=0.9]
\draw[->,thick] (0,0) to [out=45,in=135] (1,0);
\draw[->,thick] (2,0) to [out=135,in=45] (1,0);
\begin{scope}[shift={(3,0)}]
\draw[->,thick] (1,0) to [out=45,in=135] (2,0);
\draw[->,thick] (1,0) to [out=135,in=45] (0,0);
\end{scope}
\begin{scope}[shift={(6,0)}]
\draw[->,thick] (0,0) to [out=45,in=135] (2,0);
\draw[->,thick] (1,0) to [out=45,in=135] (2,0);
\end{scope}
\begin{scope}[shift={(9,0)}]
\draw[->,thick] (2,0) to [out=135,in=45] (0,0);
\draw[->,thick] (2,0) to [out=135,in=45] (1,0);
\end{scope}
\begin{scope}[shift={(12,0)}]
\draw[->,thick] (2,0) to [out=135,in=45] (0,0);
\draw[->,thick] (1,0) to [out=135,in=45] (0,0);
\end{scope}
\begin{scope}[shift={(15,0)}]
\draw[->,thick] (0,0) to [out=45,in=135] (2,0);
\draw[->,thick] (0,0) to [out=45,in=135] (1,0);
\end{scope}
\end{tikzpicture}
\end{center}
and six in $\Tcir$:
\begin{center}
\tikzset{>=latex}
\begin{tikzpicture}[scale=0.9]
\draw[->,thick] (0,0) to [out=45,in=135] (1,0);
\draw[->,thick] (1,0) to [out=45,in=135] (2,0);
\begin{scope}[shift={(3,0)}]
\draw[->,thick] (2,0) to [out=135,in=45] (1,0);
\draw[->,thick] (1,0) to [out=135,in=45] (0,0);
\end{scope}
\begin{scope}[shift={(6,0)}]
\draw[->,thick] (0,0) to [out=45,in=135] (2,0);
\draw[->,thick] (2,0) to [out=135,in=45] (1,0);
\end{scope}
\begin{scope}[shift={(9,0)}]
\draw[->,thick] (2,0) to [out=135,in=45] (0,0);
\draw[->,thick] (1,0) to [out=45,in=135] (2,0);
\end{scope}
\begin{scope}[shift={(12,0)}]
\draw[->,thick] (2,0) to [out=135,in=45] (0,0);
\draw[->,thick] (0,0) to [out=45,in=135] (1,0);
\end{scope}
\begin{scope}[shift={(15,0)}]
\draw[->,thick] (0,0) to [out=45,in=135] (2,0);
\draw[->,thick] (1,0) to [out=135,in=45] (0,0);
\end{scope}
\end{tikzpicture}
\end{center}
Since every $T' \in \Tcir$ has underlying graph a path, there is a unique $w = w_{T'} \in S_n$ satisfying the condition of Proposition~\ref{prop:braidcoarse}.  This gives a bijection between $\Talt$ and $S_n$, identifying $C_T$, $T \in\Talt$ and $C_w$, $w \in S_n$.
\end{example}

\begin{example}
Let $n = 4$.   We have $|\Talt|=24$, consisting of 8 trees that are stars and 16 trees that are paths.
For a generic simple polypositroid $P$, we have $|\T(P)| = 20$ by Theorem~\ref{thm:fvector}, since the cyclohedron has 20 vertices.
There are 16 trees in $T \in \Talt$ such that $\varphi(T)$ is a directed path.  By Corollary~\ref{cor:Tw}, these trees belong to $\T(P)$, for any $P$.  

There are 8 trees in $\Talt$ such that $\varphi(T)$ is a (circular-alternating) star.  These 8 trees are the four cyclic rotations of the following two trees
$$
 (2 \to 1 \leftarrow 4 \to3)\qquad \text{and} \qquad   (3 \to 4 \leftarrow 1 \to 2).
$$
For each of these trees, $C_T$ is a union of two of the cones $C_w$.  For example, take $T = (2 \to 1 \leftarrow 4 \to3)$ with dual tree $T' = (2,4 \to 3 \to 1)$.  According to Proposition~\ref{prop:braidcoarse}, we have $C_{2 \to 1 \leftarrow 4 \to3} = C_{2431} \cup C_{4231}$.  Similarly, we obtain
$$
C_{4\to 3 \leftarrow 2 \to 1} = C_{4213} \cup C_{2413}, \qquad C_{4\to 1 \leftarrow 2 \to 3} = C_{2413} \cup C_{2431}, \qquad C_{2\to 3 \leftarrow 4 \to 1} = C_{4213} \cup C_{4231}.
$$
Thus we have 
$$
C_{2 \to 1 \leftarrow 4 \to3}  \cup C_{4\to 3 \leftarrow 2 \to 1}  = C_{4\to 1 \leftarrow 2 \to 3} \cup  C_{2\to 3 \leftarrow 4 \to 1}.
$$
Each $\T(P)$ contains either both $(2 \to 1 \leftarrow 4 \to3)$ and $(4\to 3 \leftarrow 2 \to 1)$ or both $(4\to 1 \leftarrow 2 \to 3)$ and $(2\to 3 \leftarrow 4 \to 1)$.  Switching between these two choices corresponds to switching the matching $M_{\{2.4\},\{1,3\}}$ in $\E(P)$.
\end{example}
\section{Integer points in polypositroids}
We assume in this section that $H = \{x \in \R^n \mid x_1+x_2+ \cdots + x_n = k\}$, where $k$ is an integer.  A polytope $P \subset H$ is an {\it integer polytope} if its vertices have integer coordinates.  By translating $P$ and $H$,
we may and will assume that $S := P \cap \Z^n \subset \N^n$, so that to each integer point $p = (p_1,p_2,\ldots,p_n) \in P$ one may 
associate a multiset $I_p$ of size $k$ which contains $p_1$ 1's, $p_2$ 2's, and so on.  Thus if $P$ is the matroid polytope of a matroid $M$,
then the multisets $I_p$ are honest sets, equal to the bases of $M$.  

If $I = \{i_1\leq i_2 \leq \cdots \leq i_k\}$ and $J = \{j_1 \leq j_2 \leq \cdots \leq j_k\}$ are two multisets consisting of elements in $\{1,2,\ldots,n\}$, we define two multisets $\sort_1(I,J)$ and $\sort_2(I,J)$ of the same size as follows.  Let $I \cup J = \{a_1 \leq a_2 \leq \cdots \leq a_{2k}\}$.  Then $\sort_1(I,J) = \{a_1,a_3,\ldots,a_{2k-1}\}$ and $\sort_2(I,J) = \{a_2,a_4,\ldots,a_{2k}\}$.  The following characterization of integer alcoved polytopes is given in \cite{LP}.

\begin{theorem}[{\cite[Theorem 3.1]{LP}}] \label{thm:sorted}
Suppose $P \subset H$ is an integer polytope such that $S:= P \cap \Z^n \subset \N^n$.  Then
$P$ is an alcoved polytope if and only if for any $p,p' \in S$, there exist $q,q' \in S$ so that $I_q = \sort_1(I_p,I_{p'})$ and 
$I_{q'}=\sort_2(I_p,I_{p'})$.
\end{theorem}
If a collection $S$ of nonnegative integer points satisfies the condition in Theorem \ref{thm:sorted}, then we call $S$ {\it sort-closed}.

Murota \cite{Mur} studies certain collections of lattice points called {\it $M$-convex sets}, which are essentially equivalent to the {\it discrete polymatroids} of Herzog and Hibi \cite{HH}.  We use the terminology of the latter.  A {\it base polymatroid} is a generalized permutohedron $P$ such that all the values $f_P(S)$ of the support function are nonnegative; see for example \cite[Section 3]{CL}.  Any generalized permutohedron can be translated so that the nonnegativity condition holds.  A {\it discrete (base) polymatroid} is a collection of multisubsets of $[n]= \{1,2,\ldots,n\}$ satisfying an exchange criterion.  The exchange criterion can be formulated in the language of generalized permutohedra as follows.

\begin{theorem}[{\cite[Theorem 2.3]{HH}}]\label{thm:HH}
Suppose $P \subset H$ is an integer polytope.  Let $S:= P \cap \Z^n$.  Then $P$ is a generalized
permutohedron if and only if for any $p, q \in S$ we have:
\begin{quote}
whenever $p_i > q_i$ we can find $j$ so that $p_j < q_j$ and $p_i -e_i+e_j \in S$.
\end{quote}
\end{theorem}
If a collection $S$ of integer points satisfies the condition in Theorem \ref{thm:HH} we say $S$ satisfies the Exchange Lemma.  Combining Theorem~\ref{thm:sorted} and Theorem~\ref{thm:HH}, we obtain the following charaterization of integer polypositroids:
\begin{theorem}\label{thm:polyinteger}
Suppose $P \subset H$ is an integer polytope such that $S:= P \cap \Z^n \subset \N^n$.  
Then $P$ is an integer polypositroid if and only if
$S$ is sort-closed, and $S$ satisfies the Exchange Lemma.
\end{theorem}

In the case that $P$ consists of $0$-$1$ vectors, Theorem~\ref{thm:polyinteger} characterizes positroids as those collections $M \subset \binom{[n]}{k}$ of $k$-element subsets that are both sort-closed and satisfies the Exchange Lemma.

\begin{corollary}\label{cor:sortclosed}  Positroids are exactly the sort-closed matroids.
\end{corollary}

\begin{remark}
Corollary~\ref{cor:sortclosed} can also be deduced directly from the characterization of positroids as matroids associated to points $X \in \Gr(k,n)_{\geq 0}$ in the totally nonnegative Grassmannian.  Namely, the Pl\"ucker coordinates $\Delta_I(X)$ of such a point satisfy inequalities that give a sort-closed matroid; see \cite[Proposition 8.7]{Lam}.
\end{remark}

\part{Coxeter polypositroids}
Generalized permutohedra are defined by specifying the possible directions of edges.
Alcoved polytopes are defined by specifying the possible directions of normal vectors to facets.  The set of allowed edge directions and the set of allowed facet normal directions is related by the linear transformation $e_i \mapsto h_i$.  We give this linear transformation a root-system theoretic interpretation, and develop the theory of Coxeter polypositroids.

\section{Coxeter elements}

\subsection{Root systems}
\label{sec:root_systems}
First, we recall some terminology and a few well known facts related to root systems and
Weyl groups; see \cite{Bou, Hum} for more details.

Let $V\simeq\R^r$ be a vector space of dimension $r \geq 2$ equipped with a symmetric
positive definite bilinear form $(x,y)$.  Let $R\subset V$ be an
irreducible and reduced crystallographic {\it root system} of rank $r$.  For a root $\alpha\in R$, the
corresponding {\it coroot} is $\alpha^\vee = 2\alpha/(\alpha,\alpha)$, and the
{\it reflection\/} $s_\alpha\in GL(V)$ with respect to $\alpha$ is given by 
$$
s_\alpha:x\mapsto x - (\alpha^\vee,x)\,\alpha, \quad\textrm{for } x\in V.  
$$
The {\it Weyl group\/} $W\subset GL(V)$ is the group generated by the reflections 
$s_\alpha$, $\alpha\in R$.

Let us fix a choice of {\it positive roots\/} $R^+\subset R$ and the
corresponding choice of {\it simple roots\/} $\alpha_1,\dots,\alpha_r$ in $R$ and simple coroots $\alpha_1^\vee,\ldots,\alpha_r^\vee$.  The {\it Cartan matrix} $A = (A_{ij})$ is given by 
\begin{equation}\label{eq:Cartan}
A_{ij} = (\alpha_i^\vee,\alpha_j) \in \Z.
\end{equation}
Let $s_i = s_{\alpha_i}$ be the {\it simple reflections.} 
It is well known that all possible choices of positive roots are conjugate to 
each other by the action of the Weyl group $W$.

Let $\omega_1,\dots,\omega_r\in V$ be the basis 
of $V$ dual to the basis of simple coroots $\alpha_1^\vee,\dots, \alpha_r^\vee$, 
that is, $(\alpha_j^\vee,\omega_i)=\delta_{ij}$ for any $i,j\in\{1,\dots,r\}$.  The vectors $\omega_1,\dots,\omega_r$ are called {\it fundamental weights\/}.  Let $\Lambda \subset V$ denote the weight lattice spanned by $\omega_1,\dots,\omega_r$.

\begin{remark}
Many of our results hold even for non-crystallographic root systems.  However, for the connections to cluster algebras in Section~\ref{sec:cluster}, we must use a crystallographic root system.
\end{remark}

\subsection{Coxeter elements}
A {\it standard Coxeter element\/} $c=s_{i_1} s_{i_2}\cdots s_{i_r} \in W$ is the product
of the simple reflections $s_1,\dots,s_r$ written in some order $s_{i_1},\dots,s_{i_r}$.
More generally, a {\it Coxeter element\/} $c'= s_{i_1}'s_{i_2}'\cdots s_{i_r}'\in W$
is a similar product for some (possibly different) choice of simple reflections
$s_1',\dots,s_r'$.
In other words, Coxeter elements are Weyl group conjugates $c'=wcw^{-1}$, $w\in W$, of standard Coxeter elements $c$.
Moreover, any two Coxeter elements are conjugates of each other.
Thus all Coxeter elements have the same order, called the {\it Coxeter number\/} $h$.

The eigenvalues of a Coxeter element are 
$e^{2\pi\sqrt{-1}\, m_i /h}$, where $m_1,\dots,m_r\in \{1,\dots,h-1\}$ 
are the {\it exponents\/} of the root system.  In particular, $1$ 
is not an eigevalue of $c$.  This implies the following claim.

\begin{lemma} \label{lem:invertible}
For any Coxeter element $c$, the transformation $I-c$ is an invertible element of $GL(V)$, and the inverse is given by $ (I-c)^{-1} = -{1\over h} \sum_{j=1}^{h-1} j\, c^j$.
\end{lemma}
\begin{proof}
Since all eigenvalues of $c$ are $h$-th roots of unity, not including the identity, we have $I+c+c^2+ \cdots + c^{h-1} = 0$.  Thus
\begin{align*}
\begin{split}
(I-c) ( \sum_{j=1}^{h-1} j\, c^j) =(\sum_{j=1}^{h-1} c^j) - (h-1)I = -hI. 
\end{split}\qedhere
\end{align*}
%
\end{proof}

We say that a choice of positive roots $R^+$ is {\it compatible} with $c$ if $c$ is a standard Coxeter element with respect to $R^+$.  

Let $\Gamma = \{1,c,c^2,\ldots,c^{h-1}\} \subset W$ be the subgroup generated by $c$.
Given $R^+$ compatible with $c$, and a reduced factorization $c = s_1 s_2 \cdots s_r$, we obtain a total ordering of the root system $R$: set $\beta_1 =\alpha_1$, $\beta_2 = s_1 \alpha_2$, $\ldots$, $\beta_r = s_1 s_2 \cdots s_{r-1} \alpha_r$.  Then define $\beta_i$ for $i \in \Z$ recursively by $\beta_{i+r}:= c \beta_i$.  For each $i = 1,2,\ldots,r$, the roots $\beta_i, \beta_{i+r},\ldots,\beta_{i+(h-1)r}$ is a $\Gamma$-orbit in $R$.

\begin{proposition}\label{prop:listroots}\
\begin{enumerate}
\item
The set of $h \cdot r$ vectors $ \{\beta_i \mid  1 \leq i \leq hr \}$ is exactly the set of all roots in $R$ without repetitions.  In particular, $|R| = hr$.
\item
For each $i = 1,2,\ldots,r$, there exists a unique integer $M(i) \in [1,h-1]$ such that
$$
\beta_{i}, c\beta_i, \ldots, c^{M(i)-1}\beta_i \in R^+, \text{ and } c^{M(i)} \beta_{i}, c^{M(i)+1} \beta_i, \ldots, c^{h-1}\beta_i \in R^-.
$$
\end{enumerate}
\end{proposition}
\begin{proof}
%
%
Follows from \cite[Ch.\ VI, \S 1, n${}^\circ$ 11, Prop.~33]{Bou}.  
\end{proof}

The first part of the following result is \cite[Theorem 3.6]{KT}.  It is stated there for simply-laced Weyl groups, but holds in the multiply-laced types as well.
\begin{proposition} \label{prop:compatible}
Let $R_1^+$ and $R_2^+$ be two positive systems compatible with $c$.  Then $R_1^+$ and $R_2^+$ are related by a sequence of elementary transformations $R^+ \mapsto (R^+)'$  of the following form: suppose $s_i$ is a simple generator for $R^+$ and $c$ has a reduced factorization either starting or ending in $s_i$, then set $(R^+)' = s_i \cdot R^+$.

If $c = s_1 \cdots s_r$ and $(R')^+ = s_1 R^+$, then the total ordering of $R$ coming from $((R')^+,c)$ is the cyclic shift $(\beta_2,\beta_3,\ldots,\beta_{hr},\beta_1)$.
\end{proposition}

For the remainder of this section we fix a Coxeter element $c$ and a choice of positive roots $R^+$ compatible with $c$.  We extend the definition of the fundamental weights $\omega_i$ by defining $\omega_i$ for $i \in \Z$ recursively by $\omega_{i+r}:= c\omega_i$.

\begin{proposition} \label{prop:c_orbits}
We have $(I-c)\omega_i = \beta_i$ for all $i \in \Z$.  
\end{proposition}
\begin{proof}
According to the definitions, $s_j(\omega_i) = \omega_i - \delta_{ij}\,\alpha_j$.
Repeatedly applying the simple reflections $s_r, s_{r-1},\dots,s_1$ to 
$\omega_i$, we get $c(\omega_i) = s_1\cdots s_r (\omega_i) = \omega_i - s_1\cdots s_{i-1}(\alpha_i)$.  Thus
$$
(I-c)(\omega_i) = s_1 \cdots s_{i-1}(\alpha_i) = \beta_i,\quad
\textrm{for } i=1,\dots,r.
$$
The statement now follows from Proposition \ref{prop:listroots}.
\end{proof}

For convenience, when $\beta \in R$, we use the notation $\tbeta$ to denote $(I-c)^{-1} \beta \in \tR$.  Thus $\omega_i = \tbeta_i$ for $i = 1,2,\ldots,hr$.

Let $w_0 \in W$ be the longest element and let $i \mapsto i^\star$ denote the bijection on $\{1,2,\ldots,r\}$ determined by $w_0 \alpha_i = - \alpha_{i^\star}$.

\begin{lemma}\label{lem:bound}
Fix $k \in I$.  We have
\begin{enumerate}
\item
$(c^m \beta^\vee_i,\omega_k)\geq 0$, for $0\leq m<M(k^\star)$ and any 
$i=1,\dots,r$;
\item
$(c^m \beta^\vee_i,\omega_k) =  0$, if $M(i) < M(k^\star)$ and $M(i) \leq m<M(k^\star)$; 
\item
$(c^m \beta^\vee_i,\omega_k) =  0$, if $M(k^\star) \leq M(i)$ and $M(k^\star) \leq m < M(i)$; 
\item
$(c^m \beta^\vee_i,\omega_k)\leq 0$, for $M(k^\star)\leq m <h$ and any 
$i=1,\dots,r$.
\end{enumerate}
\end{lemma}

Define $i \prec_c j$ if $i$ and $j$ are connected in the Dynkin diagram, and $s_i$ precedes $s_j$ in $c$.  Orient the Coxeter diagram of $W$ so that an edge $(i,j)$ is oriented $j \to i$ if $i \prec_c j$.  We will use the following formulae from \cite{YZ}.  

\begin{lemma}\label{lem:YZ} We have
\begin{enumerate}
\item
$c^{M(i)}\beta_i = -\beta_{i^*}$;
\item
$M(i) + M(i^\star) = h$;
\item
If $j \to i$ then we have 
$$
M(i) - M(j) =\begin{cases} 1 & \mbox{if $i^\star \to j^\star$} \\
0 & \mbox{if $j^\star \to i^\star$}
\end{cases}.
$$
\end{enumerate}
\end{lemma}
%

\begin{proof}[Proof of Lemma \ref{lem:bound}]
Fix $i$ and $j$.  We prove $(c^m \beta^\vee_i,\omega_j) \geq 0$ for $0 \leq m < M(j^\star)$.  If $m < M(i)$, then by definition $c^m \beta_i \in R^+$, so $c^m \beta^\vee_i \in (R^\vee)^+$ and the inequality is clear.  Thus we may assume that $M(i) < M(j^\star)$ and $M(i) \leq m < M(j^\star)$.  In particular, we are assuming that $i \neq j$.

Define the support $S(\beta^\vee) \subseteq I$ of a coroot $\beta^\vee \in R^\vee$ to be the (positive or negative) simple coroots that occur in the expansion of $\beta^\vee$ into simple coroots.  For a nonnegative integer $a$, let $P_a(i) \subseteq I$ be the set of vertices $j \in I$ that can be reached from $i \in I$ by a path where at most $a$ edges are in the wrong direction.  One can check that
$$
S(\beta^\vee_k) \subset P_0(k)
$$
and by induction on $a$, we have for $a \geq 0$,
$$
S(c^a \beta^\vee_k) \subset P_a(k).
$$
By Lemma \ref{lem:YZ}(1), it follows that we have 
$$
S(c^m\beta^\vee_i) \subset P_{m-M(i)}(i^\star)
$$
for $m \geq M(i)$.
But by Lemma \ref{lem:YZ}(3), $M(j^\star) - M(i)$ is bounded above by the number of edges directed in the wrong direction on the path from $i^\star$ to $j$.  Thus if $m < M(j^\star)$, we have $j \notin P_{m-M(i)}(i^\star)$ and $(c^m \beta^\vee_i,\omega_j) = 0 \geq 0$.  This proves statements (1) and (2).  Statements (3) and (4) are similar. 
\end{proof}

\begin{example}\label{ex:Aroot}
We consider the root system of type $A_{n-1}$.  Let $V = H_0 = \{x \mid x_1+x_2+\cdots+x_n=0\} \subset \R^n$ and $R = \{e_i - e_j \mid i \neq j\}$.  Then $W$ is the symmetric group $S_n$.  We take as positive simple roots $\alpha_i = e_{i+1}-e_i$.  Then the linear functional $(\cdot, \omega_k): V \to \R$ can be identified with the function $h_n - h_k \in (\R^n)^*$.

Now choose the Coxeter element $c = s_1 s_2 \cdots s_{n-1}$, which coincides with our choice in Section~\ref{sec:alcenv}.  The Dynkin diagram is oriented as follows:
$$
1 \leftarrow 2 \leftarrow \cdots \leftarrow (n-1).
$$
We have $\beta_i = s_1 \cdots s_{i-1} (e_{i+1}-e_r) = e_{i+1}-e_1$ for $i=1,2,\ldots,n-1$.  We have $c(e_i) = e_{i+1}$, where $e_{n+1}:=e_1$.  We compute that $i^\star = n-i = M(i)$.  It is straightforward to verify Lemma~\ref{lem:bound} directly, noting also that for $R= A_{n-1}$ we have $\beta = \beta^\vee$.
\end{example}

\section{Generalized $W$-permutohedra}

A {\it $W$-permutohedron\/} $P$ is a convex polytope in the space $V$ 
which is the convex hull of an orbit $W(x)$ of the Weyl group $W$, for some $x\in V$ not lying in any of the hyperplanes $H_\alpha := \{x \in V \mid (x,\alpha) = 0\}$, $\alpha \in R$. 
One key property of $W$-permutohedra is that every edge of $P$ is parallel to
some coroot $\alpha^\vee\in R^\vee$.

\begin{definition}   A {\it generalized $W$-permutohedron\/} $P$ is a convex 
polytope in the space $V$ such that every edge $[u,v]$ of $P$ is parallel to 
a coroot $\alpha^\vee\in R^\vee$, i.e., $u-v= a \alpha^\vee$, for some $a\in \R_{> 0}$.
\end{definition}

Note that the notion of a generalized $W$-permutohedron is unchanged when we replace the root system $R$ by the dual root system $R^\vee$.  Furthermore, the class of generalized $W$-permutohedra is preserved by the action of $W$.

The normal fan to a $W$-permutohedron is the $W$-Coxeter fan, the fan associated to the hyperplane arrangement consisting of all hyperplanes $H_\alpha$ as $\alpha \in R$ varies.  The maximal cones of the $W$-Coxeter fan are indexed by $w \in W$.  The one-dimensional cones, or rays, of the $W$-Coxeter fan are generated by the set $W \cdot \{\omega_1,\ldots,\omega_r\}$ of vectors lying in the $W$-orbit of a fundamental weight.  Generalized $W$-permutohedra can be equivalently defined as convex polytopes whose normal fan refines the $W$-Coxeter fan (see for example \cite[Theorem~15.3]{PRW}).  In particular, any facet of a generalized $W$-permutohedron has a normal vector that lies in the $W$-orbit of a fundamental weight.  Thus a generalized $W$-permutohedron is given by a collection of inequalities
\begin{equation}\label{eq:genpermineq}
(x, \omega) \leq a_\omega, \qquad \omega \in W \cdot \{\omega_1,\ldots,\omega_r\},
\end{equation}
where we always assume that $a_\omega$ has been taken minimal.  Let $\CsubW = \{(a_\omega)\}$ denote the cone of vectors $(a_\omega)$ arising from generalized $W$-permutohedra.  This cone is studied in \cite{ACEP}.

Let the {\it dominant cone\/}
\begin{equation}
\label{eq:C}
C=\R_{\geq 0} \<\alpha^\vee_1,\dots,\alpha^\vee_r\> \subset V
\end{equation}
be the cone of nonnegative linear combinations of the simple (equivalently,
positive) roots in $R$.  For example, with $R = A_{n-1}$ and the conventions of Example~\ref{ex:Aroot}, this agrees with the cone \eqref{eq:typeAcone}, intersected with $H_0 \simeq V$.
The cone $C$ is the dual of the {\it dominant Weyl chamber}
$D=\R_{\geq 0}\<\omega_1,\dots,\omega_r\>$.
Namely, $D=\{y\in V\mid (x,y)\geq 0\textrm{ for any } x \in C\}$
and $C=\{x\in V\mid (x,y)\geq 0\textrm{ for any } y \in D\}$.  The following result follows from the statement that the normal fan of a generalized $W$-permutohedra is a refinement of the $W$-Coxeter fan.

\begin{theorem}   A polytope $P$ is a generalized $W$-permutohedron if and only if
it has the following form
$$
P = \bigcap_{w\in W} (v_w + w(C)),
$$
where $v_w\in V$, $w\in W$, is a collection of points in $V$ such that,
for any $u, w\in W$,
$$
v_u \in v_w + w(C),
\quad\textrm{or, equivalently,}\quad
v_u-v_w\in w(C)\cap(-u(C)).
$$
The points $v_w, w\in W$, are exactly all vertices of the polytope $P$
(possibly with repetitions).
\label{th:intersection_C}
\end{theorem}

Define 
the {\it dominance order $\leq_C$} as the partial order on points $V$ given by
$x\leq_C y$ if $y-x\in C$.  Similarly, for $w\in W$, the 
{\it $w$-dominance order\/} $\leq_{w(C)}$ is the partial order on
$V$ given by $x\leq_{w(C)} y$ if $y-x\in w(C)$, for $w\in W$.  The following
result easily follows from Theorem~\ref{th:intersection_C}.

\begin{corollary}\label{cor:genW}
A polytope $P$ be a generalized $W$-permutohedron if and only if
for each $w\in W$, $P$ has a unique minimum element $v_w\in P$ 
in the $w$-dominance order $\leq_{w(C)}$.
\end{corollary}
\cite[Theorem~15.3]{PRW} also implies that the condition on the points
$v_w$, $w\in W$ can be reformulated, as follows.
It is enough to require the condition
$v_u - v_w \in w(C)\cap (-u(C))$ only for the pairs $u,w\in W$ such that
$u= w\, s_i$, for a simple reflection $s_i$.
In this case, the cone $w(C)\cap (-w \, s_i (C))$ is 
the $1$-dimensional cone spanned by the coroot $w(\alpha^\vee_i)$:
$$
w(C)\cap (-w\,s_i(C)) = \R_{\geq 0} \<w(\alpha^\vee_i)\>.
$$

\begin{theorem}
A polytope $P$ is a generalized $W$-permutohedron if and only if
it has the following form
$$
P = \bigcap_{w\in W} (v_w + w(C)),
$$
where $v_w\in V$, $w\in W$, is a collection of points in $V$ such that,
for any $w\in W$ and $i=1,\dots,r$,
$$
v_{w}-v_{ws_i} = a\, w(\alpha^\vee_i)\quad\textrm{for }a\in\R_{\geq 0}.
$$
\label{th:intersection_C_simple}
\end{theorem}

\section{Twisted $(W,c)$-alcoved polytopes}
In \cite{LP2}, we studied the $W$-alcoved polytopes: polytopes in $V$ with the property that all facet normals belong to $R$.  Here, we introduce a twisted variant of $W$-alcoved polytopes that depends on the choice of a Coxeter element $c$.
\subsection{Coxeter twisted roots}
\begin{definition}
Define the {\it $c$-twisted root system} $\tilde R$ as
the image $\tilde R = (I-c)^{-1}(R)$ of the root system $R$ 
under the transformation $(I-c)^{-1}$.
We call the elements of $\tilde R$ the {\it $c$-twisted roots}, or simply {\it twisted roots} if the Coxeter element is understood.
\end{definition}
We have that $\tilde R = -\tilde R$.  Note also that $\tilde R$ does not depend on the choice of positive system $R^+$.  

According to Lemma~\ref{prop:c_orbits}, the twisted root system 
$\tilde R$ is exactly the set of weights that lie in the $\Gamma$-orbits
of the vectors $\omega_1,\omega_2,\ldots,\omega_r$:
\begin{equation}\label{eq:tR}
\tilde R=
(I-c)^{-1}(R)= \{c^t(\omega_i)\mid t=0,\dots,h-1; \ i=1,\ldots,r\} = \{\omega_i \mid i = 1,\ldots,hr\}. 
\end{equation}
Note that $\tilde R = \Gamma \cdot \{\omega_1,\ldots,\omega_r\} \subset W \cdot \{\omega_1,\ldots,\omega_r\}$.

\begin{definition}
A $(W,c)$-twisted alcoved polytope is a polytope $P \subset V$ whose facets are normal to twisted roots.  More precisely, $P$ is a nonempty set with the presentation
$$
P = P(a_\omega) =  \{x \in V \mid (x,\omega) \leq a_\omega \mbox{ for $\omega \in \tilde R$}\}.
$$
Here, $a_\omega$ are arbitrary real numbers which we always assume to be chosen minimal.
\end{definition}

\begin{example}\label{ex:Aroot2}
We continue Example~\ref{ex:Aroot}.  Applying Lemma~\ref{lem:invertible}, the twisted roots are given by 
$$
\tR= \left\{-\frac{1}{n}\left(\sum_{k=1}^{n-1} k (e_{i+k}-e_{j+k}) \right) \mid i \neq j \right\} \subset V.
$$
For example, $(I-c)^{-1}\alpha_1 = \frac{1}{5}(-4,1,1,1,1)$.  In $(\R^n)^*/h_n$, this is equal to $h_n-h_1 =(0,1,1,1,1)$.  Identifying $\tR$ with a subset of $(\R^n)^*/h_n$, we get $\tR = \{h_i - h_j \mid i \neq j\}$.  The notion of $(W,c)$-twisted alcoved polytope here agrees with our notion of alcoved polytope in Definition~\ref{def:alcoved} .  
\end{example}

If $W$ and $c$ are understood to be fixed, we may simply use the name ``twisted alcoved polytope''.

\begin{remark}\label{rem:conj}
Suppose that $c$ and $c'$ are two Coxeter elements.  Then there exists $w \in W$ so that $c' = wcw^{-1}$.  The $c$-twisted roots $\tilde R_c$ and the $c'$-twisted roots $\tilde R_{c'}$ are related by $\tilde R_{c'} = w \cdot \tilde R_c$.  We have that $P \subset V$ is a $(W,c')$-twisted alcoved polytope if and only if $w \cdot P \subset V$ is a $(W,c)$-twisted alcoved polytope.
\end{remark}

\begin{definition}
For a compact subset $Q \subset V$ the {\it $(W,c)$-twisted alcoved envelope} $\env(Q)$ is the smallest $(W,c)$-twisted alcoved polytope containing $Q$.
\end{definition}

Clearly, $\env(P) = P$ if and only if $P$ is a $(W,c)$-twisted alcoved polytope. 
Note that the intersection of generalized $W$-permutohedra may not be a generalized $W$-permutohedron.  Thus the ``generalized $W$-permutohedron envelope'' is not a well-defined operation.

Recall that the dominant cone $C$ was defined in \eqref{eq:C}. 
Let $Q \subset V$ be a compact subset.  For $i = 0,1,\ldots,h-1$, let $v_i \in V$ be  the minimum point in dominance order $\leq_{c^i(C)}$ such that $Q \subset v_i + c^i(C)$.  Thus, every facet of $v_i+c^i(C)$ touches $Q$. 

\begin{proposition}\label{prop:envelope}
The $(W,c)$-twisted alcoved envelope of $Q$ is given by the following intersection:
$$
\env(Q) = \env_c(Q)= \bigcap_{i=0}^{h-1}( v_i + c^i(C)).
$$
\end{proposition}
\begin{proof}
Since $Q$ is compact, for each $\omega \in \tilde R$, there is a minimal number $a_\omega$ such that $Q$ belongs to the half space $H_{\omega}^+ := \{x \in V \mid (x,\omega) \leq a_\omega\}$.  Then $\env(Q) = \bigcap_{\omega \in \tilde R} H_{\omega}^+$.
The convex cone $C$ has inward pointing normals given by $\omega_1,\omega_2,\ldots,\omega_r$.  Thus $v_0 + C = \bigcap_{j=1}^r H_{-\omega_j}^+$ and similarly we have $v_i + c^i(C) = \bigcap_{j=1}^r H_{-c^i\omega_j}^+$.  The claim follows from \eqref{eq:tR} and the fact that $\tR = -\tR$.
\end{proof}
\subsection{Faces of twisted alcoved polytopes} \label{ssec:facetwist}
We call a $(W,c)$-twisted alcoved polytope $P$ {\it generic} if it is full-dimensional and every twisted root $\omega \in \tilde R$ defines a facet $F = \{(x,\omega) = a_\omega\} \cap P$ of $P$.  

Now let $P$ be a $(W,c)$-twisted alcoved polytope and $F$ a face of $P$.  Let $S(F) \subset \tilde R$ be the set of twisted roots $\omega$ such that $F$ lies on the hyperplane $(x,\omega) = a_\omega$ and $P$ lies in the halfspace $(x,\omega) \leq a_\omega$.  The set $S(F)$ is a $(W,c)$-analogue of the graph $T_F$ of Section~\ref{ssec:facegraphs}.

Suppose that $\beta$ and $\beta'$ are two distinct roots.  We say that $\beta$ and $\beta'$ are {\it alternating} if $(\beta,\beta') \geq 0$.  Note that $\beta$ and $\beta'$ are alternating if and only if $-\beta$ and $-\beta'$ are.  Also note that this notion of alternating agrees with the notion used in the definition of alternating trees in Section~\ref{ssec:facegraphs}.  Suppose that $\omega, \omega' \in \tR$ are distinct twisted roots.  Then we say that $\omega$ and $\omega'$ are alternating if $(I-c)\omega$ and $(I-c)\omega'$ are.

\begin{lemma}\label{lem:Walternating}
Suppose that $P$ is a generic simple $(W,c)$-twisted alcoved polytope and $v$ a vertex of $P$.  Then $S(v) \subset \tR$ is a basis of $V$ and that consists of pairwise alternating twisted roots.
\end{lemma}
\begin{proof}
That $S(v)$ is a basis follows from the assumption that $P$ is generic and simple.  Suppose $\tbeta, \tbeta' \in S(v)$ are not alternating.  Then $(\beta ,\beta' ) < 0$ so there is a root of the form $\beta'' = b \beta + c \beta'$ with $b,c \in \R_{>0}$.  Thus there is a twisted root of the form $\tbeta'' = b\tbeta + c\tbeta'$.  Clearly, $a_{\tbeta''} = \max_{x \in P} (x,\tbeta'') \leq b a_{\tbeta} + c a_{\tbeta'}$.  We have
$$
a_{\tbeta"} \geq (v, \tbeta'') = (v,b\tbeta + c\tbeta') = b a_{\tbeta} + c a_{\tbeta'} \geq a_{\tbeta''},
$$
implying that $(v,\tbeta'') \in a_{\tbeta''}$ and thus $\tbeta'' \in S(v)$, contradicting the statement that $S(v)$ is a basis.
\end{proof}

\section{$(W,c)$-polypositroids}

\begin{definition}\label{def:Wpoly}
A convex polytope $P \subset V$ is called a {\it $(W,c)$-polypositroid\/} 
if it is both a generalized $W$-permutohedron and a $(W,c)$-twisted alcoved polytope.  In other words, we require that
\begin{enumerate}
\item[(1)] every edge $[u,v]$ of $P$ is parallel to a coroot $\alpha^\vee\in R^\vee$, i.e.,
$u-v= a \alpha^\vee$,  $a\in\R_{>0}$;
\item[(2)] every facet of $P$ is orthogonal to a twisted root 
$(I-c)^{-1}(\beta)\in\tilde R$, $\beta\in R$.
\end{enumerate}
\end{definition}

By Example~\ref{ex:Aroot2}, Definition~\ref{def:Wpoly} agrees with the notion of ``polypositroid in $H_0$" of Definition~\ref{def:poly} with the choices of Example~\ref{ex:Aroot}.  If $W$ and $c$ are understood to be fixed, we may simply use the name ``polypositroid''.  Let us hasten to point out that the notion of a $(W,c)$-polypositroid \emph{does not} depend on a choice of $R^+$.

\begin{remark}\label{rem:cc'}
Suppose $c$ and $c'$ are two Coxeter elements.  Then there exists $w \in W$ so that $c' = wcw^{-1}$.  We have that $P$ is a $(W,c')$-polypositroid if and only if $w \cdot P$ is a $(W,c)$-polypositroid.
\end{remark}

\begin{remark}
The notion of $(W,c)$-polypositroid is unchanged if the root and coroot vectors are dilated by scalars.  In particular, the notion of $(W,c)$-polypositroid is identical for a root system $R$ and its dual $R^\vee$.
\end{remark}

\section{Coxeter necklaces}
Consider the cone 
$$
A :=\R_{\geq 0}\<\beta^\vee_1,\beta^\vee_2,\ldots,\beta^\vee_r\>.
$$
Since for $i = 1,2,\ldots,r$ we have $\beta^\vee_i \in R^+$ and $c^{-1}(\beta^\vee_i) \in R^-$, we have $A\subseteq C\cap (-c(C))$.  But in general these two cones 
are not equal to each other.  Note also that $A$ depends only on $c$ and $R^+$, and not on the choice of reduced factorization of $c$ (though the enumeration of the set $\{\beta^\vee_1,\beta^\vee_2,\ldots,\beta^\vee_r\}$ does depend on the reduced factorization).

\begin{definition}\label{def:necklace}
A {\it $(W,R^+,c)$-Coxeter necklace} is a sequence $(v_0,v_1,\ldots,v_{h-1})$ of points in $V$, such that for any $i=1,\dots,h$,
$$
v_{i}-v_{i-1} \in c^i(A).
$$
Here, we set $v_{i+h} := v_i$ and thus $v_h:=v_0$,.
\end{definition}

With the choices made in Example~\ref{ex:Aroot}, Definition~\ref{def:necklace} agrees with the notion of Coxeter necklace in Definition~\ref{def:Anecklace}.

\begin{proposition}\label{prop:necklace}
Suppose that $\v = (v_0,v_1,\ldots,v_{h-1})$ is a $(W,R^+,c)$-Coxeter necklace.  Then each point $v_i$ is a vertex of the $(W,c)$-twisted alcoved polytope
$$
Q(\v) := \bigcap_{i=0}^{h-1}( v_i + c^i(C)).
$$
In particular, $\env_c(\v) = Q(\v)$.
\end{proposition}
\begin{proof}
It suffices to show that $v_i \in Q(\v)$, or equivalently, $v_j-v_i\in c^i(C)$, for any $i,j\in\{0,\dots,h-1\}$.  Rotating by a power of $c$, it is enough to show that $v_j - v_0\in C$.  Equivalently, we need to show that $(v_j-v_0,\omega_k)\geq 0$, for any $j$ and 
any $k$.

We have $v_j - v_0 = (v_1- v_0) + (v_2 - v_1) + \cdots + (v_j-v_{j-1})$.
The vector $v_{m+1}-v_m$ is a nonnegative
linear combination of the coroots $c^m\, \beta^\vee_i$, for 
$i=1,\dots,r$.
According to Lemma~\ref{lem:bound},
if $j\leq M(k^\star)$, 
then $(\beta^\vee,\omega_k)\geq 0$,
for all coroots $\beta^\vee$ involved in the expression of $v_{m+1}-v_m$,
for $m=j-1,j-2,\dots,0$.
This implies that $(v_j-v_0,\omega_k)\geq 0$.

On the other hand, if $j \geq M(k^\star)$, we can express $v_j-v_0$ in a different 
way as $v_j-v_0= v_j-v_h=-((v_{j+1}-v_{j})+(v_{j+2}-v_{j+1})+\cdots+(v_{h}-v_{h-1}))$.
Now, according to Lemma~\ref{lem:bound}, we have $(\beta^\vee,\omega_k)\leq 0$,
for any coroot $\beta^\vee$ involved in the expression of $v_{m+1}-v_m$,
for $m=j,j+1,\dots,h-1$.
Thus again we get $(v_j-v_0,\omega_k)\geq 0$.
\end{proof}

\begin{lemma}\label{lem:Rchange}
Let $R_1^+$ and $R_2^+$ be two choices of positive roots for $R$ compatible with $c$.  Then there is a bijection $\v \mapsto \v'$ between $(W,R_1^+,c)$-Coxeter necklaces and $(W,R_2^+,c)$-Coxeter necklaces such that $Q(\v) = Q(\v')$.
\end{lemma}
\begin{proof}
By Proposition \ref{prop:compatible}, we may assume that $R_2^+ = s_1 \cdot R_1^+$, and $c = s_1 s_2 \cdots s_r$ is a reduced factorization in $R_1^+$.  Letting $s'_i = s_1 s_i s_1$ be the simple generators of $R_2^+$, we have $c =s'_2 \cdots s'_r s'_1$.  Let $\beta_1,\beta_2,\ldots$ be the enumeration of $R$ associated to the factorization $c = s_1 s_2 \cdots s_r$.  Then the enumeration of $R$ associated to $c =s'_2 \cdots s'_r s'_1$ is $\beta_2, \beta_3,\ldots$.  

Let $v'_i = v_i + [\beta^\vee_{ir+1}](v_{i+1}-v_i)$, where $[\beta^\vee_{ir+1}](v_{i+1}-v_i)$ denotes the coefficient of $\beta^\vee_{ir+1}$ in the vector $v_{i+1}-v_i$ which lies in the cone $c^i(A)$ spanned by $\beta^\vee_{ir+1},\beta^\vee_{ir+2},\ldots,\beta^\vee_{(i+1)r}$.  It is clear that $\v \mapsto \v'$ is a bijection from $(W,R_1^+,c)$-Coxeter necklaces to $(W,R_2^+,c)$-Coxeter necklaces.  We claim that $Q(\v) = Q(\v')$.  It suffices to show that each $v'_j$ belongs to $Q(\v)$, and we can reduce to showing that $v'_j - v_0 = (v_j - v_0) + a \beta^\vee_{jr+1} \in A$, where $a \in \R_{\geq 0}$.  The proof is identical to that of Proposition \ref{prop:necklace}, using Lemma \ref{lem:bound}.
\end{proof}

\begin{lemma}\label{lem:cchange}
Let $c$ and $c' = wcw^{-1}$ be two Coxeter elements.  Then $\v$ is a $(W,R^+,c)$-Coxeter necklace if and only if $w \cdot \v$ is a $(W,w \cdot R^+, c')$-Coxeter necklace.  We have $w \cdot Q(\v) = Q(w \cdot \v)$ where $Q(\v)$ denotes a $(W,c)$-twisted alcoved polytope and $Q(w\cdot v)$ denotes a $(W,c')$-twisted alcoved polytope.
\end{lemma}
\begin{proof}
Clear from the definitions.
\end{proof}

We call a Coxeter necklace $\v = (v_0,v_1,\ldots,v_{h-1})$ {\it generic} if each difference $v_{i}-v_{i-1}$ belongs to the interior of the cone $c^i(A)$.

\begin{lemma}\label{lem:onevertex}
Suppose that $\v$ is a generic $(W,R^+,c)$-Coxeter necklace.  Then for each twisted root $\omega$, the face $\{ x \in Q(\v) \mid (x,\omega) =0 \}$ contains at most one of the vertices $v_0,v_1,\ldots,v_{h-1}$.  In particular, $Q(\v)$ is a generic $(W,c)$-twisted alcoved polytope.
\end{lemma}
\begin{proof}
By acting with $c$, we may assume that $\omega = \omega_k$ where $k \in \{1,2,\ldots,r\}$.  In the proof of Proposition \ref{prop:necklace}, we note that we have the strict inequalities $(\beta^\vee_k,\omega_k) >0$ and $(\beta^\vee_{k-r},\omega_k) < 0$.
\end{proof}

Whereas Theorem~\ref{th:intersection_C} describes a generalized
$W$-permutohedron as an intersection of $|W|$ cones, our next result describes a $(W,c)$-polypositroid as an intersection of $h$ cones.

\begin{proposition}\label{prop:genWenv}
Fix a Coxeter element $c$.  Suppose $P$ is a $W$-generalized permutohedron.  Then for any choice of $R^+$, we have 
$$
\env_c(P) = Q(\v) = \bigcap_{i=0}^{h-1} (v_i + c^i(C)),
$$
where $\v = (v_0,\dots,v_{h-1})$ is a $(W,R^+,c)$-Coxeter necklace.  In particular, this holds for $P$ a $(W,c)$-polypositroid, in which case $\env_c(P) = P$.
\end{proposition}

\begin{proof}
By Proposition \ref{prop:envelope}, any $(W,c)$-twisted alcoved polytope has the form $P = \bigcap_{i=0}^{h-1} (v_i + c^i(C))$ for some uniquely determined points $v_i$, and thus this holds for $\env_c(P)$.  

Since $P$ is also a $W$-generalized permutohedron, by Corollary \ref{cor:genW}, we must have $v_i = v_{c^i}$, the vertex of $P$ that is the minimum in $c^i(C)$-dominance order.
Let us show that the conditions $v_i-v_{i-1} \in c^i(A)$ hold.
It is enough to show this for $i=0$.
(The general case is obtained by the action of $c^i$.)
Consider the following sequence of vertices that connect $v_0=v_{id}$ with $v_1=v_{c}$:
$$
v_{id},\  v_{s_1}, \ v_{s_1s_2},\ \dots,\ v_{s_1s_2\cdots s_r} =v_c.
$$
According to Theorem~\ref{th:intersection_C_simple},
$v_{s_1 s_2\cdots s_i}-v_{s_1 s_2 \cdots s_{i-1}} = a_i \, s_1\cdots 
s_{i-1}(\alpha^\vee_i)$, for $a_i\in\R_{\geq 0}$.
Thus $v_1 - v_0 = \sum_{i=1}^r a_i\, \beta^\vee_i \in A$.  We conclude that $\v=(v_0,v_1,\ldots,v_{h-1})$ is a Coxeter necklace, as required.
\end{proof}

\section{Balanced arrays}
\begin{definition}
A \emph{$W$-balanced array} is a collection $(m_\alpha)_{\alpha \in R}$ of nonnegative real numbers satisfying the equality
\begin{equation}\label{eq:ba}
\sum_\alpha m_\alpha \alpha^\vee = 0.
\end{equation}
A {\it $W$-balanced pair} is a pair $((m_\alpha),z)$ consisting of a $W$-balanced array $(m_\alpha)$ and a vector $z \in V$.  
\end{definition}

Now let $c=s_1s_2\cdots s_r$ be a standard Coxeter element with respect to $R^+$ and $\beta_1,\beta_2,\ldots,$ be the corresponding ordering of roots.  We define a \emph{Coxeter necklace} 
$$
\v(m_\alpha,z) = (v_0,v_1,\ldots,v_{h-1},v_h = v_0)
$$
by setting $v_0 = z$ and 
\begin{equation}\label{eq:m2v}
v_i = v_{i-1} + \sum_{k=(i-1)r+1}^{ir} m_{\beta_k} \beta^\vee_k
\end{equation}
for $i=1,2,\ldots,h-1$.  The equality $v_h = v_0$ follows from \eqref{eq:ba} and Proposition~\ref{prop:listroots}.

\begin{proposition}\label{prop:balneck}
The map $(m_\alpha,z) \mapsto \v(m_\alpha,z)$ is a bijection between $W$-balanced pairs and $(W,R^+,c)$-Coxeter necklaces, for any choice of $(R^+,c)$.
\end{proposition}
\begin{proof}
In \eqref{eq:m2v}, the $m_\beta$ can be recovered from $\v(m_\alpha,z)$ because $\{\beta^\vee_{(i-1)r+1}, \ldots,\beta^\vee_{ir}\}$ is a basis.  The result easily follows.
\end{proof}

\begin{proposition}\label{prop:changeword}
Fix a $W$-balanced pair $((m_\alpha),z)$.  Then the $(W,R^+,c)$-Coxeter necklace $\v(m_\alpha,z)$ depends on $(R^+,c)$ and not on the reduced word of $c$.
\end{proposition}
\begin{proof}
Changing the reduced word of $c$ replaces $\{\beta_1,\ldots,\beta_r\}$ by a permutation of the same set.
\end{proof}

For a balanced pair $(m_\alpha,z)$, we also define a \emph{Coxeter belt} 
$$
\u(m_\alpha,z):=(u_0,u_1,\ldots, u_{hr-1},u_{hr} = u_0)
$$
by setting $u_0 = z$ and
$$
u_i = u_{i-1} + m_{\beta_i} \beta^\vee_i
$$
for $i = 1,2,\ldots,hr-1$.  Note that given $(m_\alpha,z)$, the Coxeter belt depends on $R^+$, on $c$, and on a reduced word for $c$.

Suppose $R^+$ has simple roots $\alpha_1,\ldots,\alpha_r$.  Then $s_1 R^+$ is also a positive system, and its simple roots are $-\alpha_1, s_1 \alpha_2,\ldots, s_1 \alpha_r$.

\begin{proposition}\label{prop:changeR+}
The $(W,R^+,c)$-Coxeter belt for the $W$-balanced pair $((m_\alpha),z)$ with respect to $(R^+,c=s_1s_2\cdots s_r)$ is the same, up to a cyclic shift, as the $(W,s_1R^+,c)$-Coxeter belt for the balanced pair $((m_\alpha),z'=u_1)$ with respect to $(s_1 R^+,c = (s_1s_2s_1) \cdots (s_1s_r s_1) s_1)$.
\end{proposition}
\begin{proof} Follows from Proposition~\ref{prop:compatible}.
\end{proof}

Given a balanced pair $(m_\alpha,z)$, let $Q(m_\alpha,z):=Q(\v(m_\alpha,z))$.
\begin{corollary}\label{cor:belt}
Let $(m_\alpha,z)$ be a balanced pair.  Then $\env_c(\u(m_\alpha,z)) = Q(m_\alpha,z)$, and each point of the Coxeter belt $\u(m_\alpha,z)$ is a vertex of $Q(m_\alpha,z)$.
\end{corollary}
\begin{proof}
By Proposition~\ref{prop:necklace}, we have $\env_c(\v(m_\alpha,z)) = Q(m_\alpha,z)$.  Since the Coxeter necklace $\v(m_\alpha,z)$ is a subset of the Coxeter belt $\u(m_\alpha,z)$, we have $\env_c(\u(m_\alpha,z)) \supseteq \env_c(\v(m_\alpha,z))$.  To establish equality, it suffices to show that each $u_i$ belongs to $Q(m_\alpha,z)$.  By combining the action of $c$ with Proposition~\ref{prop:changeR+}, it suffices to show that $u_1 \in Q(\v(m_\alpha,z))$.  This follows from Lemma~\ref{lem:Rchange}.  

The claim that every point on the Coxeter belt is a vertex of $Q(m_\alpha,z)$ follows from applying Proposition~\ref{prop:necklace} to the $(W,s_1R^+,c)$-Coxeter necklace appearing in Proposition~\ref{prop:changeR+}.
\end{proof}

\begin{definition}
A {\it $(W,c)$-balancedtope} is a polytope of the form $Q(m_\alpha,z)$.  
\end{definition}
By Propositions~\ref{prop:changeword} and \ref{prop:changeR+} and Corollary~\ref{cor:belt}, up to changing $z$ (or equivalently, up to a translation), the $(W,c)$-balancedtope $Q(m_\alpha,z)$ does not depend on the choice of reduced word of $c$, or on the choice of $R^+$.  

\section{Prepolypositroids}

Let $\R^{\tR} = \{(a_\omega)_{\omega \in \tR}\}$ denote the vector space whose coordinates are labeled by the set $\tR$ of twisted roots.  Let $\CpolyW \subset \R^{\tR}$ denote the cone cut out by the inequalities
\begin{equation}\label{eq:conefacets}
a_{c^{m-1}\omega_k}+ a_{c^m\omega_k} \geq \sum_{k \to i} -A_{ik} a_{c^m \omega_i}+ \sum_{i \to k} -A_{ik} a_{c^{m-1} \omega_i}
\end{equation}
for $k \in \Z$ and $1 \leq m \leq r$.
Recall that $j \to i$ if $i$ and $j$ are connected in the Dynkin diagram and $i$ occurs before $j$ in all reduced words of $c$. The twisted roots appearing in \eqref{eq:conefacets} all belong to the set $\{c^{m-1}\omega_k = \tbeta_{k+(m-1)r},\tbeta_{k+(m-1)r+1},\ldots,\tbeta_{k+mr}= c^m \omega_k\}$.


\begin{proposition}\label{prop:coneindep}
The inequalities \eqref{eq:conefacets} depend on $c$ and not on the choice of reduced word of $c$, or on the choice of $R^+$.
\end{proposition}
\begin{proof}
That the inequalities \eqref{eq:conefacets} do not depend on the reduced word of $c$ is apparent.  For the second part, using Proposition~\ref{prop:compatible}, we need to check what happens if we replace $R^+$ by $s_1 R^+$, where $c = s_1s_2 \cdots s_r$.  If $j \to_{c,R^+} i$ then we also have $j \to_{c,s_1R^+} i$ unless one of $i,j$ is equal to 1, in which case the relation reverses.  This is exactly the required condition for the inequalities \eqref{eq:conefacets} to be preserved when $R^+$ is changed to $s_1 R^+$.  
\end{proof}

\begin{example}
Let us take $(W,R^+,c)$ as in Examples~\ref{ex:Aroot} and~\ref{ex:Aroot2}.  
In the notation of Part~\ref{part:1}, we have $a_{ij} = a_{h_i - h_j}$.  The inequalities \eqref{eq:conefacets} are of three types: (1) $a_{i, i+1} + a_{i+1,i+2} \geq a_{i,i+2}$, (2) $a_{i+1, i} + a_{i+2,i+1} \geq a_{i+2,i}$, and (3) $a_{i,j} + a_{i+1,j+1} \geq a_{i,j+1} + a_{j,j+1}$ where $i,i+1,j,j+1$ are distinct.  The inequalities (1) and (2) are special cases of the triangle inequality \eqref{eq:triangle}, while (3) is a special case of \eqref{eq:anoncross}.  It will follow from Theorem~\ref{thm:pretypeA}, and can be verified directly, that the inequalities \eqref{eq:triangle} and \eqref{eq:anoncross} are consequences of the smaller set of inequalities \eqref{eq:conefacets}.  Indeed, the $n(n-1) = |R|$ inequalities in \eqref{eq:conefacets} are exactly the facet inequalities appearing in Corollary~\ref{cor:conefacets}.
\end{example}

\begin{definition}
A {\it $(W,c)$-prepolypositroid} is a $(W,c)$-twisted alcoved polytope cut out by the halfspaces $(x,\omega) \leq a_{\omega}$, $\omega \in \tR$ where $a_\omega \in \CpolyW$.
\end{definition}
We call $\CpolyW$ the {\it cone of $(W,c)$-prepolypositroids}.

\begin{theorem}\label{thm:threecones}
There are natural isomorphisms between the following three cones:
\begin{enumerate}
\item the cone of $(W,c)$-prepolypositroids.
\item the cone of $W$-balanced pairs $((m_\alpha), z)$;
\item the cone of $(W,R^+,c)$-Coxeter necklaces for any choice of $R^+$.
\end{enumerate}
Furthermore, if a $(W,c)$-prepolypositroid $P$ arises from $(a_\omega) \in \CpolyW$ then each $a_\omega \in \R$ is minimal, i.e., is a value of the support function of $P$.
\end{theorem}
Proposition~\ref{prop:balneck} gives the isomorphism between (2) and (3).  For the remainder of this section, our aim is to show that $((m_\alpha), z) \mapsto Q(m_\alpha, z)$ is a bijection between balanced pairs and $(W,c)$-prepolypositroids.

\begin{proposition}\label{prop:homogeneous}
We have 
\begin{equation}\label{eq:facetequal}
c^{m-1}\omega_k + c^m\omega_k = \sum_{k \to i} -A_{ik} c^m \omega_i + \sum_{i \to k} -A_{ik} c^{m-1} \omega_i 
\end{equation}
for any $k$ and any $m$.
\end{proposition}
\begin{proof}
By Proposition \ref{prop:coneindep}, to verify the claim we can assume that $c = s_1s_2 \cdots s_r$ and verify the equality
$$
\omega_1 + c \omega_1 = \sum_{j} -A_{j1} \omega_j
$$
where the sum is over all $j$ connected to $1$ in the Dynkin diagram.  The LHS is equal to $2\omega_1 - \alpha_1$, and we check that 
$$
(\alpha_j^\vee, 2\omega_1 - \alpha_1) = \begin{cases} -A_{j1} & \mbox{if $j$ is connected to $i$ in the Dynkin diagram} \\
0 & \mbox{otherwise.} 
\end{cases} 
$$
\end{proof}

Given a $(W,c)$-balancedtope $Q(m_\alpha,z)$, we define a collection $(a_\omega)$ of real numbers, one for each twisted root $\omega \in \tR$ by 
\begin{equation}\label{eq:maxQ}
a_\omega=a_\omega(Q(m_\alpha,z)) := \max((x,\omega) \mid x \in Q(m_\alpha,z)).
\end{equation}

\begin{proposition}\label{prop:satisfy}
For any $(W,c)$-balancedtope $Q(m_\alpha,z)$, the collection $(a_\omega)$ satisfies the inequalities \eqref{eq:conefacets}.
\end{proposition}
\begin{proof}
It follows from Lemma~\ref{lem:YZ} that the inequalities \eqref{eq:conefacets} is equivalent to the set of ``negated" inequalities 
\begin{equation}\label{eq:conefacets2}
a_{-c^{m-1}\omega_k}+ a_{-c^m\omega_k} \geq \sum_{k \to i} -A_{ik} a_{-c^m \omega_i}+ \sum_{i \to k} -A_{ik} a_{-c^{m-1} \omega_i}.
\end{equation}
for all $k$ and $m$.  By Proposition \ref{prop:coneindep}, to verify the claim we can assume that $c = s_1s_2 \cdots s_r$ and verify just one of the inequalities \eqref{eq:conefacets2}, say
$$
a_{-\omega_1} + a_{-c \omega_1 } \geq \sum_{j} -A_{j1} a_{-\omega_j}
$$
where the sum is over all $j$ connected to $1$ in the Dynkin diagram.  By Proposition \ref{prop:homogeneous}, the claim is translation invariant.  Thus we may assume that $v_0 = 0$, that is, $z = 0$.  

The maximum of $-\omega_1,\ldots,-\omega_r$ on $Q(m_\alpha,z)$ occurs at the vertex $v_0$.  The maximum of $-c\omega_1$ occurs at vertex $v_1$, where the value taken is greater than or equal to $(-c\omega_1, v_0) = 0$.  The required inequality now follows from $a_{-c\omega_1} \geq 0$.
\end{proof}

By Proposition~\ref{prop:satisfy}, we have an injective and linear map $(m_\alpha,z) \mapsto (a_\omega)$ from the cone of $W$-balanced pairs to the cone $\CpolyW$ of $(W,c)$-prepolypositroids.  The cone of balanced pairs has exactly $|R|$ facets, given by $m_\alpha = 0$, for $\alpha \in R$.  Each equality $a_{-c^{m-1}\omega_k}+ a_{-c^m\omega_k} = \sum_{k \to i} -A_{ik} a_{-c^m \omega_i}+ \sum_{i \to k} -A_{ik} a_{-c^{m-1} \omega_i}$ defines a face of $\CpolyW$.

\begin{lemma}\label{lem:face2face}
The map $(m_\alpha,z) \mapsto (a_\omega)$ sends the facet $\{m_\alpha = 0\}$ of the cone of $W$-balanced pairs to the face $\{a_{-c^{m-1}\omega_k}+ a_{-c^m\omega_k} = \sum_{k \to i} -A_{ik} a_{-c^m \omega_i}+ \sum_{i \to k} -A_{ik} a_{-c^{m-1} \omega_i}\}$ of $\CpolyW$, where $\alpha$ satisfies $\tilde \alpha = (I-c)^{-1} \alpha = c^{m-1} \omega_k$.
\end{lemma}
\begin{proof}
In the end of the proof of Proposition \ref{prop:satisfy}, it suffices to note that $a_{-c\omega_1} = m_{\alpha_1}$.  Indeed, for $i \in[2,r]$, we have
$$
(\beta_i^\vee,c\omega_1) = ( c^{-1}s_1 s_2 \cdots s_{i-1}\alpha_i^\vee,\omega_1) =  (s_{r} \cdots s_i\alpha_i^\vee,\omega_1) = 0
$$
but $(\beta^\vee_1,c\omega_1) = (\alpha^\vee_1,\omega_1-\alpha_1) = -1$.
\end{proof}
It follows that each equality $\{a_{-c^{m-1}\omega_k}+ a_{-c^m\omega_k} = \sum_{k \to i} -A_{ik} a_{-c^m \omega_i}+ \sum_{i \to k} -A_{ik} a_{-c^{m-1} \omega_i}\}$ defines a facet of $\CpolyW$.

\begin{proof}[Proof of Theorem~\ref{thm:threecones}]
We compare (1) and (2).
It follows from Lemma~\ref{lem:face2face} that each equality $a_{-c^{m-1}\omega_k}+ a_{-c^m\omega_k} = \sum_{k \to i} -A_{ik} a_{-c^m \omega_i}+ \sum_{i \to k} -A_{ik} a_{-c^{m-1} \omega_i}$ defines a facet of the cone of $(W,c)$-prepolypositroids.  It then follows from the same lemma that the map $(m_\alpha,z) \mapsto (a_\omega)$ is a linear isomorphism, completing the proof of the isomorphism between the three cones.

The last sentence of Theorem~\ref{thm:threecones} follows from \eqref{eq:maxQ}.
\end{proof}

\section{From prepolypositroids to polypositroids}

\subsection{Alcoved envelope of generalized $W$-permutohedra}

\begin{theorem}\label{thm:permtopre} \
\begin{enumerate}
\item
The cone of generalized $W$-permutohedra $\CsubW$ projects to to the cone of $(W,c)$-prepolypositroids $\CpolyW$ by projecting the vector $(a_{w\omega_i}) \in \R^{W \cdot \{\omega_1,\ldots,\omega_r\}}$ of \eqref{eq:genpermineq} to $(a_{c^m \omega_i}) \in \R^{\tR}$.
\item
The $(W,c)$-twisted alcoved envelope of a generalized $W$-permutohedron is a $(W,c)$-prepolypositroid.
\item
The vertices $(v_\id, v_c, v_{c^2},\ldots, v_{c^{h-1}})$ of a generalized $W$-permutohedron is a $(W,c)$-Coxeter necklace.
\end{enumerate}
\end{theorem}
\begin{proof}
(3) was established in the proof of Proposition~\ref{prop:genWenv}.  (1) and (2) thus follows from Theorem~\ref{thm:threecones}.
\end{proof}

\begin{conjecture}\label{conj:genpermalc}
The maps in Theorem \ref{thm:permtopre} are surjective.
\end{conjecture}

\subsection{Type $A$}

\begin{theorem}\label{thm:pretypeA}
Suppose $R$ is of type $A$.  Then every $(W,c)$-prepolypositroid is also a generalized $W$-permutohedron.  Thus the class of $(W,c)$-prepolypositroids is identical to the class of $(W,c)$-polypositroids.
\end{theorem}
\begin{proof}
Let $c$ be the Coxeter element of Example~\ref{ex:Aroot}.  Then by Theorem~\ref{thm:main} and Theorem~\ref{thm:threecones}, the class of $(W,c)$-prepolypositroids is exactly the class of polypositroids.  Thus the result holds in this case.  Now let $c' = wcw^{-1}$ be an arbitrary Coxeter element.  Since the class of generalized $W$-permutohedra is preserved under the action of $W$, the result holds by Remark~\ref{rem:cc'}.
\end{proof}

It follows from Theorem~\ref{thm:pretypeA} that Conjecture~\ref{conj:genpermalc} holds in type $A$.

\subsection{A prepolypositroid that is not a polypositroid}\label{ssec:D4ex}
We give an example of a $(W,c)$-prepolypositroid that is not a $(W,c)$-polypositroid. 
Let $R = D_4$.  We take as positive simple roots
$$
\alpha_1=(1,-1,0,0), \qquad \alpha_2=(0,1,-1,0), \qquad \alpha_3=(0,0,1,-1), \qquad \alpha_4=(0,0,1,1),
$$
and let $c = s_1s_2s_3s_4$, so that the Dynkin diagram is oriented
$$
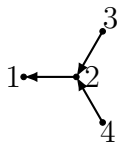
\begin{tikzpicture}[scale=0.7]
\tikzset{>=latex}
\filldraw[black] (0:0) circle (1.5pt);
\filldraw[black] (180:1) circle (1.5pt);
\filldraw[black] (-60:1) circle (1.5pt);
\filldraw[black] (60:1) circle (1.5pt);
\draw[->,thick] (0:0)--(180:1);
\draw[->,thick] (-60:1)--(0:0);
\draw[->,thick] (60:1)--(0:0);
\node at (180:1.2) {$1$};
\node at (0:0.3) {$2$};
\node at (-60:1.2) {$4$};
\node at (60:1.3) {$3$};
\end{tikzpicture}
$$
%
%
The ordering of the 24 roots of $R$ is given by (here $r=4$ and $h=6$)
\begin{align*}
&(1, -1, 0, 0), &&(1, 0, -1, 0), &&(1, 0, 0, -1), &&(1, 0, 0, 1), \\ 
&(0, 1, -1,0), &&(1, 1, 0, 0), &&(0, 1, 0, 1), &&(0, 1, 0, -1), \\ 
&(1, 0, 1, 0), &&(0, 1,1, 0), &&(0, 0, 1, -1), &&(0, 0, 1, 1), \\
&(-1, 1, 0, 0), &&(-1, 0, 1, 0), &&(-1, 0, 0, 1), &&(-1, 0, 0, -1), \\
&(0, -1, 1, 0), &&(-1, -1, 0, 0), &&(0, -1, 0, -1), &&(0, -1, 0, 1), \\
&(-1, 0, -1, 0), &&(0, -1, -1, 0), &&(0, 0, -1, 1), &&(0, 0, -1, -1),
\end{align*}
and we see that $M(k) = 3$ for $k=1,2,3,4$.
The twisted roots $\tR$ are, in the same order,
\begin{align*}
\label{eq:D4twist}
&(1, 0, 0, 0), &&(1, 1, 0, 0), &&(\tfrac{1}{2}, \tfrac{1}{2}, \tfrac{1}{2}, -\tfrac{1}{2}), &&(\tfrac{1}{2}, \tfrac{1}{2}, \tfrac{1}{2},\tfrac{1}{2}), \\
&(0, 1, 0, 0), &&(0, 1, 1, 0), &&(-\tfrac{1}{2}, \tfrac{1}{2}, \tfrac{1}{2}, \tfrac{1}{2}), &&(-\tfrac{1}{2}, \tfrac{1}{2}, \tfrac{1}{2}, -\tfrac{1}{2}), \\
&(0, 0, 1, 0), &&(-1, 0, 1, 0), &&(-\tfrac{1}{2}, -\tfrac{1}{2}, \tfrac{1}{2}, -\tfrac{1}{2}), &&(-\tfrac{1}{2}, -\tfrac{1}{2}, \tfrac{1}{2}, \tfrac{1}{2}),\\
&(-1, 0, 0, 0), &&(-1, -1, 0, 0), &&(-\tfrac{1}{2}, -\tfrac{1}{2}, -\tfrac{1}{2}, \tfrac{1}{2}), &&(-\tfrac{1}{2}, -\tfrac{1}{2}, -\tfrac{1}{2}, -\tfrac{1}{2}), \\
&(0, -1, 0, 0), &&(0, -1, -1, 0), &&(\tfrac{1}{2}, -\tfrac{1}{2}, -\tfrac{1}{2}, -\tfrac{1}{2}), &&(\tfrac{1}{2}, -\tfrac{1}{2}, -\tfrac{1}{2}, \tfrac{1}{2}), \\
&(0, 0, -1, 0), &&(1, 0, -1, 0), &&(\tfrac{1}{2}, \tfrac{1}{2}, -\tfrac{1}{2}, \tfrac{1}{2}), &&(\tfrac{1}{2}, \tfrac{1}{2}, -\tfrac{1}{2}, -\tfrac{1}{2}).
\end{align*}
We consider the Coxeter necklace
$$
\v=( v_0=(0, 0, 0, 0), v_1=v_2=(1, 0, 0, 1), v_3=v_4=(1, 1, 1, 1), v_5=(0, 0, 1, 1)).
$$
The polytope $Q(\v)$ is a $(W,c)$-prepolypositroid and the $a_\omega$ are given by 
$$(1, 2, 1, 2, 1, 2, 1, 0, 1, 1, 0, 1, 0, 0, 0, 0, 0, 0, 0, 1, 0, 1, 1,0)$$
in the same order as $\tR$.
The inequalities \eqref{eq:conefacets} can be verified directly.  For example, we have
$$
2+2 = a_{(1,1,0,0)}+a_{(0,1,1,0)} \geq a_{(\tfrac{1}{2}, \tfrac{1}{2}, \tfrac{1}{2}, -\tfrac{1}{2})}+ a_{(\tfrac{1}{2}, \tfrac{1}{2}, \tfrac{1}{2},\tfrac{1}{2})} + a_{(0,1,0,0)} \geq 1+2+1.
$$
$$ 1+1 = a_{(\tfrac{1}{2}, \tfrac{1}{2}, \tfrac{1}{2}, -\tfrac{1}{2})} + a_{(-\tfrac{1}{2}, \tfrac{1}{2}, \tfrac{1}{2}, \tfrac{1}{2})} \geq a_{(0, 1, 1, 0)}=2.$$
Now, one computes (for example, by \cite{polymake}) that $Q(\v)$ has seven vertices
$$
\{(0, 0, 1, 1), (1, 0, 1, 2), (1, 0, 0, 1), (1, 1, 1, 1), (1, 0, 1,0), (0, 0, 0, 0), (0, 1, 0, 1)\}
$$
and that there is an edge connecting $(1,0,1,2)$ and $(1,0,1,0)$.  Indeed, one can check that this edge is the intersection of the three facets indexed by $(1,0,0,0)$, $(0,0,1,0)$, and $(0,-1,0,0)$.  This edge is in the direction $(0,0,0,1)$, which is not a root direction.  Thus, $Q(\v)$ is not a generalized $W$-permutohedron.

It turns out that $Q(\v)$ {\bf is} the $(W,c)$-twisted alcoved envelope of a generalized permutohedron, namely one with vertices
$$
\{(0, 0, 1, 1), (1, 0, 0, 1), (1, 1, 1, 1), (1, 0, 1, 0), (0, 0, 0,0), (0, 1, 0, 1)\},
$$
consistent with Conjecture~\ref{conj:genpermalc}.

\section{Prepolypositroids and finite type cluster algebras}\label{sec:cluster}

We briefly recall some basic facts concerning finite type cluster algebras, following \cite{YZ}.  Let $\A(W,R^+,c)$ denote the finite type cluster algebra with principal coefficients of type $R$ and associated to a compatible pair $(R^+,c)$.  Usually, the choice of positive system $R^+$ is not made explicit in the theory of cluster algebras, but for our purposes it is necessary.

The cluster algebra $\A(W,R^+,c)$ is a commutative subring of the field of rational functions $\C(x_1,x_2,\ldots,x_r,y_1,y_2,\ldots,y_r)$, where $x_i$ (resp. $y_i$) are called {\it initial mutable cluster variables} (resp. {\it coefficient variables}).  The choice of $c$ determines an initial exchange matrix $B$, given by the formula \cite[(1.4)]{YZ}.  The cluster algebra $\A(W,R^+,c)$ contains a distinguished set of {\it cluster variables}, and associated to each cluster variable is a \emph{$g$-vector} which belongs to $V$.

\begin{theorem}[{\cite[Theorem 1.4 and Theorem 1.10]{YZ}}]
The cluster variables $x_{\tbeta}$ of $\A(W,R^+,c)$ are labeled by the set 
$$
\Pi(c):= \{c^m \omega_i \mid i = 1,2,\ldots,r \text{ and } 0 \leq m \leq M(i)\}
$$
where $M(i)$ is defined in Proposition~\ref{prop:listroots}.  Furthermore, $\tbeta$ is the {\rm $g$-vector} of $x_{\tbeta}$.
\end{theorem}

It follows from Lemma~\ref{lem:YZ} that $\pm \omega_i \in \Pi(c)$.  The cluster variables $x_{\tbeta}$ are arranged into {\it clusters}.  We say that $\tbeta$ and $\tgamma$ (or  $x_{\tbeta}$ and $x_{\tgamma}$) are {\it $(R^+,c)$-compatible}, or simply {\it compatible}, if they belong to the same cluster; the clusters are exactly the collections of $r$ pairwise-compatible cluster variables.  The {\it $(W,R^+,c)$-cluster fan} (also called the {\it $c$-Cambrian fan}) is the complete fan in $V$ with cones ${\rm span}_{\geq 0}(\tgamma_1,\tgamma_2,\ldots,\tgamma_s)$ where $\{\tgamma_1,\tgamma_2,\ldots,\tgamma_s\} \subset \Pi(c)$ is a set of pairwise $(R^+,c)$-compatible vectors.  A polytope with normal fan equal to the $(W,R^+,c)$-cluster fan is called a {\it $(W,R^+,c)$-generalized associahedron}.  The following result combines work of Hohlweg, Lange and Thomas \cite{HLT}, Reading and Speyer \cite{RS}, and Yang and Zelevinsky \cite[Remark 6.1]{YZ}.

\begin{theorem}\label{thm:genass}
The $(W,R^+,c)$-cluster fan is a refinement of the $W$-Coxeter fan.  Furthermore, $(W,R^+,c)$-generalized associahedra exist and are generalized $W$-permutohedra.
\end{theorem}

The $W$-Coxeter fan has maximal cones $C_w$ labeled by Weyl group elements $w \in W$.  By Theorem~\ref{thm:genass} every maximal cone of the $(W,R^+,c)$-cluster fan is a union of the cones $C_w$. A {\it $(W,R^+,c)$-singleton} \cite{HLT} is a Weyl group element $w$ such that $C_w$ is itself a maximal cone of the $(W,R^+,c)$-cluster fan.  Hohlweg, Lange and Thomas \cite[Theorem 1.2]{HLT} characterize the set of $(W,R^+,c)$-singletons as prefixes (up to commutation relations) of a particular reduced word of $w_0$ that depends on $R^+$ and $c$.

The following result should be compared with Corollary~\ref{cor:Tw}.
\begin{proposition}
Let $P$ be a generic simple $(W,c)$-polypositroid.
For any choice of $R^+$, and any $(W,R^+,c)$-singleton $w$, the cone $C_w$ is a maximal cone of the normal fan $\Ncal(P)$.
\end{proposition}
\begin{proof}
By definition, the normal fan $\Ncal(P)$ is a coarsening of the $W$-Coxeter fan.  It follows that there exists a maximal cone $C$ of $\Ncal(P)$ that contains the simplicial cone $C_w$.  But each generating ray of $C_w$ must be a generating ray of $C$, for otherwise $P$ would not be generic.  But then we must have $C = C_w$, since both are simplicial cones.
\end{proof}

The cluster variables are related by {\it exchange relations}.  A distinguished subset of the exchange relations are called {\it primitive exchange relations} in \cite{YZ}.

\begin{theorem}[{\cite[Theorem 1.5]{YZ}}]
The primitive exchange relations of $\A(W,R^+,c)$ are of the form 
\begin{equation}\label{eq:initial}
x_{-\omega_k} x_{\omega_k} = y_k \prod_{k \to i} x_{\omega_i}^{-A_{ik}} \prod_{i \to k} x_{-\omega_i}^{-A_{ik}} + 1;
\end{equation}
for $k = 1,2,\ldots,r$ and
\begin{equation}\label{eq:noninitial}
x_{c^{m-1}\omega_k} x_{c^m\omega_k} = \prod_{k \to i} x_{c^m\omega_i}^{-A_{ik}} \prod_{i \to k} x_{c^{m-1}\omega_i}^{-A_{ik}} + Y,
\end{equation}
for $k = 1,2,\ldots,r$ and $1 \leq m \leq M(k)$, where $Y$ is some monomial in the $y_i$.
\end{theorem}

The relations \eqref{eq:initial} and \eqref{eq:noninitial} are homogeneous with respect to the $g$-vector grading.  In particular, we have for each $k =1,2,\ldots,r$,
\begin{equation}\label{eq:ginitial}
-\omega_k + \omega_k = 0;
\end{equation}
and for each $ 1 \leq m \leq M(k)$,
\begin{equation}\label{eq:gnoninitial}
c^{m-1}\omega_k + c^m\omega_k = \sum_{k \to i} -A_{ik} c^m \omega_i + \sum_{i \to k} -A_{ik} c^{m-1} \omega_i. 
\end{equation}
The latter we recognize as a special case of Proposition~\ref{prop:satisfy}.

More generally, we say that $\tbeta, \tgamma \in \Gamma(c)$ are an {\it exchangeable pair} if we have a (necessarily unique) exchange relation that exchanges $x_{\tbeta}$ for $x_{\tgamma}$.  This exchange relation takes the form
\begin{equation}\label{eq:exchange}
x_{\tbeta} x_{\tgamma} = \prod_{\tdelta \in E(\tbeta,\tgamma)} x_\tdelta^{c_{\beta,\gamma;\delta}} + \text{ other monomial},
\end{equation}
where $c_{\beta,\gamma;\delta} > 0$, and $E(\beta,\gamma) \subset \Pi(c)$ consists of elements that are pairwise compatible, and compatible with both $\tbeta$ and $\tgamma$.   Here, the key point is that one of the two monomials on the RHS of the exchange relation \eqref{eq:exchange} does not involve any of the coefficient variables, known as {\it sign-coherence}.  For a description of all the exchange relations in a principal coefficient finite-type cluster algebra, see \cite{ST}.  Note that in type $A$, any incompatible pair $(\tbeta,\tgamma)$ is automatically exchangeable, but this is not the case in general type.

The following is the main result of this section.
\begin{theorem}\label{thm:moreinequalities}
Let $P$ be a $(W,c)$-prepolypositroid defined by the inequalities $(x,\omega) \leq a_{\omega}$ where $(a_\omega) \in \CpolyW$.  Then for any choice of $R^+$ compatible with $c$, and any pair $(x_{\tbeta},x_{\tgamma})$ of exchangeables cluster variables we have 
\begin{equation}\label{eq:moreineq}
a_{\tbeta} + a_{\tgamma} \geq \sum_{\tdelta \in E(\tbeta,\tgamma)} c_{\beta,\gamma;\delta} a_{\tdelta}
\end{equation}
where $E(\tbeta,\tgamma)$ and $c_{\beta,\gamma;\delta}$ are defined in \eqref{eq:exchange}.
\end{theorem}

\begin{proof}
We begin by noting that \eqref{eq:exchange} is homogeneous with respect to the $g$-vector grading, so
\begin{equation}\label{eq:anyexc}
\tbeta + \tgamma = \sum_{\tdelta \in E(\tbeta,\tgamma)} c_{\beta,\gamma;\delta} \tdelta,
\end{equation}
see \cite{ST,PPPP}.  In \cite[Proposition 2.22]{PPPP}, it is shown that any linear dependence \eqref{eq:anyexc} is a positive sum of linear dependencies of the form \eqref{eq:gnoninitial}.  Replacing a linear dependence \eqref{eq:gnoninitial} by the corresponding inequality \eqref{eq:conefacets}, we deduce that the inequality \eqref{eq:moreineq} is a positive sum of the inequalities \eqref{eq:conefacets}. 
\end{proof}
%

\begin{example}
Pick $(W,R^+,c)$ as in Example~\ref{ex:Aroot}.  Then the inequalities \eqref{eq:anyexc} are all of the form \eqref{eq:triangle} or \eqref{eq:anoncross}.
\end{example}

The following is a variant of the noncrossing condition of Lemma~\ref{lem:Talternating} .  
\begin{corollary}\label{cor:exchange}
Let $P$ be a generic simple $(W,c)$-prepolypositroid and $F$ be a face of $P$.  Let $S(F) \subset \tR$ be as defined in Section~\ref{ssec:facetwist}.  Then for any choice of $R^+$ compatible with $c$, and any exchangeable pair $(\tbeta,\gamma) \in \Gamma(c)$, we have that $S(F)$ can contain at most one of $\tbeta$ and $\tgamma$.
\end{corollary}
\begin{proof}
Suppose we have an exchangeable pair $\tbeta,\tgamma \in S(F)$.  Let $x \in F$.  Then by \eqref{eq:anyexc} we have
$$
a_{\tbeta} + a_{\tgamma} =(x,\tbeta + \tgamma) = (x, \sum_{\tdelta \in E(\tbeta,\tgamma)} c_{\beta,\gamma;\delta} \tdelta) \leq \sum_{\tdelta \in E(\tbeta,\tgamma)} c_{\beta,\gamma;\delta} a_{\tdelta}.
$$
By \eqref{eq:moreineq}, we must have equality, giving $\tdelta \in S(F)$ for $\tdelta \in E(\tbeta,\tgamma)$.  But then $S(F)$ contains twisted roots that are not linearly independent, contradicting the assumption that $P$ is generic simple.
\end{proof}
Corollary \ref{cor:exchange} has the following defect: while $(W,c)$-prepolypositroids depend only on the choice of $c$, the notion of an exchangeable pair $(\tbeta,\tgamma)$ depends additionally on a choice of $R^+$.  We thus pose the following question: 
\begin{question}
Which pairs of $c$-twisted roots is exchangeable for {\it some} choice of $R^+$ compatible with $c$?
\end{question}

\begin{corollary}\label{cor:manyassociahedron}
Let $P$ be a generic simple $(W,c)$-prepolypositroid, and $R^+$ be a positive system compatible with $c$.  Then removing the facets $(x,\omega) \leq a_\omega$ indexed by facet normals $\omega \notin \Pi(c)$ gives a $(W,R^+,c)$-generalized associahedron.
\end{corollary}
\begin{proof}
Removing the stated facets gives the polytope cut out by the inequalities $(x,\omega) \leq a_\omega$ for $\omega \in \Pi(c)$.  These $a_\omega$ satisfy 
\begin{equation}
a_{c^{m-1}\omega_k} + a_{c^m\omega_k} > \sum_{k \to i} -A_{ik} a_{c^m \omega_i} + \sum_{i \to k} -A_{ik} a_{c^{m-1} \omega_i},
\end{equation}
for each $ 1 \leq m \leq M(k)$.  According to \cite[Theorem 2.23]{PPPP}, these inequalities cut out the {\it deformation cone} of the $(W,R^+,c)$-generalized associahedron.  In other words, the inequalities $(x,\omega) \leq a_\omega$, $\omega \in \Pi(c)$ define a $(W,R^+,c)$-generalized associahedron.
\end{proof}

On the other hand, not every maximal cone in the normal fan of a generic simple $(W,c)$-prepolypositroid $P$ is a maximal cone in some $(W,R^+,c)$-cluster fan, as the following example shows.

\begin{example}
We continue the example from Section~\ref{ssec:D4ex}.
By slightly perturbing the $W$-balanced pair associated to $\v$, we obtain a generic simple $(W,c)$-prepolypositroid $Q(\v')$ whose normal fan is a refinement of that of $Q(\v)$ but not a refinement of the $W$-Coxeter fan.  There is a vertex $v$ of $Q(\v')$ with $S(v)$ given by
$$
S(v) = \{(1,0,0,0),(0,0,1,0),(-\tfrac{1}{2},-\tfrac{1}{2},\tfrac{1}{2},-\tfrac{1}{2}),(0,-1,0,0)\}.
$$
The roots in $S(v)$ are pairwise $c$-noncrossing in the sense of Section~\ref{ssec:cnoncross}.  The dual cone is spanned by the vectors
$$
\{(0, -1, 0, 1),(1, 0, 0, -1), (0, 0, 1, 1), (0, 0, 0, -2)\},
$$
the last of which is not in a direction of a root.  Since the $(W,R^+,c)$-cluster fan is a refinement the $W$-Coxeter fan (Theorem~\ref{thm:genass}), the vertex cone $C_v$ in the normal fan of $P$ is not a maximal cone for any cluster fan associated to $(W,c)$.
\end{example}

\section{Normal fans of $(W,c)$-prepolypositroids}\label{sec:Wnormal}


%

\subsection{Coxeter noncrossing roots}\label{ssec:cnoncross}
Recall that a pair of distinct roots $(\beta,\gamma) \in R$ is said to be {\it alternating} if $(\beta,\gamma) = (\gamma,\beta)\geq 0$.  Let us say that $(\beta,\gamma) \in R$ are \emph{$c$-noncrossing} if either $( \gamma^\vee,\tbeta) = 0$ or $(\beta^\vee,\tgamma) = 0$.  We say that $(\tbeta,\tgamma) \in \tR$ are alternating (resp. $c$-noncrossing) if $(\beta,\gamma) \in R$ are.  It is straightforward to see that with the choices in Example~\ref{ex:Aroot}, ``alternating" and ``$c$-noncrossing" agrees with the corresponding notions in Part~\ref{part:1}.

\begin{lemma}\label{lem:ceasy}
Two roots $\beta,\gamma$ are alternating (resp. $c$-noncrossing) if and only if $c\beta,c\gamma$ are alternating (resp. $c$-noncrossing).
\end{lemma}

\begin{lemma}\label{lem:noncrossingchange}
Let $c$ be a Coxeter element and $c' = wcw^{-1}$.  Then 
$(\beta,\beta')$ are $c$-noncrossing if and only if $(w\beta,w\beta')$ are $c'$-noncrossing.
\end{lemma}

\begin{conjecture}\label{conj:twodim}
Let $P$ be a generic simple $(W,c)$-prepolypositroid and suppose that $(\tbeta,\tgamma) \in \tR$ span a two-dimensional face of the normal fan of $P$.  Then $(\beta,\gamma)$ must be alternating and $c$-noncrossing.
\end{conjecture}
By Lemma~\ref{lem:Walternating}, the alternating part of Conjecture~\ref{conj:twodim} holds.
By Lemma~\ref{lem:Talternating}, Conjecture~\ref{conj:twodim} holds when $R$ is of type $A$.  We will show in Proposition~\ref{prop:noncrosscluster} that the condition ``alternating and $c$-noncrossing" is essentially the same as cluster compatibility.  Thus Conjecture~\ref{conj:twodim} is consistent with Corollary~\ref{cor:exchange}, since exchangeable pairs of cluster variables are incompatible (and the converse holds in type $A$).

\begin{remark}
The notion of $c$-noncrossing depends only on the choice of $c$, and not of $R^+$.  Furthermore, $(\beta,\gamma)$ is $c$-noncrossing if and only if $(-\beta,\gamma)$, $(\beta,-\gamma)$, and $(-\beta,-\gamma)$ are.  This is consistent with our usage of ``noncrossing" in type $A$ for directed edges: two directed edges are noncrossing if the underlying undirected edges are.
\end{remark}

\subsection{Reflection factorizations}
For $w \in W$, write $\ell_R(w)$ for the length of the shortest factorization of $w$ into reflections $s_\gamma \in W$, $\gamma \in R$.  We define a partial order $\leq_R$ on $W$ by
$$
u \leq v \qquad \Leftrightarrow \qquad \ell_R(v) = \ell_R(u) + \ell_R(v u^{-1}).
$$
Note that $\ell_R$ and $\leq_R$ do not depend on the choice of $R^+$.  It is well-known that for any Coxeter element $c$, we have $\ell_R(c) = r$.   We refer the reader to \cite{Bes,BW} for general background on reflection factorizations and reflection order.

%

\subsection{Bipartite positive systems}
We say that $(R^+,c)$ is {\it bipartite}, or $c$ (resp. $R^+$) is bipartite with respect to $R^+$ (resp. $c$), if $R^+$ is compatible with $c$ and in addition, 
\begin{equation}\label{eq:bipartitec}
c = \tau_+ \tau_-, \qquad \tau_+ =  s_{i_1} \cdots s_{i_t}, \qquad \text{and} \qquad \tau_- = s_{i_{t+1}} \cdots s_{i_r},
\end{equation}
where the partition $I = \{1,2,\ldots,r\} = I_+ \sqcup I_-$, with $I_+:= \{i_1, \ldots, i_t\}$ and $I_-:= \{i_{t+1},\ldots,i_r\}$ makes the Dynkin diagram bipartite.  

\begin{lemma}\label{lem:bipartitepair}
Let $c$ be a fixed Coxeter element.  Then a choice of $R^+$ bipartite with respect to $c$ exists.  There are exactly $2h$ bipartite $(R')^+$ with respect to $c$, and they are of the form $\tau_+ \tau_- \tau_+ \cdots \tau_- \tau_+ R^+$ or  $\tau_- \tau_+ \cdots \tau_- \tau_+ R^+$.  
\end{lemma}
\begin{proof}
There are exactly two orientations of the Dynkin diagram that correspond to bipartite $(R^+,c)$.  The claim thus follows from \cite[Theorem 3.6]{KT}.  This theorem is stated for simply-laced root systems, but the statement and proof are valid also for multiply-laced Weyl groups.
\end{proof}


%

When $(R^+,c)$ is bipartite, the ordering of Proposition~\ref{prop:listroots} induces an ordering of $R$ of the form $(A_1 \prec A_2 \prec \cdots \prec A_{2h}=A_0)$ where $|A_i| = t$ or $|A_i|=r-t$ depending on whether $i$ is odd or $i$ is even.  The ordering within each $A_i$ depends on the choice of a reduced word of $c$, but the sets $A_i$ themselves do not. We have
$$
R^+ = \bigsqcup_{i=1}^h A_i \qquad R^- = \bigsqcup_{i=h+1}^{2h} A_i \qquad R_{\geq -1} = \bigsqcup_{i=0}^{h+1}A_i.
$$
By Lemma~\ref{lem:bipartitepair}, the ordering of $R$ corresponding to another bipartite $(R')^+$ is of the form $(A_k \prec A_{k+1} \prec A_{k+2} \prec \cdots \prec A_{k-1})$.

We say that $\beta,\gamma \in R$ are {\it $c$-opposed} if $\beta \in A_i$ and $\gamma \in A_{i+h}$ for some choice of bipartite $R^+$ with respect to $c$.  The notion of $c$-opposed does not depend on the choice of bipartite $R^+$.


If $R^+$ and $c$ are fixed, we write $\alpha \prec_{R^+,c} \alpha'$ (or simply $\alpha \prec \alpha'$) if $\alpha$ precedes $\alpha'$ in the ordering of $R$ from Proposition~\ref{prop:listroots}.

\begin{lemma}[{\cite[Lemma 3.9]{BW}}]\label{lem:BW}
Suppose that $\beta \prec_{R^+,c} \gamma$ are distinct positive roots.  Then the following are equivalent:
\begin{enumerate}
\item $s_\beta s_\gamma \leq_R c^{-1}$;
\item $(\beta^\vee, \tgamma) = 0$.
\end{enumerate}
\end{lemma}

\begin{lemma}[{\cite[Lemma 5.6]{BW}}]\label{lem:BW2}
Suppose that $\beta \prec_{R^+,c} \gamma$ are distinct positive roots.  Then: 
\begin{enumerate}
\item
if $s_\beta s_\gamma \leq_R c^{-1}$ we have $(\beta,\gamma) \geq 0$;
\item
if $s_\gamma s_\beta \leq_R c^{-1}$ we have $(\beta,\gamma) \leq 0$.
\end{enumerate}
\end{lemma}

\begin{lemma}\label{lem:copposed}
Suppose that $\beta,\gamma \in R$ are $c$-opposed and $\beta \neq -\gamma$.  Then: 
\begin{enumerate}
\item $(\beta,\gamma) = (\gamma^\vee,\tbeta) = (\beta^\vee,\tgamma) = 0$;
\item $s_\beta s_\gamma =s_\gamma s_\beta \leq_R c^{-1}$.
\end{enumerate}
\end{lemma}
\begin{proof}
We have $A_{i+h} = -A_i$.  Suppose that $\beta$ and $\gamma$ are $c$-opposed.  Then we may choose bipartite $R^+$ so that $\beta \in A_1$ and $\gamma \in A_{1+h}$.  Then $\beta = \alpha_i$ and $\gamma = -\alpha_j$ where $i,j \in I_+$. Since $i,j$ are not adjacent, $(\beta,\gamma)=0$ follows.  Also $\tbeta = \omega_i$ and $\tgamma = -\omega_j$, so (1) follows.  (2) is also clear from \eqref{eq:bipartitec}.
\end{proof}

\begin{proposition}\label{prop:noncrossfact}
Let $(R^+,c)$ be bipartite and $\beta \prec_{R^+,c} \gamma$ be distinct positive roots.  Then $s_\beta s_\gamma \leq_R c^{-1}$ if and only if $\beta,\gamma$ are alternating $c$-noncrossing.
%
\end{proposition}

\begin{proof}
If $s_\beta s_\gamma \leq_R c^{-1}$, then by Lemma~\ref{lem:BW2} we have $(\beta,\gamma) \geq 0$, and by Lemma~\ref{lem:BW} we have $(\beta^\vee,\tgamma)=0$.  Thus $(\beta,\gamma)$ is alternating $c$-noncrossing.  

Conversely, suppose that $(\beta,\gamma)$ is alternating and $c$-noncrossing.  If $(\beta^\vee,\tgamma) =0$, then by Lemma~\ref{lem:BW} we have $s_\beta s_\gamma \leq c^{-1}$, and we are done.  Next suppose that we have $(\gamma^\vee, \tbeta) = 0$.  We claim that $(\beta,\gamma) = 0$.  To see this, we assume that $R^+$ has been chosen so that $\beta \in A_1$ while $\gamma \in \bigcup_{i=1}^h A_i$.  Thus $\beta = \alpha_i$ and $\tbeta= \omega_i$.  The condition $(\gamma^\vee,\omega_i) = 0$ implies that $(\gamma^\vee,\alpha_i) \leq 0$ and the alternating condition gives $(\gamma^\vee,\alpha_i) \geq 0$.  Thus, $(\beta,\gamma) = 0$, establishing our claim.

We are thus in the situation that $(\gamma^\vee, \tbeta) = 0$ and $(\beta,\gamma) = 0$.  If $\beta,\gamma \in A_i$ for some $i$ (that is, they are close together in the ordering), then by Lemma~\ref{lem:consectree}, we have $\prod_{\delta \in A_i}s_\delta  \prod_{\delta' \in A_{i+1}} s_{\delta'} = c^{-1}$, so we know that $s_\beta s_\gamma \leq c^{-1}$ and we are done.  Now, if $\beta \in A_i$ then $-\beta \in A_{i+h}$.
Thus, we may find a different positive system $(R')^+$, bipartite with respect to $c$, so that $\gamma \prec_{(R')^+,c} -\beta$, and $\gamma,-\beta \in (R')^+$.  By Lemma~\ref{lem:BW}, we have $s_\gamma s_\beta \leq c^{-1}$, and since $s_\gamma s_\beta= s_\beta s_\gamma$, we are done.
\end{proof}

\begin{corollary}\label{cor:cnoncrossing}
Suppose that $\beta,\gamma$ are distinct roots such that $\beta \neq -\gamma$.  Then the following are equivalent:
\begin{enumerate}
\item $\beta$ and $\gamma$ are $c$-noncrossing;
\item
either $s_\beta s_\gamma \leq_R c^{-1}$ or $s_\gamma s_\beta \leq_R c^{-1}$.
\end{enumerate}
\end{corollary}
\begin{proof}
Suppose (1) holds.  Replacing $\gamma$ by $-\gamma$ does not change either of the conditions (1) or (2).  Thus we may assume that $\beta, \gamma$ are alternating $c$-noncrossing.  If $\beta,\gamma$ are $c$-opposed, we apply Lemma~\ref{lem:copposed}.  Otherwise, we pick any bipartite $R^+$ containing $\beta$ and $\gamma$ and apply Proposition~\ref{prop:noncrossfact}.

Conversely, suppose (2) holds, for concreteness let us assume that $s_\beta s_\gamma \leq_R c^{-1}$.  If $\beta,\gamma$ are $c$-opposed, we apply Lemma~\ref{lem:copposed}.  Otherwise, we can find $R^+$ such that either $\beta \prec_{R^+,c} \gamma$ are both positive roots, or $-\beta \prec_{R^+,c} \gamma$ are both positive roots.  By Proposition~\ref{prop:noncrossfact}, in both cases (2) holds.
\end{proof}
%

\subsection{Cluster compatibility and Coxeter-noncrossing}\label{ssec:ccompat}

Let 
$$R_{\geq -1} := R^+ \cup \{-\alpha_1,\ldots,-\alpha_r\}$$ denote the set of {\it almost simple roots}.  The notion of (cluster) compatibility of a pair of almost simple roots $\beta,\beta' \in R_{\geq -1}$ is defined in \cite{FZ}, and it is related to the notion of compatibility of a pair $\omega,\omega' \in \Pi(c)$ by \cite[(5.6)]{YZ}.  Namely, 
let $\psi: \Pi(c) \to R_{\geq -1}$ be defined by
\begin{equation}\label{eq:YZpsi}
\psi(\omega) = \begin{cases} -\alpha_i & \mbox{if $\omega = \omega_i$ for some $i=1,2,\ldots,r$} \\
c^{-1} \omega - \omega = c^{-1} (I-c) \omega & \mbox{otherwise.}
\end{cases}
\end{equation}
Then 
$$
\omega,\omega' \in \Pi(c) \text{ are compatible if and only if } \psi(\omega),\psi(\omega') \in R_{\geq -1}  \text{ are.}
$$
\begin{theorem}[{\cite[Theorem 8.3]{BW}}]\label{thm:BW}
Suppose that $\beta \prec_{R^+,c} \gamma$ are two distinct positive roots.  Then
$(\beta,\gamma)$ are $(W,R^+,c)$-compatible if and only if $s_\beta s_\gamma \leq_R c^{-1}$.
\end{theorem}

\begin{proposition}\label{prop:noncrosscluster}
Suppose that $(R^+,c)$ is bipartite and $\beta,\gamma \in R^+.$
Then $(\beta,\gamma)$ is alternating $c$-noncrossing if and only if they are $(W,R^+,c)$-compatible.
\end{proposition}
\begin{proof}
Let us suppose that $\beta \prec \gamma$.  By Theorem~\ref{thm:BW}, we must show that $\beta,\gamma$ are alternating $c$-noncrossing if and only if $s_\beta s_\gamma \leq_R c^{-1}$.  This follows from Proposition~\ref{prop:noncrossfact}.
\end{proof}

Proposition~\ref{prop:noncrosscluster} says that the condition ``alternating and $c$-noncrossing" is an approximation to the notion of cluster compatibility that does not depend on a choice of $R^+$.


\subsection{Coxeter noncrossing trees}

An $r$-tuple $T=(\gamma_1,\ldots,\gamma_r)$ of roots is called a {\it tree} if they form a basis of $V$. 

\begin{definition}
A \emph{$c$-noncrossing tree} is an ordered sequence $T=(\gamma_1,\ldots,\gamma_r)$ of roots such that 
$$
s_{\gamma_1} s_{\gamma_2} \cdots s_{\gamma_r} = c^{-1}.
$$
Let $\T_{W,c}$ denote the set of $c$-noncrossing trees.
\end{definition}

The terminology is justified by the following result.
\begin{lemma}
A $c$-noncrossing tree $T=(\gamma_1,\ldots,\gamma_r)$ is an ordered basis of $V$. 
\end{lemma}
\begin{proof}
Suppose the roots $\gamma_1,\ldots,\gamma_r$ span a proper subspace $W \subsetneq V$.  Then any vector in the orthogonal complement $W^\perp \subset V$ (with respect to $(\cdot,\cdot)$) will be invariant under $c^{-1}$.  This contradicts the fact that $c^{-1}$ does not have the eigenvalue one; see Lemma~\ref{lem:invertible}.
\end{proof}

\begin{remark}
For a $c$-noncrossing tree $T=(\gamma_1,\ldots,\gamma_r)$, define the operation of {\it $i$-th sign reversal}
$$
T\mapsto (\gamma_1,\ldots,\gamma_r);
$$
and the operation of {\it $j$-th conjugation\/} 
$$ 
T \mapsto
(\gamma_1,\dots,\gamma_{j-1}, s_{\gamma_j}(\gamma_{j+1}), \gamma_j, \gamma_{j+2},\dots,
\gamma_r).
$$
These operations transform $c$-noncrossing trees into $c$-noncrossing trees.  It follows from the results of Deligne \cite{Del} and Bessis \cite{Bes} that any two $c$-noncrossing trees are related by repeated application of sign reversal and conjugation.  Furthermore, the conjugation actions give an action of the braid group.
\end{remark}

\begin{lemma}\label{lem:consectree}
Let $(\beta_1,\beta_2,\ldots,\beta_{hr})$ denote the ordering of $R$ of Proposition~\ref{prop:listroots}.  Then for any $i$, we have that $(\beta_i,\beta_{i+1},\ldots,\beta_{i+r-1})$ is a $c$-noncrossing tree.
\end{lemma}
\begin{proof}
By Proposition~\ref{prop:compatible}, it suffices to show this for $i = 1$.  Suppose $c = s_1s_2 \cdots s_r$.  We calculate
$$
s_{\beta_1} \cdots s_{\beta_r} = s_1 (s_1 s_2 s_1) (s_1 s_2 s_3 s_2 s_1) \cdots (s_1 \cdots s_r \cdots s_1) = s_r s_{r-1} \cdots s_1 = c^{-1}. 
$$
\end{proof}

According to Corollary~\ref{cor:cnoncrossing}, the ``$c$-noncrossing" condition characterizes when a pair of roots can belong to a $c$-noncrossing tree.  However, in general this pairwise condition is insufficient to characterize $c$-noncrossing trees.

\begin{example}
Let $R = B_3$ with simple roots $\alpha_1 = (1,-1,0)$, $\alpha_2 = (0,1,-1)$, and $\alpha_3 = (0,0,1)$, and choose $c = s_1s_2s_3$.  Take the three roots 
$$
\gamma_1 = (-1,0,-1), \qquad \gamma_2 = (-1,1,0), \qquad \gamma_3 = (0,1,-1),
$$
with corresponding $c$-twisted roots
$$
\tgamma_1 = (0,0,-1), \qquad \tgamma_2 = (-1,0,0), \qquad \tgamma_3 = (0,1,0).
$$
The roots $\gamma_1,\gamma_2,\gamma_3$ are pairwise $c$-noncrossing (and also alternating).  However, all three roots are long, so no ordering of them can give a reflection factorization of $c^{-1}$.
\end{example}

\begin{definition} 
Let $T=(\gamma_1,\ldots,\gamma_r)$ be a $c$-noncrossing tree.  Then the \emph{dual tree} $T' = \varphi(T) $ is given by 
$$T' := (s_{\gamma_1} \cdots s_{\gamma_{r-1}} \gamma_r, \ldots, s_{\gamma_1} \gamma_2,  \gamma_1).$$  
Also, define $T'' = \varphi^{-1}(T)$ by
$$
T'':=(\gamma_r, s_{\gamma_r} \gamma_{r-1},\ldots, s_{\gamma_r} \cdots s_{\gamma_2} \gamma_1).
$$
\end{definition}

\begin{proposition}
Let $T$ be a $c$-noncrossing tree.  Then the trees $T' = \varphi(T)$ and $T'' = \varphi^{-1}(T)$ are $c$-noncrossing tree.  The maps $\varphi$ and $\varphi^{-1}$ are inverse bijections from $\T_{W,c}$ to $\T_{W,c}$. 
\end{proposition}
\begin{proof}
Let $T = (\gamma_1,\ldots,\gamma_r)$ and $T' = (\gamma'_1,\ldots,\gamma'_r)$.
We have 
$$
c^{-1} = s_{\gamma_1} s_{\gamma_2} \cdots s_{\gamma_r}  = s_{\gamma_2} (s_{\gamma_2} s_{\gamma_1} s_{\gamma_2}) s_{\gamma_3} \cdots s_{\gamma_r} = \cdots = s_{\gamma'_1} \cdots s_{\gamma'_r}.
$$
The proof for $T''$ is similar, and it is straightforward to see that $\varphi$ and $\varphi^{-1}$ are inverse.
\end{proof}

For a tree $T=(\gamma_1,\ldots,\gamma_r)$, let $C_T \subset V$ denote the cone spanned by $\tgamma_1,\ldots,\tgamma_r$, and let $C'_T \subset V$ denote the cone spanned by $\gamma_1^\vee,\ldots,\gamma_r^\vee$.  The following result is a general root-system theoretic version of Proposition~\ref{prop:dualcones}.

\begin{proposition}
Let $T$ be a $c$-noncrossing tree and $T' = \varphi(T)$.  Then the two cones $C_T$ and $C'_{T'}$ are dual. 
\end{proposition}
\begin{proof}
Let $T' = (\delta_1,\ldots,\delta_r)$ and let $\kappa_1,\kappa_2,\ldots,\kappa_r$ be the dual basis to $\delta_1^\vee,\ldots,\delta_r^\vee$, i.e., $(\delta_i^\vee,\kappa_j) = \delta_{ij}$.  Then 
$$
c\cdot \kappa_j = s_{\delta_r} \cdots s_{\delta_1} \cdot \kappa_j = \kappa_j - s_{\delta_r} \cdots s_{\delta_{j+1}} \delta_j.
$$
Thus $(I-c)\kappa_j = s_{\delta_r} \cdots s_{\delta_{j+1}} \delta_j$, so the dual cone to $C'_{T'}$ is given by $C_{\varphi^{-1}(T')}$.
\end{proof}

\begin{corollary}
Suppose that $P$ is a simple $(W,c)$-twisted alcoved polytope, and all maximal cones of the normal fan of $P$ are of the form $C_T$ for a $c$-noncrossing tree $T$.  Then $P$ is a generalized $W$-permutohedron and thus a $(W,c)$-polypositroid.
\end{corollary}

\begin{question}\label{q:ctree}
Let $P$ be a generic simple $(W,c)$-polypositroid.  Are all maximal cones of the normal fan of $P$ of the form $C_T$ for a $c$-noncrossing tree $T$?
\end{question}
Question \ref{q:ctree} has an affirmative answer in type $A$.  See Section~\ref{ssec:typeAtree}.

\subsection{Cluster cones}
A \emph{$(W,c)$-cluster cone} is a maximal cone $C \subset V$ in the $(W,R^+,c)$-cluster fan, for some choice of positive roots $R^+ \subset R$.  It follows from the results of Brady-Watt \cite{BW} that some cluster cones are of the form $C_T$ for a $c$-noncrossing tree $T$, though we do not know whether this is true in general.

\begin{proposition}
Suppose that $(R^+,c)$ is bipartite.  Let $C ={\rm span}(\tgamma_1,\ldots,\tgamma_r)$ be a cluster cone such that $\{\tgamma_1,\ldots,\tgamma_r\} \cap \{\omega_1,\ldots,\omega_r\} = \emptyset$.  Then there is an ordering of $(\tgamma_1,\ldots,\tgamma_r)$ that gives a $c$-noncrossing tree $T$.
\end{proposition}
\begin{proof}
As a simplicial complex on the set of rays, the $(W,R^+,c)$-cluster fan is isomorphic, via the bijection \eqref{eq:YZpsi} to the cluster complex of \cite{FZ} defined on the set of almost positive roots $R_{\geq -1}$.  The condition $\{\tgamma_1,\ldots,\tgamma_r\} \cap \{\omega_1,\ldots,\omega_r\}  = \emptyset$ is equivalent to the condition that the almost positive roots $\psi(\tgamma_i)$ are positive.

According to \cite[Note 4.2 and Theorem 8.3]{BW}, a sequence $\delta_1 \prec \delta_2 \cdots \prec \delta_r$ of positive roots forms a simplex in the cluster complex if and only if $(\delta_1,\ldots,\delta_r)$ is a $c$-noncrossing tree.  The claim follows.
\end{proof}

\begin{remark}
Reading and Speyer \cite{RS} have found a {\it linear} isomorphism from the cluster fan of \cite{FZ} (with rays the almost positive roots), to the $(W,R^+,c)$-cluster fan, called the $g$-vector fan in \cite{RS}.  
\end{remark}

\subsection{Type $A$}\label{ssec:typeAtree}
We make explicit the relation between reflection factorizations of $c^{-1}$ and noncrossing trees (in type $A$).
\begin{lemma}\label{lem:Atree}
Let $T$ be a noncrossing (undirected) tree on $[n]$.  Then there is an ordering $e_1,\ldots,e_{n-1}$ of the edges so that
$s_{e_1} \cdots s_{e_{n-1}} = c = (12\cdots n)$.  Varying the possible orderings gives the same reduced factorization of $c$ up to commutation relations.
\end{lemma}
\begin{proof}
Let us draw $T$ in the interior of a disk with the vertices arranged in clockwise order on the boundary.
Let $f^{(i)}_1,f^{(i)}_2,\ldots,f^{(i)}_s$ be edges incident to the vertex $i$, in counterclockwise order.  An ordering $e_1,\ldots,e_{n-1}$ of all the edges of $T$ satisfies
$s_{e_1} \cdots s_{e_{n-1}} = c$ if and only if $\{f^{(i)}_1,\ldots,f^{(i)}_s\}$ appear in the same order in $e_1,\ldots,e_{n-1}$ for any vertex $i$.  Since $T$ contains no cycles, it is not difficult to see that such an ordering $e_1,\ldots,e_{n-1}$ exists.  The last statement follows from: $s_e$ and $s_{e'}$ commute if they have no vertex in common.
\end{proof}

\begin{example}
Let $T$ be the (solid) tree of Figure~\ref{fig:dualtrees}, with the five edges 
$$(1,2),(2,4),(2,6),(3,4),(5,6).$$  Then around the vertices $2$, $4$, and $6$, we obtain the counterclockwise orderings
$$
(1,2) < (2,6) < (2,4), \qquad (3,4) < (2,4), \qquad (5,6) < (2,6).
$$
One may check that 
$$
(123456) = (56) (12) (26) (34) (24) =  (12) (56) (26) (34) (24) =  (56) (12) (34) (26) (24)  = \cdots,
$$
consistent with Lemma~\ref{lem:Atree}.
\end{example}

\part{Membranes}

In this part we discuss membranes, which are certain triangulated
2-dimensional surfaces embedded into $\R^n$.
They can be viewed as a polypositroidal version of the plabic graphs
from \cite{TP}.

\section{Root loops and root membranes}
\label{sec:root_membranes}

Let $R$ be an irreducible reduced crystallographic root system of rank $r$
in a Eucledian vector space $V\simeq\R^r$;
and let $\Lambda\subset V$, $\Lambda\simeq \Z^r$, 
be the weight lattice for $R$;
see Section~\ref{sec:root_systems}.

\begin{definition}
A {\it plane graph\/} is a planar graph with a particular drawing on the 
plane without crossing edges, considered up to a homeomorphism.

A {\it cactus} $G$ is a finite connected undirected plane graph with 
at least 2 vertices such that every face of $G$ 
(except the outer face) is a triangle, 
i.e., every face has exactly 3 distinct vertices
connected by 3 edges.
In other words, a cactus is either a single edge, or a triangulated disk,
or a wedge of smaller cacti along their boundary vertices.

If a cactus $G$ is a wedge of smaller cacti then we say that
$G$ is {\it decomposable.}  Otherwise, we say that $G$ is 
{\it indecomposable.}


For a cactus $G$, there is a unique (up to a cyclic shift) sequence 
$B=(b_1,\dots,b_m)$  of {\it boundary vertices\/} 
connected by {\it boundary edges\/} $\{b_1,b_2\}$, $\{b_2,b_3\}$, \dots,
$\{b_m,b_1\}$
obtained by walking along the boundary of $G$ in the clockwise direction.
\end{definition}

\begin{remark}
If $G$ is a decomposable cactus, then the sequence $B$ has repeated entries.
The way a cactus decomposes into indecomposable cacti is given 
by a noncrossing set partition of $[m]$ without singleton blocks.
Blocks of this noncrossing set partition correspond to boundary vertices of
connected components of the dual plane graph $G^*$,
cf.\ Remark~\ref{rem:connected}.
\end{remark}



\begin{definition}  
(1)
An {\it $R$-loop\/} $L=(\lambda^{(1)},\dots,\lambda^{(m)})$ is a sequence 
of weights $\lambda^{(a)}\in \Lambda$, indexed by elements $a\in\Z/m\Z$,
 such that
$\lambda^{(a+1)}-\lambda^{(a)}\in R$, for any $a\in \Z/m\Z$.

\smallskip
\noindent
(2)
An {\it $R$-membrane\/} $M=(G,\emb)$ with {\it boundary loop\/}
$L=(\lambda^{(1)},\dots,\lambda^{(m)})$ is a cactus $G$ on a vertex set $\Vert$ 
with the sequence of boundary edges $B=(b_1,\dots,b_m)$
together with a (not necessarily injective) embedding map $\emb:\Vert\to \Lambda$ such that, 
\begin{itemize}
\item
$\emb(u)-\emb(v)\in R$, for any edge $\{u,v\}$ of $G$, 
and
\item
$\emb(b_a)= \lambda^{(a)}$, for any $a\in\Z/m\Z$.
\end{itemize}
(In particular, we require that $f(u)\ne f(v)$ for any edge $\{u,v\}$ of $G$.)

Equivalently, an $R$-membrane is a cactus $G$ together with a
graph homomorphism\footnote{A {\it graph homomorphism\/}
$f:G_1\to G_2$ from a graph $G_1=(V_1,E_1)$ to another graph $G_2=(V_2,E_2)$ is
a map $f:V_1\to V_2$ such that, for any edge $\{u,v\}\in E_1$,
we have $\{f(u),f(v)\}\in E_2$.}
$f$ from $G$ to the graph on the vertex set $\Lambda$ with edges
$\{\lambda,\mu\}$ for $\lambda-\mu\in R$.


\smallskip
\noindent
(3)
An {\it $R$-line segment\/} is a line segment 
$\conv(\lambda,\mu)\subset V$ where $\lambda-\mu \in R$ and 
an {\it $R$-triangle\/} is a triangle
$\conv(\lambda,\mu,\nu)\subset V$,
where $\lambda,\mu,\nu\in \Lambda$ such that 
$\lambda-\mu,\mu-\nu,\nu-\lambda\in R$.
\end{definition}

An $R$-loop $L$ can be viewed as a closed piecewise-liner curve $\<L\>$ in $V$,
and $R$-membranes $M$ with boundary loop $L$ can be viewed 
as 2-dimensional simplicial complexes embedded into $V$
as surfaces $\<M\>$ composed of $R$-triangles and $R$-line segments 
with a given boundary curve $\<L\>$, as follows.

For an $R$-loop 
$L=(\lambda^{(1)},\dots,\lambda^{(m)})$, let $\<L\>\subset V$ be the 
closed piecewise-linear curve given by the union of $R$-line segments
$$
\<L\>:=\bigcup_{a\in\Z/m\Z} [\lambda^{(a)},\lambda^{(a+1)}].
$$

Let $M=(G,\emb)$ be an $R$-membrane with boundary loop $L$.
For a face $\Delta$ of $G$ with vertices $u,v,w$,
let $\<\Delta\>:=\conv(\emb(u),\emb(v),\emb(w))\subset V$
be the corresponding $R$-triangle.
The triangulated surface $\left<M\right>\subset V$ associated with 
the membrane $M$ is given by the union 
$$
\left<M\right> :=
\<L\> \cup
\bigcup_{\Delta\textrm{ face of }G} \<\Delta\> \,.
$$




\begin{definition}
Let $\val:\{\textrm{$R$-triangles}\}\to \R_{>0}$
be any positive real function, or valuation,\footnote{There are
several natural valuations on $R$-triangles.  For example,
$\val\<\Delta\>$ can be the Euclidean area of triangle 
$\<\Delta\>\subset V$, or it can be the area of the projection 
of $\<\Delta\>$ to some plane,
or it can be $\val\<\Delta\>=1$ for all $R$-triangles.}
on the set of 
all $R$-triangles.
We say that an $R$-membrane $M$ with boundary loop $L$ is
{\it minimal,} with respect to the valuation $\val$,
if its surface area
$$
\Area\, M := \sum_{\Delta\textrm{ face of }G} 
\val \<\Delta\> 
$$
has minimal possible 
value among all membranes with the same boundary loop $L$.
\end{definition}

\begin{remark}
\label{rem:Plateau}
The famous {\it Plateau's problem\/} 
originally raised by Lagrange is the problem
in geometric measure theory concerning the existence of a minimal surface 
with a given boundary.   It was solved by Jesse Douglas \cite{Dou}
and  Tibor Rad\'o \cite{Rad}. 
We view the problem about characterization of minimal membranes $M$
with a given boundary loop $L$ as a discrete version of Plateau's problem.
Unlike the situation with its  continuous counterpart,  the 
existence of a minimal membrane is trivial.  There can be
many minimal membranes with a given boundary. 
However, we think that the characterization of minimal membranes might 
provide a better understanding of the continuous Plateau's problem.
\end{remark}


\section{Membranes of type A}
\label{sec:Mem_A}

Let us now specialize definitions from the previous section to type $A $.
Let $$
R=\{e_i - e_j\mid i,j\in[n],\ i\ne j\} \subset \R^n
$$ 
be the $A_{n-1}$ 
root system embedded in $\R^n$, and let $\Lambda \simeq \Z^n \subset \R^n$.  These are the root and weight lattices of $\GL(n)$.  Recall that $e_1,\ldots,e_n$ denote the
standard coordinate vectors in $\Z^n$.
We assume the valuation\footnote{In type $A$,
the Euclidean area of any $R$-triangle equals $\sqrt{3}$.
It is convenient to rescale it to $1$.}
of any $R$-triangle is
$\val \<\Delta\>=1$. 

In this case, we call $R$-loops and $R$-membranes  simply 
{\it loops\/} and {\it membranes.}
Let us formulate their definitions.

\begin{definition}
A {\it loop\/} $L=(\lambda^{(1)},\dots,\lambda^{(m)})$ 
is a
cyclically ordered sequence of integer vectors $\lambda^{(a)}\in \Z^n$ 
such that $\lambda^{(a+1)}-\lambda^{(a)} = e_{i_a}-e_{j_a}$, for $a\in \Z/m\Z$,
for some sequence of roots $e_{i_1}-e_{j_1},\ldots,e_{i_m}-e_{j_m}$.

\medskip

A {\it membrane\/} $M=(G,\emb)$ with {\it boundary loop\/}
$L=(\lambda^{(1)},\dots,\lambda^{(m)})$ 
is a cactus $G$ on a vertex set $\Vert$ 
with the sequence of boundary vertices $B=(b_1,\dots,b_m)$
together with 
a map $f:\Vert \to \Z^{n}$ such that
\begin{itemize}
\item
 for any edge
$\{u,v\}$ of $G$ there exist indices $i\ne j$ such that 
$\emb(u)-\emb(v) = e_i - e_j$,
\item
$\emb(b_a)=\lambda^{(a)}$, for any $a\in \Z/m\Z$.
\end{itemize}

\medskip



Let $\<L\>, \<M\>\subset \R^n$ denote the images of a loop $L$
and a membrane $M$ in $\R^n$.
\end{definition}

Note that here we do not require $m$ and $n$ to be equal.
Also note that the embedding $\<M\>\subset\R^n$ of a membrane lies
on some affine hyperplane $H=H_k :=
\{(x_1,\dots,x_n)\in\R^n\mid x_1+\cdots + x_n =k\}$, where  $k\in\Z$.
Clearly, a loop $L$ is determined (up to affine translation)
by a sequence of roots $e_{i_1}-e_{j_1},\ldots,e_{i_m}-e_{j_m}$ with
equal multisets of indices $\{i_1,\dots,i_m\}=\{j_1,\dots,j_m\}$.

All faces $\Delta$ in a membrane are of one of the following two types:
\begin{itemize}
\item
{\it black triangles\/} embedded into $\R^n$
as triangles $\<\Delta\>$ of the form $\conv(-e_i,-e_j,-e_k)$, up to an affine 
translation; and 
\item{\it white triangles} embedded as triangles $\<\Delta\>$ of the form 
 $\conv(e_i,e_j,e_k)$, up to an affine translation.
\end{itemize}

For a membrane $M=(G,f)$,  let $G^*$ be the graph which is 
the plane dual of the cactus $G$. The graph $G^*$ is drawn in a disk
so that
\begin{enumerate}
\item
There are $m$ marked points on the boundary of the disk 
(labelled $1,\dots,m$ clockwise) to which 
boundary edges of $G^*$ are attached.   Only one edge of $G^*$ 
can be attached
to a marked point on the disk.  But it is allowed that both ends
of an edge are attached to two different marked points on the boundary.
(The marked point on the boundary of the disk labelled 
$a$ corresponds to the boundary edge $\{b_a,b_{a+1}\}$ of the cactus $G$.
Note that we do not regard these $m$ marked boundary points 
as vertices of $G^*$.  The vertices of $G^*$ are 
located strictly inside the disk.)
\item
The vertices of $G^*$ are 3-valent.  
The vertices of $G^*$ are colored in 2 colors: black and white.
(The vertices of $G^*$ correspond to triangles of the membrane $M$.
They are colored according to the colors of triangles in $M$.
Note that the $m$ marked boundary points of the disk are not colored.)
\item 
The faces $F_v$ of $G^*$ (associated with vertices $v$ of $G$) 
are labelled by the vectors $f(v)\in\Z^n$.
For a pair of faces $F_u$ and $F_v$ sharing an edge, we 
have $f(u)-f(v)=e_i - e_j$, for some $i\ne j$. 
\end{enumerate}

Plane graphs $G^*$ satisfying conditions (1), (2) above
are basically (3-valent)  
plabic graphs \cite[Section~11]{TP}\footnote{``Plabic'' is an abbrevation for ``planar bicolored.''
In this work, we will consider only 3-valent plabic graphs without boundary leaves, which we simply call ``plabic graphs".  The study of more general plabic graphs can be reduced to these.}.

%

\begin{definition} 
A {\it plabic graph\/} is 
the plane dual $G^*$ of a cactus $G$
with all vertices colored in two colors: black and white.
(This graph may contain edges between 
vertices of the same color.)

A {\it $\Z^n$-labelled plabic graph\/} is a pair
$(G^*,f)$, where $G^*$ is the plane dual of a cactus $G$, such that
$M=(G,f)$ is a membrane.
Equivalently, a $\Z^n$-labelled plabic graph
is a pair $(G^*,f)$ satisfying conditions (1), (2), (3) above.
\end{definition}

Clearly, by the definition, membranes $M=(G,f)$ are in bijection 
with $\Z^n$-labelled plabic graphs $(G^*,f)$.

\begin{remark}  
\label{rem:connected}
For a  membrane $M=(G,f)$, the cactus $G$ is indecomposable
if and only if the plabic graph $G^*$ is connected.
\end{remark}

Let us give another description of membranes and 
$\Z^n$-labelled plabic graphs.
We recall the definition of strands (or trips) in plabic graphs.

\begin{definition}  \cite[Section~13]{TP} 
For a plabic graph $G^*$,
a {\it strand\/} in $G^*$ 
is a directed walk along the edges of $G^*$ that satisfies
the following ``rules of the road'':
\begin{itemize}
\item Turn right at a black vertex.
\item Turn left at a white vertex.
\end{itemize}
Each strand is either a walk between two marked points on the 
boundary of the disk, or a closed walk.

The {\it strand permutation\/} $\pi:[m]\to [m]$ of a plabic graph $G^*$ 
is given by $\pi(s)=t$, if the strand that starts at the marked
point $s$ on the boundary of the disk ends at the  marked point $t$.
\end{definition}

Let $\Strand(G^*)$ be the set of all strands in $G^*$.
For every edge $\{a,b\}$ of $G^*$, there are two strands in
$\Strand(G^*)$ that pass through the edge: one 
passing in the direction $a\to b$ and the other passing in the
direction $b\to a$.
We call such a pair of strands an {\it intersecting\/} pair of strands. 
If these two intersecting strands happen to be the same strand, we 
call it a {\it self-intersecting\/} strand.

\begin{theorem}  
\label{prop:Zn_labelled_plabic}
Let $G^*$ be a fixed plabic graph, and
let $F_0$ be a fixed reference face of $G^*$.
The set of all $\Z^n$-labelled plabic graphs $(G^*,f)$,
and thus all membranes $(G,f)$, 
are in bijection with the following data:
\begin{enumerate}
\item An integer vector in $\Z^n$, 
which is the label of the reference face $F_0$.
\item A map $g:\Strand(G^*)\to\{1,\dots,n\}$ that satisfies the condition:
\begin{equation}\label{eq:gST}
g(S)\ne g(T),\quad \textrm{for any pair of intersecting strands $S$ and $T$.} 
\end{equation}
\end{enumerate}
Explicitly, the strand labelling $g$ is obtained from the face
labelling $f$ by the following condition:
if $S\in\Strand(G^*)$ is the strand passing through some edge
$\{a,b\}$ of $G^*$ in the direction $a\to b$, and
$F_u$ and $F_v$ are the two adjacent faces of $G^*$ located, resp., 
to the left and to the right of the edge $a\to b$,
and if $f(u)-f(v)=e_i-e_j$, then $g(S) = i$.
\end{theorem}

In particular, \eqref{eq:gST} implies that, for any membrane $(G,f)$,
the plabic graph $G^*$ cannot have self-intersecting strands.
Changing the label of the reference face $F_0$ 
accounts for affine translations of membranes in $\R^n$.   
Up to affine translations,
membranes correspond just to strand labellings $g$ satisfying 
condition (2).

\begin{proof}  
First, it is easy to check, using the rules of the road, that the description 
of strand labelling $g$ in terms face labelling $f$ is locally consistent,
that is, for any vertex of $G^*$, the label of some strand passing through 
this vertex obtained using its incoming edge to the vertex coincides
with the label obtained using its outgoing edge.  This implies 
the global consistence of the strand labelling $g$:  For any strand $S$,
the label $g(S)$ obtained using any edge of $S$ does not depend 
on a choice of the edge.

Condition (2) for the strand labelling $g$ follows from the fact,
that, for a pair strands $S$ and $T$ intersecting at an edge of $G^*$ with 
two faces $F_u$ and $F_v$ adjacent to the edge, 
we have $g(S)=i$ and $g(T)=j$, where $f(u)-f(v)=e_i-e_j\ne 0$.

Conversely, let $v_0$ be the vertex of the cactus $G$ corresponding
to the reference face $F_0$ of $G^*$.  
Let $f(v_0)\in \Z^n$ be any integer vector, and let $g$ be 
any strand labelling satisfying condition (2).
For any vertex $v$ of the cactus $G$, we can construct 
the vector $f(v)\in \Z^n$ by picking a path $P$ in $G$ 
from $v_0$ to $v$ and using the relationship
between $f$ and $g$, for all edges of the path $P$.
The rules of the road imply that the label $f(v)$ does not change 
if we locally modify the path $P$ along a (triangular) 
face of the cactus $G$.  This implies 
the independence of this construction for $f(v)$
from a choice of path $P$.  Clearly, this function $f$, 
constructed from $f(v_0)$ and $g$, gives a valid membrane $(G,f)$.
\end{proof}

By Theorem~\ref{prop:Zn_labelled_plabic}, the strands of a $\Z^n$-labeled plabic graph are not self-intersecting.
\begin{corollary}
The strand permutation $\pi$ of a $\Z^n$-labeled plabic graph is a {\it derangement,}
i.e., a permutation in $S_m$ such that $\pi(s)\ne s$ for any $s\in[m]$.
\end{corollary}
In \cite{TP}, strand permutations were {\it decorated\/} permutations with 2 types
of fixed points; see Section~\ref{sec:positroids}.  Here we do not allow plabic graphs to have 
boundary leaves, so their strand permutations do not have fixed points.

%
%

\section{Moves of plabic graphs and membranes}
\label{sec:moves_membranes}

In \cite[Section~12]{TP}, the three types of local moves of plabic graphs
were defined, which are shown below on Figure~\ref{fig:3_plabic_moves}.

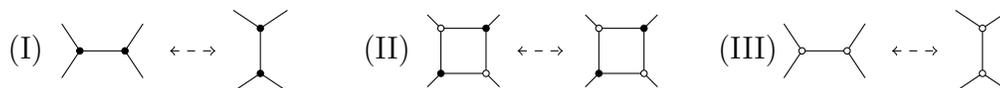
\begin{figure}[h]
\begin{center}
\begin{tikzpicture}[scale=0.6]
\node at (-1.2,0){(I)};
\filldraw[black] (0,0) circle (2pt);
\filldraw[black] (1,0) circle (2pt);
\draw (0,0)--(1,0);
\draw(0,0)--(-0.4,0.6);
\draw(0,0)--(-0.4,-0.6);
\draw(1,0)--(1.4,0.6);
\draw(1,0)--(1.4,-0.6);
\draw[dashed,<->] (2,0)--(3,0);
\begin{scope}[shift={(4,-0.5)}, rotate = 90]
\filldraw[black] (0,0) circle (2pt);
\filldraw[black] (1,0) circle (2pt);
\draw (0,0)--(1,0);
\draw(0,0)--(-0.4,0.6);
\draw(0,0)--(-0.4,-0.6);
\draw(1,0)--(1.4,0.6);
\draw(1,0)--(1.4,-0.6);
\end{scope}

\begin{scope}[shift={(8,0)}]
\node at (-1.2,0){(II)};

\draw (0,-0.5)--(1,-0.5)--(1,0.5)--(0,0.5)--(0,-0.5);
\draw (0,-0.5)--(-0.3,-0.8);
\draw (0,0.5)--(-0.3,0.8);
\draw (1,-0.5)--(1.3,-0.8);
\draw (1,0.5)--(1.3,0.8);
\filldraw[black] (0,-0.5) circle (2pt);
\filldraw[black] (1,0.5) circle (2pt);
\filldraw[fill=white] (0,0.5) circle (2pt);
\filldraw[fill=white] (1,-0.5) circle (2pt);
\draw[dashed,<->] (1.7,0)--(2.7,0);
\begin{scope}[shift={(4,-0.5)}, rotate = 90]
\draw (0,-0.5)--(1,-0.5)--(1,0.5)--(0,0.5)--(0,-0.5);
\draw (0,-0.5)--(-0.3,-0.8);
\draw (0,0.5)--(-0.3,0.8);
\draw (1,-0.5)--(1.3,-0.8);
\draw (1,0.5)--(1.3,0.8);
\filldraw[fill=white] (0,-0.5) circle (2pt);
\filldraw[fill=white] (1,0.5) circle (2pt);
\filldraw[black] (0,0.5) circle (2pt);
\filldraw[black] (1,-0.5) circle (2pt);
\end{scope}
\end{scope}

\begin{scope}[shift={(16,0)}]
\node at (-1.2,0){(III)};
\draw (0,0)--(1,0);
\draw(0,0)--(-0.4,0.6);
\draw(0,0)--(-0.4,-0.6);
\draw(1,0)--(1.4,0.6);
\draw(1,0)--(1.4,-0.6);
\filldraw[fill=white] (0,0) circle (2pt);
\filldraw[fill=white] (1,0) circle (2pt);
\draw[dashed,<->] (2,0)--(3,0);
\begin{scope}[shift={(4,-0.5)}, rotate = 90]
\draw (0,0)--(1,0);
\draw(0,0)--(-0.4,0.6);
\draw(0,0)--(-0.4,-0.6);
\draw(1,0)--(1.4,0.6);
\draw(1,0)--(1.4,-0.6);
\filldraw[fill=white] (0,0) circle (2pt);
\filldraw[fill=white] (1,0) circle (2pt);
\end{scope}
\end{scope}
\end{tikzpicture}
\end{center}
\caption{Moves of plabic graphs: 
(I) contraction-uncontraction of black vertices,
(II) square move, 
(III) contraction-uncontraction of white vertices.}
\label{fig:3_plabic_moves}
\end{figure}
%
%
It is easy to see from the rules of the road that we have:

\begin{lemma} \cite[Lemma~13.1]{TP} Any two plabic graphs 
connected with each other by a sequence of local moves of 
types {\rm (I), (II),} or {\rm (III)}
have the same strand permutations.
\end{lemma}

The local moves of plabic graphs can be converted into local
moves of membranes, as follows.

\begin{lemma}
\label{lem:pyramid}
Let $M=(G,f)$ be a membrane.  Let $F_u$ be a square face of $G^*$ with 
vertices of alternating colors as we go arond $F_u$, i.e., a face of $G^*$
on which one can perform a square move {\rm (II).}
Let $F_v, F_w, F_z, F_t$ be the 4 adjacent faces of $G^*$ in the clockwise order.
Then $\conv(f(u), f(v), f(w), f(z), f(t))$ is a square pyramid in $\R^n$ 
such that
\begin{itemize}
\item 
The pyramid has one square face (the base) 
and 4 faces given by equilateral triangles.
All edges of the pyramid have equal lengths.
\item 
$f(u)$ is the apex of the pyramid.
\item
The base is the square $\conv(f(v), f(w), f(z), f(t))$
with vertices arranged as we go around the base.
\end{itemize}
\end{lemma}

For a pyramid, as in the lemma above, let $\widetilde{f(u)}\in\R^n$ be
the reflection of $f(u)$ with respect to the affine plane containing 
the points $f(v), f(w), f(z), f(t)$, i.e., it is given by 
\begin{equation}
\label{eq:tilde_f_u}
\widetilde{f(u)} + f(u) = f(v)+f(z) = f(w)+f(t).
\end{equation}
Clearly, $\conv(\widetilde{f(u)}, f(u), f(v), f(w), f(z),f(t))$
is an octahedron, which is the union of two square pyramids.
The following lemma follows from the definitions.

\begin{lemma}\label{lem:membranemove}
Let $M=(G,f)$ be a membrane, and let $G^*\to \tilde G^*$ be any 
local move of the plabic graph $G^*$ of type
{\rm (I),} {\rm (II),} or {\rm (III).}
Let $\tilde G$ be the plane dual of the plabic graph $\tilde G^*$.
The vertex set $\tilde \Vert$ of $\tilde G$ can be naturally identified
with the vertex set $\Vert$ of $G$.
Let $\tilde f:\tilde \Vert \to\Z^n$ be the map defined, as follows.
\begin{itemize}
\item  For a move of type {\rm (I)} or {\rm (III),} assume that  
\begin{equation}
f(u)\ne f(w), 
\label{eq:fufw}
\end{equation}
where $F_u$ and $F_w$ are two of the four faces $F_u,F_v,F_w,F_z$
of $G^*$ involved in the move, which are not adjacent faces of $G^*$.
Then define $\tilde f := f$.
\item For a square move {\rm (II)},
let $F_u,F_v,F_w,F_z,F_t$ be the faces of $G^*$ involved in the move, 
labelled as in Lemma~\ref{lem:pyramid},
then set 
$\tilde f(u)  := \widetilde{f(u)}$, given by {\rm (\ref{eq:tilde_f_u}).}
For all other $x\ne u$, set $\tilde f(x) :=f(x)$.
\end{itemize}

Then $\tilde M := (\tilde G, \tilde f)$ is a valid membrane
with the same boundary loop $L$ as $M$.
\end{lemma}

\begin{definition}
Local moves of membranes of types {\rm (I), (II), (III)} are
the moves $M\to \tilde M$ in Lemma~\ref{lem:membranemove}.
\end{definition}

Moves of types {\rm (I)} and {\rm (III)} correspond to 
{\it tetrahedron moves\/}
of membranes, where we replace 2 triangles on the surface
of a tetrahedron by the other 2 triangles.
Moves of type {\rm (II)} correspond to {\it octahedon moves\/} of membranes
where we replace 4 triangles forming a half of the surface
of an octahedron by the 4 triangles on the other half of the surface
of the octahedron, as shown in Figure~\ref{fig:tetra_octa}.  See also~\cite{FP}.

\begin{figure}[h]
\begin{center}
\begin{tikzpicture}[scale=0.6]
\node at (-2.5,0){(I) and (III)};
\filldraw[black] (0,0) circle (2pt);
\filldraw[black] (1,0) circle (2pt);
\draw (0,0)--(1,0);
\draw(0,0)--(-0.4,0.6);
\draw(0,0)--(-0.4,-0.6);
\draw(1,0)--(1.4,0.6);
\draw(1,0)--(1.4,-0.6);
\draw[dashed,<->] (2,0)--(3,0);
\begin{scope}[shift={(4,-0.5)}, rotate = 90]
\filldraw[black] (0,0) circle (2pt);
\filldraw[black] (1,0) circle (2pt);
\draw (0,0)--(1,0);
\draw(0,0)--(-0.4,0.6);
\draw(0,0)--(-0.4,-0.6);
\draw(1,0)--(1.4,0.6);
\draw(1,0)--(1.4,-0.6);
\end{scope}

\begin{scope}[shift={(0,-2.5)}]
\draw (-0.2,0)--(1.3,0);
\draw(-0.2,0)--(0.4,1)--(1.3,0);
\draw(0.4,1)--(0.1,-0.4)--(1.3,0);
\draw[dashed,<->] (2,0)--(3,0);
\begin{scope}[shift={(3.7,0)}]
\draw (-0.2,0)--(1.3,0);
\draw (0.1,-0.4)--(-0.2,0)--(0.4,1);
\draw(0.4,1)--(0.1,-0.4)--(1.3,0);
\end{scope}
\end{scope}

\begin{scope}[shift={(8,0)}]
\node at (-1.2,0){(II)};

\draw (0,-0.5)--(1,-0.5)--(1,0.5)--(0,0.5)--(0,-0.5);
\draw (0,-0.5)--(-0.3,-0.8);
\draw (0,0.5)--(-0.3,0.8);
\draw (1,-0.5)--(1.3,-0.8);
\draw (1,0.5)--(1.3,0.8);
\filldraw[black] (0,-0.5) circle (2pt);
\filldraw[black] (1,0.5) circle (2pt);
\filldraw[fill=white] (0,0.5) circle (2pt);
\filldraw[fill=white] (1,-0.5) circle (2pt);
\draw[dashed,<->] (1.7,0)--(2.7,0);
\begin{scope}[shift={(4,-0.5)}, rotate = 90]
\draw (0,-0.5)--(1,-0.5)--(1,0.5)--(0,0.5)--(0,-0.5);
\draw (0,-0.5)--(-0.3,-0.8);
\draw (0,0.5)--(-0.3,0.8);
\draw (1,-0.5)--(1.3,-0.8);
\draw (1,0.5)--(1.3,0.8);
\filldraw[fill=white] (0,-0.5) circle (2pt);
\filldraw[fill=white] (1,0.5) circle (2pt);
\filldraw[black] (0,0.5) circle (2pt);
\filldraw[black] (1,-0.5) circle (2pt);
\end{scope}
\begin{scope}[shift={(0,-2.5)}]
\draw (-0.3,0)--(0.8,0)--(1.2,0.4)--(0.1,0.4)--(-0.3,0);
\draw (0.3,1.4)--(-0.3,0);
\draw (0.3,1.4)--(0.8,0);
\draw (0.3,1.4)--(1.2,0.4);
\draw (0.3,1.4)--(0.1,0.4);
\draw[dashed,<->] (1.7,0)--(2.7,0);
\begin{scope}[shift={(3.7,0)}]
\draw (-0.3,0)--(0.8,0)--(1.2,0.4)--(0.1,0.4)--(-0.3,0);
\draw (0.3,-1)--(-0.3,0);
\draw (0.3,-1)--(0.8,0);
\draw (0.3,-1)--(1.2,0.4);
\draw (0.3,-1)--(0.1,0.4);
\end{scope}
\end{scope}
\end{scope}
\end{tikzpicture}
\end{center}
\caption{Moves of plabic graphs (top) and membranes (bottom): 
tetrahedon moves (left), and octahedron move (right).}
\label{fig:tetra_octa}
\end{figure}
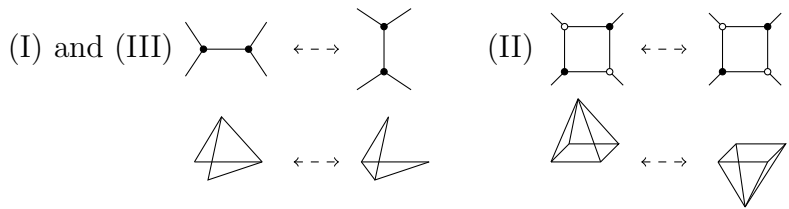

%
%
\begin{remark}
\label{rem:move_not_defined}
There is only one situation when there is a valid
a move $G^*\to \tilde G^*$ of plabic graphs, but there is no
corresponding move of membranes $M\to \tilde M$.   
This happens if the move of plabic graphs is of type (I) or (III)
and condition (\ref{eq:fufw}) fails.
In this case, in the picture shown on 
the bottom left of  Figure~\ref{fig:tetra_octa}, 
the two triangles before the move coincide with each other.
The move would transform them into two ``degenerate triangles'',
i.e., line segments, which we do not allow in a membrane.
(Recall that we require that $f(u)\ne f(v)$ for any edge $\{u,v\}$ of $G$.)

However, for any membrane $M=(G,f)$ and any square move (II) of the 
plabic graph $G^*$, there is always the associated valid 
octahedron move of the membrane $M$.  
\end{remark}

\begin{remark}
According to \cite{Po18} the three types of moves
of plabic graphs correspond to the three 3-dimensional hypersimplices: 
$\Delta_{14}$ (tetrahedron), 
$\Delta_{24}$ (octaheron), 
$\Delta_{34}$ (upside down tetrahedron).  Moves of certain 3-dimensional membranes were used in 
\cite[Section~12.6]{FP} to graphically describe 
``chain reactions'' on plabic graphs.
\end{remark}

Clearly, $\Area(M)$ is preserved under the local moves of membranes, and the
boundary loop $L$ does not change.  So the class of minimal membranes with a
given boundary loop $L$ is invariant under these 3 types of moves of membranes.

\begin{remark}\label{rem:minimalmove}
The exceptional case when condition (\ref{eq:fufw}) fails
and the move of membranes is not defined 
(see Remark~\ref{rem:move_not_defined})
can never happen in a minimal membrane.
Indeed, in this case, the two coinciding triangles in the picture 
on the bottom left of Figure~\ref{fig:tetra_octa} can be removed
from the membrane $M$, so that we get a membrane with a smaller area
but with the same boundary loop $L$.
\end{remark}

\section{Minimal membranes and reduced plabic graphs}
\label{sec:minimal_membranes_reduced}

We recall the notion of reducedness for plabic graphs.  The following definition is equivalent to
\cite[Definition~12.5]{TP} by \cite[Theorem~13.2]{TP}.

\begin{definition}
\label{def:reduced_plabic}
A plabic graph $G^*$ is {\it reduced\/} if it satisfies the conditions:
\begin{enumerate}
\item $G^*$ has no self-intersecting strands.
\item $G^*$ has no closed strands.
\item $G^*$ has no pair of strands $S,T$ with a 
{\it bad double crossing,} which means, that  
$S$ and $T$ intersect at two edges $\{a,b\}$ and $\{c,d\}$ and 
both strands are directed from $\{a,b\}$ to $\{c,d\}$.
\end{enumerate}
\end{definition}

\begin{theorem} 
\label{th:plabic_move_equiv}
\cite[Theorem~13.4]{TP}
For any two reduced plabic graphs with the same number of marked
points on the disk, the graphs have the same strand permutations
if and only if they can be obtained from 
each other by a sequence of local moves of types {\rm (I), (II), (III).}
\end{theorem}

Let us discuss a relationship between minimal membranes
and reduced plabic graphs.

\begin{theorem}
\label{th:minimal_membr=reduced}
If $M=(G,f)$ is a minimal membrane then $G^*$ is a reduced plabic graph.
Moreover, for a minimal membrane $M=(G,f)$ with boundary loop 
$L=(\lambda^{(1)},\dots,\lambda^{(m)})$ 
(where $\lambda^{(a+1)}-\lambda^{(a)} =
e_{i_a}-e_{j_a}$, for $a\in\Z/m\Z$), one can recover
the strand labelling $g$ (and thus the face labelling $f$) 
from the plabic graph $G^*$ and loop $L$ as follows:
If $S\in\Strand(G^*)$ is a strand in $G^*$ connecting two marked points
labelled $s$ (the source of $S$) and $t$ (the target of $S$) 
on the boundary of the disk,
then we have
$$
g(S) = i_s = j_t\,.
$$
\end{theorem}

\begin{proof}
Suppose that $M=(G,f)$ is a minimal membrane, but
$G^*$ is not a reduced plabic graph. 
According to \cite[Sections~12, 13]{TP}
one can apply a sequence of moves 
(I), (II), (III) to $G^*$ and obtain a plabic graph with a double 
edge.
By Remark~\ref{rem:minimalmove}, if $M=(G,f)$ is a 
{\it minimal\/} membrane, then, for any sequence of local moves of plabic
graphs starting with $G^*$, there is an associated sequence of 
local moves of membranes.
So we get some minimal membrane whose plabic graph has a double edge.
However, we have:
\begin{enumerate}
\item
There is no membrane (minimal or not) whose plabic graph 
has a double edge with vertices of different colors.
Indeed, such a plabic graph would have a self-intersecting strand,
which contradicts  Theorem~\ref{prop:Zn_labelled_plabic}.
\item 
If the plabic graph of a membrane has a double edge with vertices
of the same color, then the membrane is not minimal, because
we can remove two triangles from it (corresponding to the two vertices
of the plabic graph connected by the double edge), and get a membrane
of smaller area with the same boundary loop.
\end{enumerate}
In both cases, we get a contradiction, which proves the first part of
theorem.

The second part follows from the fact that a reduced plabic graph $G^*$
cannot contain a closed strand.  So any strand $S$ of $G^*$ connects two marked
points $s$ (source) and $t$ (target) on the boundary of the disk.
Applying Theorem~\ref{prop:Zn_labelled_plabic} 
(that relates 
the strand labelling $g$ to the face labelling of $f$ of a plabic graph)
to the first and the last edges of the strand $S$,
i.e., the edges of $S$ connected to the marked points $s$ and $t$
on the boundary of the disk,  
we get exactly the needed equality $g(S)= i_s = j_t$.
\end{proof}


In general, a non-minimal membrane
$M=(G,f)$ can give rise to a reduced plabic graph $G^*$.
However, for a special class of loops, minimality of $M$ is 
equivalent to reducedness of $G^*$.

Recall that a sequence $(c_1,\dots,c_m)$ of real numbers is {\it unimodal\/} 
if $c_1\leq \cdots \leq c_k \geq c_{k+1} \geq \cdots \geq c_m$, 
for some $k\in[m]$.
Let us say that a sequence $(c_1,\dots,c_m)$ 
is {\it cyclically unimodal\/} if it 
is unimodal up to a possible cyclic shift, i.e.,
if $(c_r, c_{r+1}, \dots , c_m, c_1,\dots,c_{r-1})$ is unimodal
for some $r\in [m]$.

\begin{definition}
\label{def:unimodal_loop}
We say that a loop $L=(\lambda^{(1)},\dots,\lambda^{(m)})$ is 
{\it unimodal\/} if each of its coordinate sequences
$(\lambda^{(1)}_i,\dots,\lambda^{(m)}_i)$, for $i\in[n]$, 
is a cyclically unimodal sequence.

Equivalently, a loop $L$ (with $\lambda^{(a+1)}-\lambda^{(a)}=e_{i_a}-e_{j_a}$, 
for $a\in\Z/m\Z$) is unimodal if there is 
no $4$-tuple of indices $a<b<c<d$ in $[m]$ 
such that $i_a=j_b=i_c=j_d$ or $j_a=i_b=j_c=i_d$.
\end{definition}


Consider a disk with $m$ marked points on its boundary 
labelled $1,\dots,m$ in the clockwise order.
For $s,t\in[m]$, let $|s,t|$ denote the chord in the disk that connects 
two marked boundary points labelled $s$ and $t$.
We say that two chords $|s,t|$ and $|s',t'|$ are 
{\it noncrossing\/}
if they do not intersect each other in the disk.

\begin{definition}   
\label{def:pi_L}
For a unimodal loop 
$L=(\lambda^{(1)},\dots,\lambda^{(m)})$ 
(with $\lambda^{(a+1)}-\lambda^{(a)}=e_{i_a}-e_{j_a}$ for $a\in\Z/m\Z$),
define $\pi=\pi_L$ as the unique permutation $\pi:[m]\to [m]$
such that 
\begin{itemize}
\item For any $s\in [m]$, we have $i_s=j_{\pi(s)}$. 
\item If $i_s = i_{s'}$, for some $s\ne s'\in [m]$,
then the two chords $|s, \pi(s)|$ and $|s',\pi(s')|$ are
noncrossing.
\end{itemize}

Explicitly, the permutation $\pi=\pi_L$ is given, as follows.
For any $i\in[n]$, let $s_1,\dots,s_p$ be all indices such that
$i_{s_1}=i_{s_2}=\cdots = i_{s_p}=i$; and let 
$t_1,\dots,t_p$ be all indices such that
$j_{t_1}=j_{t_2}=\cdots = j_{t_p}=i$.  Since $L$ is unimodal, we may assume 
that these indices are arranged 
as $s_1<s_2<\cdots<s_p< t_1<t_2<\cdots<t_p$ (up to a cyclic shift).
Then we have 
$\pi(s_1)=t_p$, $\pi(s_2)=t_{p-1}$, \ldots, $\pi(s_p)=t_1$.
\end{definition}


\begin{theorem}
\label{th:mem=plabic_separated}
Fix a unimodal loop $L$.
The map $M=(G,f)\mapsto G^*$ gives a bijection between 
the following two sets:
\begin{itemize}
\item The set of minimal membranes $M$ with boundary loop $L$. 
\item The set of reduced plabic graphs $G^*$ with strand permutation $\pi_L$.
\end{itemize}
All minimal membranes with boundary loop $L$ are connected by 
local moves of membranes of types {\rm (I), (II), (III).}
\end{theorem}

\begin{proof}  
Let $M=(G,f)$ be a minimal membrane with boundary loop $L$.
According to 
Theorem~\ref{prop:Zn_labelled_plabic}, two strands $S$ and $T$
of $G^*$ that have the same strand label $g(S)=g(T)$ cannot intersect
each other.
For $i\in[n]$, let $s_1,\dots,s_p, t_1,\dots,t_p$ be sequence 
of sources and targets as in 
Definition~\ref{def:pi_L}. 
  These sources should be connected
with the targets by strands $S_1,\dots,S_p$ 
that all have the same label $g(S_1)=\dots=g(S_p)=i$.
Since the sources and the targets are separated from each other
on the boundary of the disk,
there exists a unique matching between them, given by a collection 
of pairwise noncrossing strands.
Namely, the strand $S_a$ starting at the source $s_a$ should end
at the target $t_{p+1-a}$,
for $a=1,\dots,p$.  Thus the strand permutation of $G^*$ is exactly 
the permutation $\pi_L$ given by Definition~\ref{def:pi_L}.
Since we know that all reduced plabic graphs with the same strand
permutation are connected with each other by the local moves,
and there are corresponding local moves of membranes,
we deduce all the claims of theorem.
\end{proof}


Let us give explicit expressions for the surface area and 
the number of lattice points of a minimal 
membrane using the results of \cite{TP}.

%

\begin{definition} 
\label{def:anti_exceed_align}
\cite[Section~17]{TP} \ 
Let 
$\pi:[m]\to[m]$ be a derangement, i.e., a permutation such that
$\pi(a)\ne a$, for any $a$.   Define:
\begin{itemize}
\item The {\it number of anti-exceedances\/} in $\pi$
$$
k(\pi) := \#\{a\in[m]\mid \pi(a)<a\}.
$$
\item The {\it number of alignments\/} in $\pi$
$$
A(\pi) := \left\{(a,b)\in[m]^2\mid 
\begin{array}{l}
a<b<\pi(b) <\pi(a),   \textrm{ or }
\pi(a)<a<b<\pi(b),  \textrm{ or }
\\
\pi(b) <\pi(a) < a < b,\textrm{ or }
b<\pi(b) <\pi(a) < a
\end{array}
\right \}
$$
\end{itemize}
\end{definition}

Recall that the surface area $\Area(M)$ of a membrane 
$M=(G,f)$ is the number of faces (triangles) $\Delta$ of the cactus $G$,
or, equivalently, the number of vertices of the plabic graph $G^*$.
Also denote the number of {\it lattice points\/} of $M$ by
$$
\mathrm{LatticePoints}(M) :=\# (\left<M\right>\cap \Z^n) 
= \#\{\textrm{vertices of } G\}
= \#\{\textrm{faces of } G^*\}.
$$

\begin{proposition}[{cf.~\cite[Proposition~17.10]{TP}}]
\label{prop:alignments_formula_for_area}

Let $L$ be a unimodal loop, and let $\pi=\pi_L:[m]\to [m]$ 
be the associated permutation; see Definition~\ref{def:pi_L}.
The number of lattice points and 
the surface area of 
a minimal membrane $M$ with boundary loop $L$ are equal to 
$$
\begin{array}{l}
\mathrm{LatticePoints}(M) = k(\pi)(m-k(\pi))-A(\pi) + 1,\\[.05in]
\Area(M) = 2\left(k(\pi)(m-k(\pi))-A(\pi)\right) - m.
\end{array}
$$
\end{proposition}

\begin{proof}
For a membrane $M=(G,f)$, 
the number
$\mathrm{LatticePoints}(M)$ equals the number of faces 
of the plabic graph $G^*$.  
The needed expression for the 
number of faces of a reduced plabic graph was given in
\cite[Proposition~17.10]{TP}. 

$\Area(M)$ equals the number of vertices of the plabic graph $G^*$.
Using Euler's formula 
together with the fact that $G^*$ is a 3-valent graph, we 
deduce
$$
\Area(M) = 2 (\mathrm{LatticePoints}\, M  -1) - m, 
$$
which gives the stated expression for $\Area(M)$.
\end{proof}

In the following sections, we will discuss several special classes
of unimodal loops and membranes that have special properties:
$$
\left\{\hskip-.1in
\begin{array}{c}\textrm{positroid}\\ \textrm{loops}\end{array}
\hskip-.1in \right\} 
\subset
\left\{\hskip-.1in
\begin{array}{c}\textrm{polypositroid}\\ \textrm{loops}\end{array}
\hskip-.1in \right\} 
\subset
\left\{\hskip-.1in
\begin{array}{c}\textrm{$j$-increasing}\\ \textrm{loops}\end{array}
\hskip-.1in \right\} 
\subset
\left\{\hskip-.1in
\begin{array}{c}\textrm{unimodal}\\ \textrm{loops}\end{array}
\hskip-.1in \right\} 
$$

They are described in terms of the associated sequences of roots
$e_{i_1}-e_{j_1},\dots,e_{i_m}-e_{j_m}\in\Z^n$, by the following 
conditions:

\begin{itemize}
\item
Positroid loops: $m=n$; $i_1,\dots,i_n$ is a permutation of $1,\dots,n$;
and $j_a=a$, for $a=1,\dots,n$.
\item
Polypositroid  loops: $j_1\leq \cdots \leq j_m$;
and if $j_a=j_{a+1}$, then $i_{a+1} \in \{i_a-1, i_a - 2 ,\dots, j_a + 1\}$
(a cyclic interval in $[n]$).
\item
$j$-increasing loops: $j_1\leq \cdots \leq j_m$.
\end{itemize}

For example, we'll see that positroid loops (considered up to
affine translations) are in bijection with positroids,
and polypositroid loops are in bijection with 
integer polypositroids.

\section{Positroid membranes}
\label{sec:positr_membr}
We now discuss the distinguished class of loops and membranes related to
positroids.  Assume (in the notations of Section~\ref{sec:Mem_A}), that $L$ is a
loop such that $m=n$ and $\{i_1,\dots,i_n\}=\{j_1,\dots,j_n\}=[n]$ are usual
sets.  Moreover, by permuting the coordinates in $\R^n$, we assume that 
$j_a = a$, for $a=1,\dots,n$.  Such loops $L$  (up to affine translations)
correspond to Grassmann necklaces associated with positroids.

Let $\M\subset {[n]\choose k}$ be a positroid and $\I=(I_1,\dots,I_n)$
be the associated Grassmann necklace;
see Section~\ref{sec:positroids}.  To simplify the presentation, we shall assume that $\M$ has no loops or co-loops and thus $I_{a+1}\ne I_a$ for any $a$.  We have 
$I_{a+1} = (I_a \setminus \{a\})\cup \{i_a\}$, for $a\in\Z/n\Z$,
where $\pi=\pi_\M:a\to i_a$ is a certain permutation 
(derangement) of size $n$.

Consider the loop $L=L_\M:=(e_{I_1},\dots,e_{I_n})$.  
(Recall that $e_I:=\sum_{i\in I} e_i$.) 
It corresponds 
to the sequence of roots $e_{i_1} - e_1, \dots, e_{i_n}-e_n$.

\begin{theorem}
\label{th:min_memb_positroids}
Minimal membranes $M$ with boundary loop $L_\M$ are in bijection 
with reduced plabic graphs with strand permutations $\pi_\M$.
The bijection is given by $M=(G,f)\mapsto G^*$.

Moreover, for any such minimal membrane $M$, its  embedding $\<M\>\subset \R^n$
is contained in the positroid polytope $P_\M:=\conv(e_I\mid I\in \M)$.
\end{theorem}

\begin{proof}
Clearly, the loop $L_\M$ is unimodal.  The first part of the above 
theorem follows from Theorem~\ref{th:mem=plabic_separated}.


For such a minimal positroid membrane $M=(G,f)$,
the vectors $f(v)\in \Z^n$ are related to the 
{\it face labels\/} $I(F)\in {[n]\choose k}$ 
of the corresponding plabic graph $G^*$, which were 
studied in \cite{OPS}, as follows:  
$f(v) = e_{I(F_{v})}$, for any face $F_v$ of $G^*$.
Indeed, in the case of positroid membranes,
the relationship between $f$ and the 
strand labelling $g$ (given in Theorem~\ref{prop:Zn_labelled_plabic})
specializes to the definition of face labels of reduced plabic graphs. 

The second claim of the above theorem now follows from the result
proved in \cite{OPS}
that any face
label $I(F)$ of a reduced plabic graph $G^*$ associated with a positroid $\M$
belongs to the positroid: $I(F)\in \M$.
\end{proof}

\section{Polypositroid membranes}
\label{sec:polypositroid_membr}

We now discuss the class of loops and membranes
related to integer polypositoroids,
which includes positroid loops and membranes from Section~\ref{sec:positr_membr}.
Recall that in Section~\ref{sec:para_poly}, we gave bijections between
polypositroids $P\subset \R^n$, Coxeter necklaces
$\v=(v^{(1)},\dots,v^{(n)})$, and balanced digraphs.
If $P$ is an integer polypositroid, then the $v^{(i)}$ are integer vectors,
and the edge weights $m_{ij}$ of the balanced diagraph are nonnegative integers.

Define the {\it perimeter\/}\footnote{$\Perim\, M$ equals the 
usual Euclidien length of the perimeter of $\<M\>$ divided by $\sqrt{2}$.} 
of a membrane $M$ with boundary loop 
$L = (\lambda^{(1)},\dots,\lambda^{(m)})$
by $\Perim\, M = m$.

\begin{definition} 
\label{def:min_P_membr}
Let $P\subset\R^n$ be an integer polypositroid, and let 
$\v=(v^{(1)},\dots,v^{(n)})$ be its Coxeter necklace. 
We say that a membrane $M$ is a {\it minimal $P$-membrane\/} if 
\begin{enumerate}
\item The boundary loop $L = (\lambda^{(1)},\dots,\lambda^{(m)})$ of $M$ 
contains the points $v^{(1)}$, \dots, $v^{(n)}$ in this particular cyclic order.
\item The membrane $M$ has minimal possible perimeter $\Perim\,M$ 
among all membranes satisfying condition (1).
\item The membrane $M$ has minimal possible surface area $\Area\,M$ among 
all membranes satisfying conditions (1) and (2).
\end{enumerate}
\end{definition}

\begin{remark}
In the above definition, it is important to first minimize  
the perimeter of $M$, and only after that minimize the its surface area.
If we skip condition (2), we can always find a membrane $M$
satisfying (1), whose surface area $\Area\,M$ equals zero.
\end{remark}

Define the {\it standard root order\/} as  the total order $<$ on all roots 
$e_i-e_j$, $i,j\in[n]$, $i\ne j$: 
\begin{equation}\label{eq:standardroot}
\begin{array}{l}
(e_n-e_{1}) < (e_{n-1}-e_1)<(e_{n-2}-e_1)< \cdots < (e_2 - e_1) < \\
<(e_{1}-e_{2}) < (e_{n}-e_2)<(e_{n-1}-e_2)<\cdots < (e_3 - e_2) < \\
<(e_{2}-e_{3})< (e_{1}-e_3)  < (e_{n}-e_3)<\cdots < (e_4 - e_3) < \\
\qquad \qquad \qquad \qquad
\cdots \cdots \cdots
\cdots
\\
<(e_{n-1}-e_n)< (e_{n-2}-e_n) < (e_{n-3}-e_n)< \cdots < (e_1-e_n).
\end{array}
\end{equation}
In other words, we have $e_i-e_j < e_{i'}-e_{j'}$ whenever $j<j'$, 
or ($j=j'$ and $i'\in\{i-1,i-2,\dots, j+1\}$), 
where elements of the interval are considered modulo $n$.  
\begin{remark} The total order \eqref{eq:standardroot} differs from the one in Example~\ref{ex:Aroot}: instead it arises from the simple system $\{e_n-e_1, e_{n-1}-e_n, \ldots, e_2-e_3\}$ and Coxeter element $ c= s_0 s_{n-1} s_{n-2} \cdots s_2$ where $s_0:=(1n)$ is the reflection associated to the root $e_n-e_1$.
\end{remark}

\begin{definition}
A {\it polypositroid loop\/} is a loop $L = (\lambda^{(1)},\dots,\lambda^{(m)})$,
with $\lambda^{(a)}\in\Z^n$, such that 
$\lambda^{(a+1)}-\lambda^{(a)}$,
$a=1,\dots,m$, is a 
weakly increasing sequence of roots in the standard root order.
\end{definition}

The following two lemmas easily follow from the definitions.

\begin{lemma}
Any polypositroid loop is unimodal. 
\end{lemma}

Let $P\subset\R^n$ be an integer polypositroid with Coxeter necklace
$\v=(v^{(1)},\dots,v^{(n)})$,
and let $m_{ij}$ be the edge multiplicities of the associated balanced
digraph.
Define the loop $L_P=(\lambda^{(1)},\dots,\lambda^{(m)})$,
such that $\lambda^{(1)} = v^{(1)}$, and the sequence of roots 
$\lambda^{(a+1)}-\lambda^{(a)}$, $a=1,\dots,m$, is 
the weakly increasing sequence of roots (in the standard root order)
such that each root $e_j - e_i$ is repeated $m_{ij}$ times.
Note that $m$ equals the total number of edges of the balanced bigraph.

\begin{lemma}  
{\rm (1)}  The map $P\to L_P$ is a bijection between integer polypositroids
$P\subset \R^n$ and polypositroid loops in $\R^n$.

\smallskip
\noindent
{\rm (2)} 
The Coxeter necklace $\v$ is a subsequence of the associated
polypositroid loop $L_P$.
Explictly, $v^{(i+1)}=\lambda^{(1+d_1+d_2+\cdots + d_i)}$, for $i=0,\dots,n$; 
here $d_i=-(v^{(i+1)}-v^{(i)})_i =\sum_{j \neq i} m_{ij}$,
is the outdegree (or the indegree) of vertex $i$
of the associated balanced digraph.
\end{lemma}

Recall that, for a unimodal loop $L$, we defined 
the permutation $\pi_L:[m]\to[m]$,
see Definition~\ref{def:pi_L}.

\begin{theorem}
Let $P\subset\R^n$ be an integer polypositroid,
and let $L=L_P$ be the corresponding polypositroid loop.
The following two sets coincide:
\begin{itemize}
\item
the set of minimal $P$-membranes,
\item 
the set of minimal membranes with boundary loop $L$.
\end{itemize}
Each of these sets is in bijection (via $M=(G,f)\mapsto G^*$)
with 
\begin{itemize}
\item 
the set of reduced plabic graphs with strand permutation 
$\pi_L$.
\end{itemize}


\noindent
All minimal $P$-membranes are connected 
with each other a sequences of local moves of types
{\rm (I), (II), (III).}
\end{theorem}

\begin{proof}
Let us show that the boundary loop $L$ of any minimal $P$-membrane
is exactly the polypositroid loop $L_P$.
Indeed, by Definition~\ref{def:min_P_membr}, for each $i=1,\dots,n$, 
the loop $L$ contains the points $v^{(i)}$ and $v^{(i+1)}$ conected 
by a shortest possible piecewise-linear curve with line segments given
some roots.  Thus this portion of the loop $L$
between the points $v^{(i)}$ and $v^{(i+1)}$ should contain
exactly $m_{ij}$ copies of the root $e_j-e_i$, for all $j\ne i$,
cf. formulas~(\ref{eq:neckdef}) and (\ref{eq:Gv}) in
Section~\ref{sec:para_poly}.
Note that any way to arrange these roots 
(in each portion of $L$ between $v^{(i)}$ and $v^{(i+1)}$) would produce 
a unimodal loop.  So the surface area of the membrane $M$ is given 
by Proposition~\ref{prop:alignments_formula_for_area} in terms 
of the number of alignments and the number of anti-exceedances of
the permutation $\pi$ associated with $L$ 
(see Definitions~\ref{def:pi_L}
and \ref{def:anti_exceed_align}).
The number of anti-exceedances of $\pi$ equals $k(\pi)=\sum_{i>j} m_{ij}$.
In order to minimize the surface area of $M$, we need to maximize
the number of alignments $A(\pi)$ of $\pi$.
This maximum is achieved if and only if $L=L_P$.
The theorem now follow from Theorem~\ref{th:mem=plabic_separated}.
\end{proof}

For a balanced digraph on the vertex set $[n]$ with edge
multiplicities $m_{ij}$, $i,j\in[n]$ define the number $m$
of edges, the number $k$ of anti-exceedances,
and the number $A$ of alignments, as follows:
$$
\begin{array}{l}
\displaystyle
m :=\sum_{i\neq j} m_{ij},
\qquad 
k := \sum_{i>j} m_{ij},\\[.1in]
\displaystyle
A := \sum_{
\scriptsize \begin{array}{l}
i<i'<j'<j, \textrm{ or}\\
i'<j'<j<i, \textrm{ or}\\
j'<j<i<i', \textrm{ or}\\
j<i<i'<j'
\end{array}
} m_{ij}\, m_{i'j'}
+ \sum_{i,\, j<j'} (m_{ij} m_{ij'} + m_{ji}m_{j'i})
+ \sum_{i\neq j} {m_{ij} \choose 2}.
\end{array}
$$
(Notice that we regard a pair of edges of the digraph with the same
initial points and/or the same end-points as an alignment.)
Proposition~\ref{prop:alignments_formula_for_area} implies the 
following formulae.

\begin{corollary}
The number of lattice points and the surface area of 
any minimal $P$-membrane $M$ are equal to 
$$
\begin{array}{l}
\mathrm{LatticePoints}(M)= k(m-k)-A + 1,\\[.05in]
\Area(M) = 2\left(k(m-k)-A\right) - m,
\end{array}
$$
where $m$ is the number of edges, $k$ is the number 
of anti-exceedances, and $A$ is the number of alignments 
of the balanced digraph associated with $P$.
\end{corollary}

\section{Positroid lifts}
\label{sec:postr_lifts}

For a membrane $M=(G,f)$ and a vertex $v$ of $G$, let $f(v)_i$ denote
the $i$-th coordinate of the vector $f(v)\in\Z^n$.

\begin{lemma}
\label{lem:min_max_boundary}
Let $M=(G,f)$ be a minimal membrane, and let $i\in[n]$. 
The minimal/maximal value 
of the $i$-th coordinate $f(v)_i$ over all vertices $v$ of $G$ is achieved on some boundary 
vertex $b_j$ of $G$.
\end{lemma}

\begin{proof} 
Suppose that this is not true, and the minimal 
value of the $i$-th coordinate is achieved on some internal vertex $v$ of $G$,
and it is strictly less than $f(b_j)_i$ for all boundary vertices $b_j$.
Let $F_v$ be the face of the plabic graph $G^*$ that corresponds to
the vertex $v$ of $G$. The $i$-th coordinate might take the same minimal 
value on some other vertices of $G$ that correspond to other faces of $G^*$
adjacent to $F_v$.  Let $R$ be the maximal connected region 
formed by such faces of the plabic $G^*$.  By our assumption, the region $R$
does not include any boundary regions of $G^*$, thus the region $R$ 
is bounded by a closed curve $C$ formed by some edges of $G^*$.  
Assume that $C$ is oriented clockwise.
For any other face of $G^*$ adjacent to $R$, 
the $i$-th coordinate is strictly greater.
This mean that, for any edge $a\to b$ of $G^*$ on the curve $C$
(oriented in same the clockwise direction),  the strand $S$ that passes 
through the edge $a\to b$ has label $g(S)=i$; see
Theorem~\ref{prop:Zn_labelled_plabic}.
Since any two intersecting strands cannot have the same label, we conclude
that all edges on the closed curve $C$ belong to the same strand $S$.
This means that the strand $S$ is either self-intersecting or closed.
Since the membrane $M$ is minimal, the plabic graph $G^*$ is reduced,
see Theorem~\ref{th:minimal_membr=reduced}. 
However, by Definition~\ref{def:reduced_plabic} a reduced plabic graph cannot contain self-intersecting 
or closed strands.
We obtain a contradiction. 
The proof of the claim about the maximal value of $f(v)_i$ is analogous.
\end{proof}

Let $\d=(d_1,\dots,d_n)$ be a nonnegative integer vector. 
We say that a loop $L=(\lambda^{(1)},\dots,\lambda^{(m)})\in (\Z^n)^m$ 
is {\it $\d$-boxed\/} if $\min_{j\in [m]} \lambda^{(j)}_i = 0$
and $\max_{j\in [m]} \lambda^{(j)}_i = d_i$, for all $i\in[n]$.
In other words, the curve $\<L\>$ lies in the box
$[0,d_1]\times \cdots \times[0,d_n]\subset \R^n$ and have points on each facet
of the box.  



According to Lemma~\ref{lem:min_max_boundary}, for any minimal membrane
$M$ with a $\d$-boxed boundary loop $L$, we have 
$\<M\> \subset [0,d_1]\times \cdots \times[0,d_n]$.
Moreover, both $\<L\>$ and $\<M\>$ belong to the intersection
the box $[0,d_1]\times \cdots \times[0,d_n]$ with some affine hyperplane
$x_1+\cdots + x_n = k$.

\begin{definition}
Let $d=d_1+\cdots+d_n$.
For an integer vector $\lambda=(\lambda_1,\dots,\lambda_n)
\in [0,d_1]\times \cdots \times[0,d_n]$,
let  $\lift(\lambda)$ be the 01-vector in $\Z^d$ given by 
$$
\lift(\lambda) := (0^{d_1-\lambda_1}, 1^{\lambda_1}, 
0^{d_2-\lambda_2}, 1^{\lambda_2}, \dots,
0^{d_n-\lambda_n}, 1^{\lambda_n}) \in \{0,1\}^d,
$$
where $a^r$ denotes $a$ repeated $r$ times.

Let $L=(\lambda^{(1)},\dots,\lambda^{(m)})\in (\Z^n)^m$ 
be a $\d$-boxed loop.
The {\it lift\/} of $L$ is the loop
$\lift(L):=(\lift(\lambda^{(1)}),\dots,\lift(\lambda^{(m)}))
\in (\Z^d)^m$.

For a minimal membrane $M=(G,f)$ with boundary loop $L$,
the {\it lift\/} of $M$ is the membrane $\lift(M) := (G,\lift(f))$,
where $\lift(f): v\mapsto \lift(f(v))\in \Z^d$, for a vertex $v$ of $G$. 
\end{definition}

Let $\proj:\R^d\to\R^n$ be the map\footnote{The map $\proj$ is 
a projection if all $d_i$'s are positive.}
given by 
$$
\proj:(x_1,\dots,x_d)\mapsto 
(x_1+ x_2 + \cdots + x_{d_1}, 
x_{d_1+1} + \cdots + x_{d_1 + d_2},  \cdots,
x_{d_1+\cdots +  d_{n-1} + 1} + \cdots 
+ x_{d}).
$$
For a membrane $\tilde M = (G,\tilde f)$, where $f:\Vert\to\Z^d$,
define $\proj(\tilde M) := (G, \proj(\tilde f))$.


\begin{proposition}
\label{prop:L-to-positroid}
Let $L=(\lambda^{(1)},\dots,\lambda^{(m)})\in (\Z^n)^m$ be
a $\d$-boxed unimodal loop.  Then $m=d:=d_1+\cdots + d_n$.
The sequence of roots
$e_{i_a}-e_{j_a} = \lift(\lambda^{(a+1)})-\lift(\lambda^{(a)})\in\Z^m$,
$a=1,\dots,m$, associated with the loop $\lift(L)$ satisfies 
the condition: both sequences $i_1,\dots,i_m$
and $j_1,\dots,j_m$ are permutations of $1,\dots,m$.

The following three sets are in bijection with each other:
\begin{itemize}
\item
Minimal membranes $M$ with boundary loop $L$.
\item
Minimal membranes $\tilde M$ with boundary loop $\lift(L)$.
\item
Reduced plabic graphs $G^*$ with strand permutation $\pi_L$,
see Definition~\ref{def:pi_L}.
\end{itemize}
Explicitly, the bijections are given by the maps (which form
a commutative diagram):
$M\mapsto \tilde M = \lift(M)$, $\tilde M\mapsto M = \proj(\tilde M)$,
$M=(G,f) \mapsto G^*$, $\tilde M=(G,\tilde f) \mapsto G^*$. 
\end{proposition}

The loop $\lift(L)$ is obtained from a positroid loop in $\Z^m$ by a 
permutation of coordinates in $\Z^m$.   This is exactly the positroid
whose permutation is equal to $\pi_L$.
Thus the lifted membranes $\tilde M$ are permuted positroid membranes.

\begin{proof}
The first claim easily follows from the definitions.
The claim about bijections between the sets
of membranes $M$, $\tilde M$, and reduced plabic graphs $G^*$ 
follows from Theorem~\ref{th:mem=plabic_separated}.
Indeed, both loops $L$ and $\lift(L)$ correspond to the 
same permutation $\pi_L = \pi_{\lift(L)}$,
see Definition~\ref{def:pi_L}.
\end{proof}

Let us specialize this construction to polypositroid loops.
As in the previous section,
let $P\subset \R^n$ be an integer polypositroid. 
Assume that $P$ belongs to the positive orthant $\R_{\geq 0}^n$
and has points on each coordinate plane in $\R^n$.
Let $m_{ij}$ be the edge multiplicities 
of the balanced digraph associated with $P$, 
let $m :=\sum_{i,j} m_{ij}$ be the total number of edges
of the digraph,
let $k=\sum_{i>j} m_{ij}$ be its number of anti-exceedances,
and let $d_i := \sum_{j} m_{ij}$, $i\in[n]$, 
be the outdegrees (or the indegrees) of the balanced digraph.
%
%
Clearly, we have $m=d_1+\cdots + d_n$.
Let $\d=(d_1,\dots,d_n)$.



\begin{corollary}
For a polypositroid $P$ as above,
the polypositroid loop $L=L_P$ is $\d$-boxed. 
Its lift $\lift(L)$ is a positroid loop in $\Z^m$.

The maps $\lift:M\mapsto\tilde M$ and $\proj:\tilde M\mapsto M$  
give a bijection between minimal membranes $M$ with boundary loop $L$ 
minimal positroid membranes $\tilde M$ with boundary loop $\lift(L)$.
\end{corollary}

The first claims in this statement are straightforward from the definitions,
the last claim is a special case of Proposition~\ref{prop:L-to-positroid}.

Let $\M\subset {[m]\choose k}$ be the positroid associated 
with the positroid loop $\lift(L)$, and let 
$P_\M:=\conv(e_I\mid I \in \M\}\subset \R^m$ be the positroid polytope
of the positroid $\M$.

\begin{lemma}
Under the map $\proj:\R^m\to\R^n$,  
we have $\proj(P_\M)\subseteq P$.
\end{lemma}

\begin{proof} 
The positroid polytope $P_\M$, being an alcoved polytope in $\R^m$,
is given by inequalities of the form $x_i+x_{i+1} + \cdots + x_j \leq c_{ij}$
for all cyclic intervals $[i,j]$ in $[m]$.
Similarly, the polytope $P$, being an alcoved polytope in $\R^n$,
is given by inequalities of the form $y_a+y_{a+1}+\cdots+ y_b \leq d_{ab}$
for all cyclic intervals $[a,b]$ in $[n]$.
One can check from the definitions that the inequalities 
for $P_\M$ corresponding to cyclic intervals $[i,j]$ that consist of unions 
of blocks $\{1,\dots,d_1\}$, $\{d_1+1,\dots,d_1+d_2\}$, etc., project 
exactly to the inequalities defining the polytope $P$. 
\end{proof}

Theorem~\ref{th:min_memb_positroids} now implies the following claim.

\begin{corollary}
\label{cor:MinP}
For any minimal $P$-membrane $M$, 
the embedding $\<M\>\subset \R^n$ belongs to the polypositroid:
$\<M\>\subset P$.
\end{corollary}

\begin{remark}
The same plabic graph $G^*$ can appear in different membranes
$M=(G,f)$
of different dimensions.   For any reduced plabic graph $G^*$ with $m$ boundary
edges, there is always the associated positroid membrane of dimension $m-1$
that lies in a hyperplane $\{x_1+\cdots + x_m = k\}\subset \R^m$.
But there might also be other lower dimensional membranes with the 
same plabic graph, which are obtained by projections of this positroid
membrane.
\end{remark}

Define the {\it essential dimension\/} of a reduced 
plabic graph $G^*$ as the minimal dimension a minimal membrane
whose plabic graph is equal to $G^*$.

\begin{proposition} 
For a reduced plabic graph $G^*$ with $m$ boundary 
edges, the essential dimension equals $m-1$ if and only if
$G^*$ is a plabic graph for the top positroid cell in $\Gr(k,m)_{\geq 0}$
for some $k\in[m-1]$, i.e., if its strand permutation 
is $\pi:i\mapsto i+k\pmod m$.
\end{proposition}

\begin{proof}
If a graph $G^*$ has a maximal possible essential dimenion $m-1$ 
then its strand permutation $\pi$ does not have alignments.
Indeed, if $\pi$ has an alignment $(i,j)$ then we can project
the positroid membrane for $G^*$ to a lower-dimensional 
membrane by mapping $(x_i,x_j)\mapsto x_i + x_j$
and leaving all other coordinates.
The only permutations with no alignments are are permutations
given by $\pi:i\mapsto i+k\pmod m$, for some $k$.

Assume now that $G^*$ is a plabic graph for the top 
positroid cell in $\Gr(k,m)_{\geq 0}$ and $M=(G,f)$ is a minimal membrane
with $f:\Vert \to \Z^n$.  If $n<m$, then we can find two different 
strands $S$ and $T$ in $G^*$, with the same label $g(S)=g(T)$;
see Theorem~\ref{prop:Zn_labelled_plabic}.
According to Theorem~\ref{prop:Zn_labelled_plabic},
the strands $S$ and $T$ can not intersect in the plabic graph $G^*$.
It not hard to show, using the techniques of \cite{TP}, that for any $i$
and $j$, there is {\it some\/} plabic graph for the top cell
in $\Gr(k,m)_{\geq 0}$ whose $i$-th and $j$-th strands intersect.
Also, according to Theorem~\ref{th:plabic_move_equiv} 
(\cite[Theorem~13.4]{TP}) all plabic graphs for the top cell are 
connected with each other by local moves.
This means that even if the strands $S$ and $T$ do not intersect
in $G^*$, one can always find a sequence of local moves that result in 
a plabic graph where the pair strands with the same sources and targets
as $S$ and $T$ intersect each other.
Since local moves preserve minimal membranes, we deduce that for 
two different strands, we cannot have $g(S)=g(T)$.  Thus $n=m$.
\end{proof}


Reduced plabic graphs of essential dimension 2 are 
the bipartite plabic graphs that can be drawn on the plane 
as subgraphs of the regular hexagonal lattice.

%
%
%
%
%
%
%
%
%
%

\section{Semisimple membranes}
\label{sec:semisimple_membranes}


Recall (e.g., see \cite{Hum}), that for a Coxeter element $c\in W$ in the Weyl
group associated with root system $R\subset V\simeq \R^r$, there exists a
unique 2-dimensional plane $P\subset V$, $P\simeq\R^2$, called the {\it Coxeter
plane,} such that $P$ is $c$-invariant, and the Coxeter element $c$ acts on $P$
by rotations by $2\pi/h$.  Note that the Coxeter element $c$ defines an
orientation on the Coxeter plane $P$, assuming that $c$ acts on $P$ by a
clockwise rotation.  Let $p:V\to P$ be the orthogonal projection onto the
Coxeter plane.

\begin{definition}
\label{def:simple_semi_simple}
Fix a Coxeter element $c\in W$. 
We say that an $R$-membrane $M$ is {\it semisimple\/} if 
the projection $p:\left<M\right> 
\to p(\left<M\right>)$ onto 
the Coxeter plane is a bijective map  between 
$\left<M\right>$ and $p(\left<M\right>)$
and each component of the projection $p(\<L\>)$ of
the boundary loop $L$ of $M$ is oriented clockwise
in the Coxeter plane.


Equivalently, an $R$-membrane $M=(G,f)$ is semisimple if the 
orientation of any triangle $\Delta$ in the cactus $G$ agrees
with the orientation of the projection $p(\<\Delta\>)$
in the Coxeter plane.
\end{definition}

\begin{remark}
A {\it simple\/} membrane $M$ is a semisimple membrane such that
$\<M\>\simeq p(\<M\>)$ is homeomorphic to a disk
(or to a line segment when $m=2$).
A semisimple membrane is simple if and only if the graph $G^*$
is connected.
Any semisimple membrane is obtained by taking wedges of
simple membranes along their boundary vertices.
\end{remark}


Let us now discuss the type $A$ case.
Assume that $c=(12\cdots n)\in S_n$ is the standard  long cycle in $S_n$, which is a Coxeter element in type $A$ case.
We can identify the corresponding Coxeter plane with $\R^2$
and assume that 
$$
p:\R^n\to\R^2,
\quad\quad p:e_i \mapsto u_i, \textrm{ for } i=1,\dots,n,
$$
is the projection that sends
the coordinate vectors $e_1,\dots,e_n$ in $\R^n$ to the vertices
$u_1,\dots,u_n$ of a regular $n$-gon in $\R^2$ centered at the origin $0$
arranged in the clockwise order.

Recall that a loop $L =(\lambda^{(1)},\dots,\lambda^{(m)})$, with 
$\lambda^{(a+1)}-\lambda^{(a)}=e_{i_a}-e_{j_a}$, 
for $a\in\Z/m\Z$, is called {\it $j$-increasing\/}
if $j_1\leq j_2 \leq \cdots \leq j_m$.
In particular, any polypositroid loop is $j$-increasing.

\begin{theorem}  
\label{th:minimal=semisimple}
Let $L$ be any $j$-increasing polypositroid loop.
A membrane $M$ with boundary loop $L$ is minimal if and only if
$M$ is semisimple.
\end{theorem}

\begin{proof}
Let $M=(G,f)$ be a minimal membrane with boundary loop $L$.
So the plabic graph $G^*$ is 
reduced.  Let $\Delta$ be any triangle in the cactus $G$, let $d$ be the 
corresponding vertex in the plabic graph $G^*$, and let 
$S_1, S_2, S_3$ be the three strands in $G^*$ that pass through the
three edges of $G^*$ adjacent to the vertex $d$ 
in the directions away from $d$ arranged, resp., in the clockwise order.
Since $G^*$ is a reduced plabic graph, the segments
of these three strands between the vertex $d$ and their
target points $t_1,t_2,t_3$ 
on the boundary of the disk cannot intersect each other.
So the three target points $t_1,t_2,t_3$ of the strands 
$S_1,S_2,S_3$, resp., are arranged
in the clockwise order on the boundary of the disk.
Thus the labels $g(S_1), g(S_2), g(S_3)\in [n]$
of these three strands (which are 
$j_{t_1}, j_{t_2}, j_{t_3}$ resp.; 
see Theorem~\ref{prop:Zn_labelled_plabic})
are ordered as $g(S_1)<g(S_2)< g(S_3)$ (up to a cyclic shift). 

The triangle $\Delta$ is embedded into $\R^n$ 
either as $\conv(-e_{g(S_1)}, -e_{g(S_2)}, -e_{g(S_3)})$
or as $\conv(e_{g(S_1)}, e_{g(S_2)}, e_{g(S_3)})$
(up to a parallel translation)
depending on the color of the triangle.
In both cases, the rules of the road imply that the orientation 
of the triangle $\Delta$ in the cactus $G$ agrees with the orientation 
of the projection $p(\<\Delta\>)$ onto the Coxeter plane.
This implies that the membrane $M$ is semisimple.

On the other hand, let us now assume that $M$ is a semisimple membrane
and deduce that it should be a minimal membrane.
Whenever we apply a local move to $M$, it remains semisimple.
Indeed, for a square move (II), i.e., an octahedron move
shown on Figure~\ref{fig:tetra_octa}, if the upper half of the 
surface of the octahedron projects bijectively onto the Coxeter plane,
then the lower part of the surface projects bijectively onto the 
Coxeter plane.  
For a tetrahedron move, i.e., a move of type (I) or (III)
shown on Figure~\ref{fig:tetra_octa}, observe that the four
vertices of the two triangles involved in the move
project onto the Coxeter plane as some points 
$u_i +v , u_j + v , u_k + v , u_l +v$.  Since $u_1,\ldots,u_n$
are vertices of a convex $n$-gon, it is impossible that 
one of these four points lies in the convex hull of the three
other points.  Thus if the union of two triangles involved in 
a move of type (I) or (III) projects bijectively onto the Coxeter plane,
then the same remains true after the move.

If $M$ is not minimal, then we can find a sequence of local moves
that results in a plabic graph with a pair of parallel edges
and a membrane with two coinciding triangles
$\<\Delta\>=\<\Delta'\>$.
However, in a semisimple membrane two triangles cannot coincide.
Thus $M$ should be a minimal membrane.
\end{proof}

\begin{remark}
For a positroid $\M$,
projections of minimal membranes with boundary loop $L_\M$
onto the Coxeter plane
are related to the {\it plabic tilings\/} of Oh, Postnikov, and Speyer~\cite{OPS}.
Plabic tilings are certain subdivisions (or tilings) of a polygon on 
the plane into smaller polygons (or tiles) colored in two colors. 
The tiles are not necessarily triangles.
 A plabic tiling corresponds to an equivalence class 
under moves of types (I) and (III)
of projections of minimal membranes with boundary loop $L_\M$ 
onto the Coxeter plane. 
In other words, a plabic tiling
is obtained from a projection of a membrane by combining its adjacent
triangles colored in the same color into tiles.
\end{remark}


\section{Higher octahedron recurrence and cluster algebras}
\label{sec:higher_octahedron}


Consider the collection of variables $x_{\lambda}$,
labelled by integer vectors $\lambda\in\Z^n$, that satisfy the following 
{\it higher octahedron relations}:  
\begin{equation}
\label{eq:higher_octahedron}
x_{e_i+e_k+\lambda}\cdot x_{e_j+e_l+\lambda}=
x_{e_i+e_j+\lambda}\cdot x_{e_k+e_l+\lambda}+
x_{e_i+e_l+\lambda}\cdot x_{e_j+e_k+\lambda}\,,
\end{equation}
for any $i<j<k<l$ in $[n]$, and any $\lambda\in\Z^n$.  We will call the recurrence~(\ref{eq:higher_octahedron})
the {\it higher octahedron recurrence}. The polytope $\conv(e_i+e_k, e_j+e_l, e_i+e_j,
e_k+e_l, e_i+e_l, e_j+e_k)$ is an octahedron in $\R^n$,
which explains the name of the above relations.

Clearly, each relation~\eqref{eq:higher_octahedron}
involves only the variables $x_\lambda$,
for $\lambda$ in an affine hyperplane  $\{\lambda_1+\cdots+\lambda_n=
\mathrm{Const}\}$.  So, essentially, \eqref{eq:higher_octahedron} is a recurrence relation 
on variables corresponding to points of the $(n-1)$-dimensional integer lattice.
For $n=4$, this recurrence on $\Z^3$ is equivalent to 
the {\it octahedron recurrence}; see for example \cite{Spe}.

Let us now define algebras generated by certain finite subsets
of variables $x_\lambda$ satisfying the octahedron relations.

\begin{definition}
For a loop $L\subset \Z^n$, define its {\it cloud\/} 
as the union of integer lattice points of $\<M\>$ over all minimal
membranes $M$ with boundary loop $L$:
$$
\cloud(L) : = \bigcup_{M = (G,f) \textrm{ min. membr. with bound. }L }
\{f(v)\mid v \textrm{ is vertex of } G\}.
$$
\end{definition}

\begin{remark}
According to Corollary~\ref{cor:MinP}, for a polypositroid loop $L=L_P$,
$\cloud(L_P)$ belongs to the set $P\cap \Z^n$ of lattice points of
the polypositroid $P$.
For some (poly)positroids, $\cloud(L_P) = P \cap \Z^n$.
For example, the equality holds if $P = P_\M$ where $\M = \binom{[n]}{k}$ is the uniform matroid.
In this case, $P$ is the hypersimplex $\Delta(k,n)$
and $\cloud(L_P)=\{e_I \mid I\in {[n]\choose k}\}$
is the set of all lattice points of $\Delta(k,n)$.

However, in general $\cloud(L_P)$ is not equal to $P\cap\Z^n$.
For example, if the loop $L$ is a wedge of line segments, 
then any minimal membrane with boundary $L$ has no triangles.
In this case, $\cloud(L)=L$.  Apart from some trivial cases,
this set cannot be equal to the set of lattice points of the 
polypositroid $P$, which is a convex polytope.

For the case $n=3$, a generic polypositroid is a hexagon,
and the associated cloud is a triangle with line segments attached
to its vertices.
\end{remark}

\begin{definition}  For a loop $L$, let 
$\Octa_L:=\C[x_\lambda,x_\mu^{-1}]_{\lambda\in\cloud(L),\mu\in L}$
be the commutative algebra over $\C$ generated by the variables 
$x_\lambda$, for $\lambda\in\cloud(L)$, and $x_\mu^{-1}$ for $\mu \in L$, 
modulo the octahedron 
relations~\eqref{eq:higher_octahedron}.  We call $\Octa_L$ the 
{\it octahedron algebra\/} of the loop $L$.
\end{definition}

Recall that, for any finite quiver (i.e., a directed graph) $Q$ with 
a chosen subset of vertices $B$, there is cluster algebra, whose
initial cluster variables correspond to vertices of $Q$ and 
frozen cluster variables correspond to the subset of vertices $B$,
see \cite{FZ02}.  By convention, we will assume that the inverses of 
frozen cluster variables belong to the cluster algebra.

\begin{definition}
For a membrane $M=(G,f)$, define the {\it quiver\/} of $M$ as the directed
graph $Q(M)$ on the same set of vertices $\Vert$ as the cactus $G$,
whose edges $u\to v$ are the edges of $G$ that separate triangles
of different colors, 
directed so that the adjacent black triangle is on the right
of the edge $u\to v$ and the adjacent white triangle is on the left
of the edge $u\to v$.

Let $\A_M$ denote the cluster algebra over $\C$
given by the quiver $Q(M)$ of the membrane $M$ with 
frozen variables corresponding to the boundary vertices $b_i\in\Vert$
of $G$.
\end{definition}

\begin{theorem}
Let $L$ be any $j$-increasing loop.
(In particular, $L$ can be any polypositroid loop.)
Let $M=(G,f)$ be any minimal membrane with boundary loop $L$, and let $\Vert$
be the vertex set $G$.  

For any other minimal 
membrane $M'$ with the same
boundary loop $L$, the quivers $Q(M)$ and $Q(M')$ are mutation equivalent, and we have a canonical isomorphism $\A_{M}\simeq \A_{M'}$.
The octahedron algebra $\Octa_L$ is
a (finitely generated) subalgebra of the cluster algebra $\A_M$.

More explicitly, let us identify the collection of variables
$\{x_{f(v)}\mid v \in \Vert\}$ with the initial cluster of 
$\A_M$.  We have

\begin{enumerate}
\item
The collection of variables $\{x_{f(v)}\mid v\in \Vert\}$
is an algebraically independent set in $\Octa_L$.
Any other $x_\lambda$, for $\lambda\in\cloud(L)$ is expressed in terms
of these variables by Laurent polynomials with positive integer
coefficients.


\item
Local moves of membranes  of types {\rm (I)} and {\rm (III)} 
(tetrahedron moves)
do not change the set of variables $\{x_{f(v)}\mid v\in \Vert\}$ and
they do not change the quiver $Q(M)$ of $M$ and the cluster algebra 
$\A_M$.

\item Local moves of membranes of type {\rm (II)} (octahedron moves)
change exactly one element in the set $\{x_{f(v)}\mid v\in \Vert\}$.
They correspond to (a certain class of) mutations 
of the cluster algebra $\A_M$.

\item
Any minimal membrane $M'=(G',f')$ with the same boundary loop $L$
is obtained from $M$ by a sequence of local moves.
The collection of variables $\{x_{f'(v')}\mid v'\in\Vert'\}$
(where $\Vert'$ is the set of vertices of $G'$) is a cluster 
of the cluster algebra $\A_M$.

\item 
The isomorphism $\mu_{M,M'}:\A_M\to \A_M'$ given 
by the composition of mutations coming from a sequence of local
moves connecting the membranes $M$ and $M'$ depends only on the membranes
$M$ and $M'$, and it does not depend on a choice of a sequence of local moves
connecting the membranes.

\item The octahedron algebra $\Octa_L$ is the subalgebra of the
cluster algebra $\A_M$ generated by all cluster variables
from all clusters of $\A_M$ (and inverses of frozen variables) that correspond to minimal membranes
$M'$ with the same boundary loop $L$.
\end{enumerate}
\end{theorem}

\begin{proof}  
Part (2) is clear, because
tetrahedron moves of membranes
do not change the set of points $\{f(v)\mid v\in \Vert\}$ 
and they do not change the quiver $Q$ of a membrane.

Observe that, if we apply a move $M\to\tilde M$
 of type (II) (an octahedron move), i.e., apply a square move 
of the associated plabic graph $G^*$,
then the labels $i,j,k,l$ of the 4 strands (arranged clockwise) 
going out of the vertices 
of the square are ordered as $i<j<k<l$ (up to a cyclic shift),
see the argument in the proof of Thereom~\ref{th:minimal=semisimple}.
One easily checks that the quiver $Q(\tilde M)$ is a mutation 
of the quiver $Q(M)$.  This move results in replacing one element
$x_{e_i + e_k + \lambda}$ of the set $\{x_{f(v)}\mid v\in\Vert\}$
by $x_{e_j+e_l+\lambda}$.  One checks 
from the definitions that the transformation:
$$
x_{e_i + e_k + \lambda}\to 
x_{e_j + e_l + \lambda}=
(x_{e_i+e_j+\lambda}\cdot x_{e_k+e_l+\lambda}+
x_{e_i+e_l+\lambda}\cdot x_{e_j+e_k+\lambda})/x_{e_i+e_k+\lambda},
$$
is exactly the cluster mutation of the associated variable 
in the initial cluster of $\A_M$.
So we get part (3).

Part (4) follows from 
Theorem~\ref{th:mem=plabic_separated}
and the fact the that loop $L$ is unimodal.

Part (1) now follows
from general results on cluster algebras,
namely, Fomin-Zelevinsky's Laurent phenomenon \cite{FZ}
and the positivity result of Lee-Schiffler \cite{LS}.

Part (5) follows from the observation that each element 
of the initial seed of the cluster algebra $\A_{M'}$ corresponds to
some variable $x_{f'(v')}$, $v'\in \Vert'$.
By part (1), this element is expressed by a Laurent polynomial
in terms of the variables $x_{f(v)}$, $v\in\Vert$, corresponding to
the initial seed of $\A_M$.  This Laurent expression depends 
only on the membrane $M$ 
and the integer vector $f'(v')\in\Z^n$, and it does not 
depend on a choice of a sequence of local moves connecting the membranes
$M$ and $M'$.

Part (6) is clear from the above discussion.
\end{proof}

For a $j$-increasing loop $L$, we denote by $\A_L$ the cluster algebra $\A_M$ for a minimal membrane $M$ with boundary loop $L$.  The quiver $Q(M)$, and thus the cluster algebra $\A_M$,
of a minimal membrane $M=(G,f)$ depends only on the reduced plabic graph 
$G^*$.  Since every reduced plabic graph appears in a 
minimal {\it positroid\/} membrane,
the class of cluster algebras $\A_L$ is as general as its
subclass corresponding to positroid loops. 
However, the description of these cluster algebras in terms of the higher 
octahedron recurrence allows us 
to associate some cluster variables with points of the integer lattice 
$\Z^n$ and some clusters with membranes, which provides 
an additional geometrical intuition into the structure
of these cluster algebras.   

Let $\M$ be a positroid and let $\A_{L_\M}$ be the cluster algebra for the positroid loop $L_\M$ (see Section~\ref{sec:positr_membr}).  The cluster algebra $\A_{L_\M}$ implicitly appeared in \cite{TP}.
It was shown in \cite{GL} that the cluster algebra $\A_{L_\M}$ is
isomorphic to the coordinate ring $\C[\mathring \Pi_\M]$ of an {\it open positroid variety} $\mathring\Pi_\M$ \cite{KLS}, confirming conjectures of Muller--Speyer \cite{MS} and Leclerc \cite{Lec}.

\begin{remark}
Suppose that $L_1$ and $L_2$ are two $j$-increasing loops, and $M_1=(G_1,f_1)$ and $M_2=(G_2,f_2)$ are minimal membranes with boundary loops $L_1$ and $L_2$ respectively.  If the reduced plabic graphs $G_1^*$ and $G_2^*$ are connected by the local moves (I), (II), (III), then it follows from the above remarks that $\cloud(L_1)$ and $\cloud(L_2)$ are naturally in bijection.
\end{remark}

\begin{remark}
The cluster algebra $\A_L$ is typically a cluster algebra of infinite type, with 
infinitely many cluster variables.  On the other hand, $\cloud(L)$ is a finite set, 
corresponding to a finite subset of the cluster variables of $\A_L$, and thus the octahedron algebra $\Octa_L$ is a finitely-generated subalgebra of the cluster algebra $\A_L$.  However, even when $\A_L$ is of infinite type, we may have $\Octa_L = \A_L$.  For example, this holds when the reduced plabic graph $G^*$ corresponds to the top positroid cell of $\Gr(k,n)$.  In this case, $\A_L$ is isomorphic to the homogeneous coordinate ring of the Grassmannian with the cyclic minors $\Delta_{12\cdots k}, \Delta_{23 \cdots (k+1)},\ldots$ inverted.  The equality $\Octa_L=\A_L$ follows from the fact that the homogeneous coordinate ring of the Grassmannian is generated by Pl\"ucker coordinates $\Delta_I$.
\end{remark}



\section{Asymptotic cluster algebra}
\label{sec:asympt_membr}

As we discussed in Section~\ref{sec:higher_octahedron}, membranes
are closely related to a class of cluster algebras generated by some 
collections of variables satisfying the higher octahedron recurrence. 
Minimal membranes correspond to certain clusters in these algebras 
and local moves of membranes correspond to cluster mutations.
It would be interesting to investigate the asymptotic behavior 
of these structures under dilations of the boundary loop.

We can call this area of research the ``Asymptotic Cluster Algebra''.
We anticipate that many results from statistical physics and
from asymptotic representation theory (e.g., the study of asymptotics 
properties of representations of symmetric groups), will have their
analogs in the asymptotic cluster algebra.

Under dilations of the boundary loop $L$,
minimal membranes might approach a certain limit surface $S$.

Let $\Memb(L)$ be the set of all minimal membranes with boundary loop
$L$.  For $t\in\Z_{>0}$, let $tL$ denote the loop $L$ dilated $t$ times.

\begin{conjecture}   Let $L$ be a unimodal loop.
For a positive integer $t$, consider the uniform distribution on the set 
$\Memb(tL)$.

There exists a unique surface $S\subset \R^n$ (with boundary $\<L\>$)
such that, for any $\epsilon>0$, there exists $N>0$ such that, for 
any $t\geq N$,  the probability that ${1\over t} \<M\>$,
for $M\in\Memb(tL)$,
belongs to the $\epsilon$-neighborhood of $S$ is greater than $1-\epsilon$.
\end{conjecture}

A related conjecture can be formulated in terms of measures.
For a loop $L$, consider the measure 
$\mu_L$ on the set
$\cloud(L)\subset \Z^n$ given by 
$$
\mu_L (a) = 
{ \#\{M \in \Memb(L) \textrm{ such that }
a\in \<M\>\}\over
 \#\Memb(L) \times  \# (\textrm{lattice points in 
any } M\in \Memb(L))},
$$
for $a\in\Z^n$. 
Clearly, $\sum_{a\in\Z^n} \mu_L(a) = 1$,
so $\mu_L$ is a probability distribution.

Equivalently, $\mu_L(a)$ is the probability that a random
(uniformly chosen) minimal membrane with boundary $L$ contains a lattice
point $a$.  In other words, $\mu_L$ is the density of a random membrane.

For a positive integer $t$, let $\mu_{L,t}(a) := \mu_{tL}(ta)$.
The measure (probability distribution) $\mu_{L,t}$ is supported on 
a certain subset of the lattice $({1\over t} \Z)^n$.

\begin{conjecture}
Let $L$ be a unimodular loop.  As $t\to \infty$, 
the measures $\mu_{L,t}$ converge to a certain limit measure
$$
\mu_{L,\infty}:=\lim_{t\to\infty} \mu_{L,t}
$$
supported on a certain limit surface $S\subset\R^n$.
\end{conjecture}

\begin{remark}
For $n=4$, semisimple membranes are related to the 6-vertex and 8-vertex models, 
whose asymptotic properties have been extensively studied.
\end{remark}

\begin{remark}
It would be interesting to investigate a relationship 
between the ``limit membrane'' $S$ and Plateau's problem; see Remark~\ref{rem:Plateau}.
\end{remark}

\end{document}